\documentclass[11pt]{amsart}

\usepackage{amsmath, amssymb, amsfonts, amsthm, mathrsfs, mathtools}
\usepackage{thmtools}

\usepackage[hcentering, vcentering, total={5.9in, 8.2in}]{geometry} 

\usepackage{graphicx}   
\usepackage{tikz-cd}    
\usetikzlibrary{cd}

\tikzset{
    symbol/.style={
    draw=none,
    every to/.append style={
    edge node={node [sloped, allow upside down, auto=false]{$#1$}}
    },
},
}

\usepackage{enumerate}

\usepackage[colorlinks=true, linkcolor=blue, citecolor=red, backref=page]{hyperref}
\usepackage[capitalize,nameinlink]{cleveref}  

\usepackage{fancyhdr}

\usepackage{xcolor}
\usepackage{comment}

\theoremstyle{plain}
\newtheorem{theorem}{Theorem}[section]
\newtheorem{lemma}[theorem]{Lemma}
\newtheorem{proposition}[theorem]{Proposition}
\newtheorem{corollary}[theorem]{Corollary}

\theoremstyle{definition}
\declaretheorem[style=definition,qed=$\triangle$,sibling=theorem]{definition}
\declaretheorem[style=definition,qed=$\triangle$,sibling=theorem]{example}
\declaretheorem[style=definition,qed=$\triangle$,sibling=theorem]{remark}
\declaretheorem[style=definition,qed=$\triangle$,sibling=theorem]{question}

\def\R{\mathbb{R}}
\def\C{\mathbb{C}}
\def\Z{\mathbb{Z}}

\def\S{\mathbb{S}}
\def\P{\mathbb{P}}
\def\T{\mathbb{T}}
\def\H{\mathbb{H}}
\def\bD{\mathbb{D}}
\def\CP{\mathbb{CP}}
\def\RP{\mathbb{RP}}

\def\D{\mathcal{D}}
\def\cC{\mathcal{C}}
\def\cL{\mathcal{L}}
\def\cF{\mathcal{F}}
\def\cS{\mathcal{S}}
\def\O{\mathcal{O}}

\def\a{\alpha}

\def\d{\mathrm{d}}
\def\p{\mathrm{p}}
\newcommand{\Gr}{\mathrm{Gr}}

\DeclareMathOperator{\rk}{rk}

\DeclareMathOperator{\im}{Im}

\newcommand{\defi}{\coloneqq}

\newcommand{\st}{\, | \,}

\newcommand{\Leg}{\mathfrak{Leg}}
\DeclareMathOperator{\Hom}{Hom}
\DeclareMathOperator{\Hopf}{Hopf}

\def\nbh{neighbourhood }
\def\Op{\mathcal{O}p}
\def\std{\mathrm{std}}

\def\can{\mathrm{can}}
\def\PF{\mathrm{PF}}
\def\F{\mathrm{F}}
\def\spin{\mathrm{spin}}
\def\pspin{\mathrm{prespin}}

\title{The h-principle fails for prelegendrians in corank $2$ fat distributions}

\author{Eduardo Fern\'andez}
\address{School of Mathematics\\ Georgia Institute of Technology \\ Atlanta \\ GA, USA}
\email{eduardofernandez@gatech.edu}

\author{\'Alvaro del Pino}
\address{Utrecht University, Department of Mathematics, Budapestlaan 6, 3584 CD Utrecht, The Netherlands}
\email{a.delpinogomez@uu.nl}

\author{Wei Zhou}
\address{Instituto de Ciencias Matemáticas, Consejo Superior de Investigaciones Científicas, Madrid, Spain}
\address{Departamento de Álgebra, Geometría y Topología, Facultad de Ciencias Matemáticas, Universidad Complutense de Madrid, Madrid, Spain}
\email{wzhou02@ucm.es}

\date{\today}

\begin{document}

\begin{abstract}
We investigate the $h$-principle problem for fat distributions. These are maximally non-integrable distributions with natural symplectisations and contactisations, that generalize contact distributions to higher corank. We focus on the corank-$2$ case, where we study a natural class of submanifolds, which we call prelegendrians. Their key feature is that they admit a canonical Legendrian lift to the contactisation. Our main results state that the $h$-principle fails for these submanifolds in all dimensions. To the best of our knowledge, this is the first example of rigidity in the study of maximally non-integrable distributions, outside of contact topology.

First, we find an infinite family of $(2n+1)$-tori in the standard fat $(\C^{2n+1},\D_\std)$, with the following two properties: (1) They all represent the same formal prelegendrian class, (2) but they are not prelegendrian isotopic because they are distinguished by pseudoholomorphic curve invariants of their Legendrian lift.

Secondly, we define the notion of prelegendrian stabilization in $(\C^{2n+1},\D_\std)$. This allows us to take an arbitrary prelegendrian and produce another one, in the same formal class, whose Legendrian lift is loose.

In order to prove these results we also develop the fundamentals of the theory of prelegendrians. This includes: (1) introducing the notion of front projection in $(\C^{2n+1},\D_\std)$, (2) proving that pseudoholomorphic curve invariants are robust under perturbations of the fat structure, allowing us to transport our results to non-standard fat structures, (3) introducing a zooming argument showing that any fat structure in dimension $6$ admits prelegendrians.
\end{abstract}


\maketitle

\section{Introduction}

\subsection{Context}

In this paper we study manifolds $M^n$ endowed with tangent distributions $\D \subseteq TM$ of rank $r$. We are particularly interested in the following classic problem in Differential Topology: Is it possible to classify concrete classes of maximally non-integrable distributions, up to homotopy? We approach this through the lens of the $h$-principle \cite{gromovPartialDifferentialRelations1986, cieliebakIntroductionPrinciple2024a}.

The answer to this question depends on $(r,n)$ and, in most cases, it is still a completely open problem. The state-of-the-art may be summarised as follows: The study of \emph{contact structures}, $(r,n)=(2m,2m+1)$, is a rich and very active field due to the non-trivial interplay between rigidity and flexibility, as crystallized through the tight/overtwisted dichotomy \cite{Bennequin,Eliashberg:OT,BEM}. The homotopical study of \emph{even-contact structures}, $(r,n)=(2m-1,2m)$, was settled by McDuff using Gromov's convex integration \cite{McDuff:EvenContact}. The case of \emph{Engel structures}, $(r,n)=(2,4)$, has been studied in the last few years, leading to a number of flexibility results \cite{Vogel, CPPP, CPP:Loose, delPinoVogel:EngelLutz}, but whether any form of rigidity exists is still unknown. Flexibility has also been established for a few other classes \cite{Adachi:DerivedLength, MartinezAguinagaPino,MartinezAguinaga:Growth}. The key take-away is that rigidity remains mysterious outside the realm of contact structures.

In this article, we work with \emph{fat distributions}, a notion originating in Weinstein’s theory of \emph{fat bundles} \cite{Weinstein:Fat_Bundle, montgomeryTourSubriemannianGeometries2002}. These generalize contact distributions to higher corank. In the case $(4,6)$, fat distributions coincide with the \emph{elliptic distributions} conjectured to exhibit rigidity in \cite{MartinezAguinagaPino}. Fat distributions have also been studied in \cite{Bhowmick:hoeizontal_disk_in_fat46, BhowmickDatta:Horizontal_fat, Bhowmick:HorizontalContactPartially}.

Following the ``Legendrian route'' from contact topology to detect rigidity \cite{Bennequin,Eliashberg:20Years,CasalsMurphyPresas}, we introduce the class of \emph{prelegendrian submanifolds} of fat distributions. Our main contribution is that, in corank $2$, the $h$-principle fails for prelegendrian submanifolds. Thus, rigidity is present at the submanifold level in the fat setting.

\subsection{Fat manifolds and prelegendrians}

Let $M^n$ be a smooth manifold. Let $\S T^*M$ be the space of co-oriented hyperplanes and $\P T^*M$ the unoriented version. These manifolds carry canonical contact structures $( \S T^*M, \xi_\can)$ and $(\P T^*M,\xi_\can)$, known as the spaces of (co-oriented) contact elements. They are examples of \emph{Legendrian fibrations} and fit in a commuting diagram 
\begin{equation*}
     \begin{tikzcd}
\S^{n-1} \ar[rr] \ar[d,hook] & & \RP^{n-1} \ar[d,hook] \\
(\S T^*M,\xi_\can) \ar[dr] \ar[rr] & &  (\P T^*M,\xi_\can) \ar[dl] \\ 
& M &  
\end{tikzcd}
\end{equation*}
In particular, the map $(\S T^*M,\xi_\can)\rightarrow (\P T^*M,\xi_\can)$ is a contact double cover. 

Let $\D$ be a distribution on $M$. Its \emph{annihilator} is the subbundle $\D^\perp \defi\{\alpha \in T^*M:\alpha_{|\D}=0\}\subseteq T^*M$. The space of co-oriented hyperplanes in $M$ containing $\D$ is denoted by $\S\D^\perp\subseteq \S T^* M$. Similarly, we write $\P\D^\perp\subseteq \P T^* M$ for the hyperplanes containing $\D$. In this article, we are interested in the following geometric class of maximally non-integrable distributions that generalize contact structures to higher corank:
\begin{definition}
    Let $\D$ be a distribution in a smooth manifold $M$. We say that $\D$ is \emph{fat}, and $(M,\D)$ is a \emph{fat manifold}, if one of the following equivalent conditions holds:
    \begin{itemize}
        \item For every local non-zero section $\alpha:U\rightarrow \D^\perp$, $\d \a|_\D$ is a symplectic form on $\D$.
        \item $\cC(M,\D)\defi (\S\D^\perp,\xi_\can)\subseteq (\S T^*M,\xi_\can)$ is a contact submanifold, called the \emph{contactisation} of $(M,\D)$.
        \item $\cC_\P(M,\D)\defi (\P \D^\perp,\xi_\can)\subseteq (\P T^*M,\xi_\can)$, is a contact submanifold, called the \emph{projectivised contactisation} of $(M,\D)$.
        \item $\cS(M,\D)\defi (\D^\perp\setminus 0_M,\d\lambda_\can)\subseteq (T^*M\setminus 0_M,\d\lambda_\can)$ is a symplectic submanifold, called the \emph{symplectisation} of $(M,\D)$. \qedhere
    \end{itemize}
\end{definition}
The equivalence between these conditions is shown in Theorem \ref{thm:ContactisationExistence}, but is a well-known fact \cite{montgomeryTourSubriemannianGeometries2002}.

From the definition it follows that the rank of $\D$ is $2n$ even, and that both contactisations fit in a commuting diagram of \emph{isotropic fibrations}
\begin{equation*}
    \begin{tikzcd}
    \S^{k-1} \ar[rr] \ar[d,hook] & & \RP^{k-1} \ar[d,hook] \\
    \cC(M,\D) \ar[dr] \ar[rr] & &  \cC_{\P}(M,\D) \ar[dl] \\ 
    & (M,\D) &  
    \end{tikzcd}
\end{equation*}

where $k$ is the corank of $\D$. The map $\cC(M,\D)\rightarrow \cC_{\P}(M,\D)$ is a contact double cover.

\begin{remark}
    For a distribution $\xi$ of corank $1$ on a manifold $M$, being fat is the same as being contact. In this case, $\cC_{\P}(M,\xi)=(M,\xi)$. $\cC(M,\xi)$ is the co-orientable double cover of $(M,\xi)$, which is disconnected if $(M,\xi)$ is co-orientable. $\cS(M,\xi)$ is the (complete)\footnote{In contact topology it is customary to assume that $\xi=\ker\alpha$ is co-oriented. The choice of contact form $\alpha$ defines a splitting $\cS(M,\xi)\cong (M\times \R^*, d(t\alpha))$, and it is common to treat only the positive half $(M\times(0,+\infty), d(t\alpha))$ as the symplectisation of $(M,\xi)$.} symplectisation of $(M,\xi)$.

    In the degenerate case $(M,\D)=(M,\{0\})$ the (co-oriented) contactisation is the space of (co-oriented) contact elements.
\end{remark}

It follows that the contactisations of $(M,\D)$ are compact if $M$ is compact. In particular, \emph{Gray stability} applies, yielding a homotopy invariant of the fat distribution:
\begin{lemma}
    Let $(M,\D)$ be a compact fat manifold. The contactomorphism class of $\cC(M,\D)$ and $\cC_{\P}(M,\D)$ is invariant under homotopies of the underlying fat structure. 
\end{lemma}

This indicates that it is reasonable to pursue the use of symplectic and contact invariants (say, those arising from pseudoholomorphic curves) to establish \emph{rigidity phenomena} for fat manifolds. In this article we successfully follow this program to find rigidity for the following class of submanifolds:
\begin{definition}
    Let $(M^{2n+k},\D^{2n})$ be a fat manifold. A $(n+k-1)$-submanifold $\Lambda^{n+k-1}\subseteq (M,\D)$ is said to be \emph{prelegendrian} if $\dim T\Lambda\cap \D=n$.
\end{definition}
Sometimes, when dealing with a prelegendrian submanifold $\Lambda$, we will abuse notation and treat it as an embedding by fixing some inclusion. This will be clear from context.

\begin{remark}
    Prelegendrian submanifolds of contact manifolds are precisely Legendrians.
\end{remark}

\begin{remark}
    A generic Legendrian germ $L\subseteq \cC(M,\D)$ projects to a germ of a prelegendrian in $(M,\D)$. This follows from the fact that a generic germ is transverse to the projection to $(M,\D)$.
\end{remark}

The terminology \emph{prelegendrian} is justified by the following lifting property (\cref{prop:LiftingPreisotropics}). Let $\Lambda:S\hookrightarrow (M^{2n+k},\D^{2n})$ be a prelegendrian embedding. Then, there exists a unique\footnote{Being precise, one should fix a \emph{prelegendrian co-orientation} (\cref{prop:LiftingPreisotropics}). Nonetheless, all prelegendrians treated in this article will have an obvious prelegendrian co-orientation (Subsection \ref{convention:prelegendrianCoorientation}). Do note that one can Legendrian lift to the projective contactisation uniquely, without a preferred co-orientation.} Legendrian embedding $\cL(\Lambda): S\hookrightarrow \cC(M,\D)$, called the \emph{Legendrian lift}, making the following diagram commute:
\begin{equation*}
\begin{tikzcd}
                                             & \cC(M,\D) \ar[d] \\
      S  \ar[r,"\Lambda"] \ar[ur," \cL(\Lambda) "]   &  (M,\D)   
\end{tikzcd}
    \end{equation*}
The Legendrian lift $\cL(\Lambda)$ is naturally parametric. In the closed case, Gray stability then implies that the Legendrian isotopy class of the lift is a homotopy invariant. This is the key tool that we use to distinguish exotic prelegendrians.

\begin{remark}
    The approach we have just sketched is reminiscent of knot contact homology \cite{OoguriVafa,NG:KCH,NG:KCHSurvey,Shende:Conormal}, i.e. the use of co-normal lifts as a source of invariants of smooth submanifolds. Indeed, in the case of a trivial distribution $(M,\{0\})$, prelegendrian submanifolds are hypersurfaces, which can be viewed as the fronts of their Legendrian co-normal lifts in $(\S T^*M,\xi_\can)$.
\end{remark}

\subsection{Main results}

Even though we will ultimately prove results for general fat manifolds of corank $2$, many of our constructions deal with a concrete class thereof. Namely, the almost complex analogue of the space of (co-oriented) contact elements.

\subsubsection{Almost complex Grassmannians}

Let $(X^{2n+2},J)$ be an almost complex manifold, with $n\geq 1$. The space of $J$-complex hyperplanes in $(X,J)$ is denoted by $\Gr(X,J)$ and referred to as the \emph{complex Grassmannian} of $(X,J)$. There is a natural structure of $\CP^n$-bundle $\pi_\Gr:\Gr(X,J)\to (X,J)$.

According to \cref{prop:almostComplexGrass}, there is a natural corank $2$ fat distribution $\D_{\can}$ on $\Gr(X,J)$: 
\begin{equation*}
    \D_\can(W)\defi (\d\pi_\Gr)^{-1}(W), \qquad W\in \Gr(X,J).
\end{equation*}
We refer to the fat manifold $(\Gr(X,J), \D_\can)$ as the space of \emph{$J$-complex contact elements} of $(X,J)$. 

The \emph{standard fat distribution} $\D_{\std}$ on $\C^{2n+1}$ is defined as follows. We consider $\Gr(\C^{n+1}, J_\std) = \C^{n+1}\times \CP^n$ and restrict ourselves to those elements in $\CP^n$ representing hyperplanes graphical over $\C^n \times \{0\}$. This subspace can be canonically identified with $\C^{n+1} \times \C^n$ by considering the slope. Identically, we are removing a constant section in $\C^{n+1}\times \CP^n$ and thus passing to an affine chart in each fibre. The fat manifold $(\C^{2n+1},\D_\std)$ will be called the \emph{standard (corank $2$) fat space}. 

By construction, the contactisation of $(\Gr(X,J),\D_{\can})$ is the space of co-oriented contact elements $\cC(\Gr(X,J),\D_{\can})=(\S T^* X, \xi_\can)$, and there is a commuting diagram
\begin{equation*}
    \begin{tikzcd}
    \S^{2n+1} \ar[rr] \ar[d,hook] & & \CP^n \ar[d,hook] \\
    (\S T^* X, \xi_\can) \ar[dr,"{\pi_\F}"'] \ar["{\mathrm{Hopf}}",rr] & &  (\Gr(X,J),\D_{\can}) \ar["{\pi_\Gr}",dl] \\ 
    & (X,J) &  
    \end{tikzcd}
\end{equation*}
where the \emph{Hopf map} $(\S T^* X, \xi_\can)\rightarrow (\Gr(X,J),\D_\can)$ is defined by taking the Hopf fibration in each fibre. 

\subsubsection{Prelegendrian fronts}

Let $\Lambda\subseteq (\Gr(X,J),\D_\can)$ be a prelegendrian submanifold. Its \emph{prelegendrian front} is the singular hypersurface $\pi_\Gr(\Lambda)\subseteq (X,J)$. It follows from the construction that the prelegendrian front is nothing but the front of the Legendrian lift:
\begin{equation*}
    \pi_\Gr(\Lambda)=\pi_\F (\cL(\Lambda)).
\end{equation*}

Our first contribution is the characterisation of a class of prelegendrian fronts among all Legendrian fronts:
\begin{theorem}\label{thm:PreLegendrianFronts}
Let $\cF\subseteq (X,J)$ be a front with only self-intersections and cusp singularities. Then, $\cF$ is a prelegendrian front if and only if its singularity loci are \emph{co-real} in $(X,J)$.
\end{theorem}
This structural result allows for the study of prelegendrians from a topological perspective, via their fronts. In particular, it allows us to prove that the class of compact prelegendrians  $\Lambda\subseteq (\C^{2n+1},\D_{\std})$ is non-empty, which is something not obvious just from the definitions.

\subsubsection{Co-real spinning}

In view of \cref{thm:PreLegendrianFronts}, finding prelegendrians amounts to finding compact Legendrian fronts $\cF\subseteq \C^{n+1}=J^0(\R^{2n+1},\R)$ whose singularity loci are co-real. To achieve this, we will implement a (co-real) version of the front spinning construction introduced by Ekholm, Etnyre and Sullivan \cite{EES:NonIsotopicLegendrians} (see also \cite{EkholmEtnyreSabloff,Golovko,Lambert-Cole:LegendrianProducts}).

Co-real spinning is surprisingly flexible and will allow us to go beyond the existence of compact prelegendrians. Indeed, the main result of this manuscript reads:
\begin{theorem}\label{thm:InfinitelyPrelegendrianTorus}
There exist infinitely many prelegendrian embeddings
\begin{equation*}
    \Lambda_s: \T^{2n+1}\hookrightarrow (\C^{2n+1},\D_\std), \qquad s\in \Z_{>0},
\end{equation*}
all formally prelegendrian isotopic. These prelegendrians are pairwise not prelegendrian isotopic, since their Legendrian lifts $\cL(\Lambda_s)$ to the contactisation $\cC(\C^{2n+1},\D_{\std})$ are pairwise not Legendrian isotopic. 
\end{theorem}
The Legendrian lifts $\cL(\Lambda_s)$ are distinguished by their Legendrian Contact Homology \cite{Chekanov, EES:LCH,EES:LCHContactization, EliashbergGiventalHofer, Eliashberg:ICM}. It follows that the $h$-principle fails for prelegendrian submanifolds in corank $2$ fat distributions, and this failure can be detected by means of symplectic and contact invariants. To the best of our knowledge, this is the first instance of rigidity in the study of maximally non-integrable distributions beyond the contact case.

\begin{remark}
The \emph{standard fat structure} $\D_\std$ on the complex projective space $\CP^{2n+1}$ is the underlying real distribution of the standard holomorphic contact structure on $\CP^{2n+1}$. In particular, as an open fat manifold, $(\C^{2n+1},\D_\std)$ is diffeomorphic to any affine chart of $(\CP^{2n+1},\D_\std)$ (see \cref{subsec:complex_projective} for further details).

Consequently, the prelegendrians of our main result can be realized in $(\CP^{2n+1},\D_\std)$. In fact, the arguments in the proof of \cref{thm:InfinitelyPrelegendrianTorus} show that the prelegendrian tori $\Lambda_s$ remain pairwise not prelegendrian isotopic when considered inside $(\CP^{2n+1},\D_\std)$. See \cref{rmk:InfinitelyToriComplexProjectiveSpace}.
\end{remark}

\begin{remark}
In \cref{rmk:InfinitelyToriInfinitelyManyFat} we prove more: there exist infinitely many corank-$2$ fat manifolds, each containing an infinite family of pairwise non-isotopic but formally isotopic prelegendrians.
\end{remark}

\subsubsection{Prelegendrian $N$-pushing}

To realize the formal prelegendrian isotopy between the tori from \cref{thm:InfinitelyPrelegendrianTorus}, we will define an \emph{$N$-pushing operation} along a co-real submanifold $N$ of a prelegendrian front. This is an adaptation of the analogous Legendrian construction \cite{Eliashberg:Stein, EES:NonIsotopicLegendrians}.

In \cite{EES:NonIsotopicLegendrians} this operation is referred to as stabilization, but we prefer to call it $N$-pushing to distinguish it from the concrete case of stabilizing in a cusp chart to yield a loose Legendrian \cite{Murphy:Loose}. Given a Legendrian $L\subseteq (Y,\xi)$, with $\dim L\geq 2$, write $s(L)$ for the loose Legendrian formally equivalent to $L$.

A by-product of introducing prelegendrian $N$-pushing is that there is a stabilization procedure for prelegendrians in $\Gr(X,J)$. This is our second main result:
\begin{theorem}\label{thm:PrelegendrianStabilization}
    Let $\Lambda\subseteq (\C^{2n+1},\D_\std)$ be a prelegendrian submanifold. Then, there exists a prelegendrian $s(\Lambda)\subseteq (\C^{2n+1},\D_\std)$, such that
    \begin{itemize}
        \item $\Lambda$ and $s(\Lambda)$ are formally prelegendrian isotopic by a $C^0$-small formal isotopy.
        \item Their Legendrian lifts satisfy $\cL(s(\Lambda)) \simeq s(\cL(\Lambda)).$
    \end{itemize}
    We refer to $s(\Lambda)$ as a \emph{prelegendrian stabilization} of $\Lambda$.
\end{theorem}

\begin{remark}
   One can arrange for $\Lambda$ and $s(\Lambda)$ to only differ inside an arbitrarily small ball $\Op(p)$ around a point $p\in \Lambda$.
\end{remark}

Prelegendrian stabilization gives a clear procedure to find exotic pairs of prelegendrians: It suffices to construct $\Lambda$ with $\cL(\Lambda)$ non-loose. In the following result, this is established using fillability:
\begin{corollary}\label{cor:ExoticPrelegendrianPairsBall}
    There exists a prelegendrian embedding
    \begin{equation*}
        \Lambda: \T^{n+1}\times \S^{j_0}\times \cdots\times \S^{j_k} \hookrightarrow (\C^{2n+1},\D_\std),
    \end{equation*}
    where $j_0+\cdots+j_k=n$, such that $\cL(\Lambda)\subseteq \cC(\C^{2n+1},\D_\std)$ is non-loose. In particular, $\Lambda$ and $s(\Lambda)$ are formally prelegendrian isotopic but not prelegendrian isotopic. 
\end{corollary}

\subsubsection{Conormal lifts}

We were not able to find exotic pairs of prelegendrians in $(\C^{2n+1},\D_\std)$ whose underlying smooth manifold is not a (non-trivial) product. However, examples can be found in spaces of complex contact elements $(\Gr(X,J),\D_\can)$. Namely, we can consider the prelegendrian submanifold
\begin{equation*}
    \Lambda_M := (M,TM\cap J TM)\subseteq (\Gr(X,J),\D_\can)
\end{equation*}
associated to a smooth co-oriented real hypersurface $M\subseteq (X,J)$. Its Legendrian lift $\cL(\Lambda_M) = (M,TM)\subseteq (\S T^* X,\xi_\can)$ is the co-normal of $M$. Prelegendrian stabilization (\cref{thm:PrelegendrianStabilization}) and Lagrangian fillability imply that: 
\begin{corollary}\label{cor:ExoticPrelegendrianPairGrassmanian}
    Let $Y\subseteq (X,J)$ be a codimension $0$ compact submanifold with connected boundary $M=\partial Y$. Then, $\Lambda_M$ and $s(\Lambda_M)$ are formally prelegendrian isotopic but not prelegendrian isotopic. 
\end{corollary}

\begin{remark}
If we apply \cref{cor:ExoticPrelegendrianPairGrassmanian} to the case $(X,J) = \C^{n+1}$, we will produce prelegendrians in $\C^{n+1} \times \CP^n$ that are homologically essential and thus not contained in $\C^{2n+1}$. Indeed, the complex Gauss map $M \rightarrow \CP^n$ of a hypersurface is generally homotopically non-trivial. Finding prelegendrians contained in $(\C^{2n+1},\D_\std)$ was precisely the motivation behind \cref{thm:InfinitelyPrelegendrianTorus}.
\end{remark}

\subsubsection{General corank-$2$ fat distributions}

It is tempting to think of $(\C^{2n+1},\D_\std)$ as the \emph{model} fat distribution of corank-$2$, but this is not correct. It turns out that there is no Darboux theorem in this setting. Nonetheless, $(\C^3,\D_\std)$ serves as an \emph{approximate} model when $n=1$. This is explained in  \cref{sec:general}, where we show:
\begin{theorem} \label{thm:general46}
Fix a fat distribution $(M^6,\D^4)$. Then:
\begin{itemize}
    \item It contains compact prelegendrians.
    \item The prelegendrians it contains can be stabilised.
\end{itemize}
\end{theorem}
We also prove a similar (but weaker) statement in higher dimensions.

In a similar direction, consider the following problem: we fix a path of corank-$2$ fat distributions $(\D_s)_{s \in [0,1]}$, as well as a prelegendrian $\Lambda$ for $\D_0$. There is no Gray stability in this setting, so there is no reason for a path of prelegendrians $(\Lambda_s)_{s\in [0,1]}$ in $\D_s$ to exist with $\Lambda_0 = \Lambda$. However:
\begin{lemma} \label{lem:paths}
If $\Lambda$ is compact, a path $\Lambda_s$ does exist for small $s$. In dimension $6$ a path exists for all time if $\D_0 = \D_\std$.
\end{lemma}

On the side of rigidity, we moreover obtain:
\begin{theorem} \label{thm:noPaths}
Let $\Lambda$ and $\Lambda'$ be two non isotopic prelegendrians in $(\C^{2n+1},\D_\std)$, as produced by \cref{thm:InfinitelyPrelegendrianTorus} or  \cref{cor:ExoticPrelegendrianPairsBall}. Then, there is no path $(\D_s,\Lambda_s)_{s \in [0,1]}$ of corank-$2$ fat distributions and prelegendrians such that:
\begin{itemize}
    \item $\D_0 =\D_1= \D_\std$.
    \item $\D_s$ agrees with $\D_\std$ outside of a compact.
    \item $\Lambda_0 = \Lambda$ and $\Lambda_1 = \Lambda'$.
\end{itemize}
\end{theorem}
That is, rigidity persists even if we allow homotopies of the fat structures themselves.

\subsection{Open questions}

We finish the introduction with a set of open questions about fat manifolds and prelegendrians, that should be relevant for the further development of the theory.

In this article, we build prelegendrians using spinning. In particular, our prelegendrians are topological products. It would be interesting to find more general examples, perhaps relying on higher order singularities of prelegendrian fronts. Concretely:
\begin{question}\label{problem:PrelegendrianSpheres}
    Does there exist a prelegendrian embedding $\Lambda:\S^{2n+1}\hookrightarrow(\C^{2n+1},\D_\std)$?
\end{question}

In a different direction, one could ask about fat distributions of other coranks:
\begin{question}\label{question:RigidPrelegendrians}
    Does the $h$-principle fail for prelegendrians in corank $k>2$ fat manifolds?
\end{question}
We are currently working on the case $k=4$, using ideas and constructions similar to the ones presented here. For other coranks we still lack a geometric understanding of how fat structures look (making it difficult to manipulate prelegendrians topologically).

Analogously to the case of contact structures, there is a class of subcritical preisotropic submanifolds in fat manifolds (see \cref{def:Preisotropic,def:Prelegendrian}). The $h$-principle for subcritical preisotropic immersions, and prelegendrian immersions of open manifolds, has been established by Bhowmick, building on Gromov's work \cite{Bhowmick:HorizontalContactPartially,gromovCarnotCaratheodorySpacesSeen1996}. The following question is natural: 
\begin{question}
    Does the $h$-principle hold for subcritical preisotropic embeddings and prelegendrian immersions in corank $k>1$ fat manifolds? 
\end{question}
The case $k=1$ corresponds to the usual contact case in which the answer is known to be true by work of Gromov \cite{cieliebakIntroductionPrinciple2024a,gromovPartialDifferentialRelations1986}.

Our prelegendrian stabilization from \cref{thm:PrelegendrianStabilization} naturally leads to the following:
\begin{question}
    Does there exist a class of \emph{loose} prelegendrians (containing our stabilized prelegendrians), that is governed by an $h$-principle?
\end{question}

If the answer to this question is negative, then the following question also has a negative answer:
\begin{question}
    Is the isotopy class of the prelegendrian lift $\cL(\Lambda)$ a complete invariant within each class of formally equivalent prelegendrians? 
\end{question}
This question is a concrete incarnation of Eliashberg's ``pseudoholomorphic curves or nothing'' philosophy. The problem is also analogous to the fundamental question of whether the Legendrian isotopy class of the co-normal lift is a complete isotopy invariant. This is known to be true for knots in $\R^3$ by a celebrated result of Shende \cite{ENS:KCH_complete,Shende:Conormal}. 

We finish with the most fundamental question that motivates this work:
\begin{question}
Recall that fat distributions satisfy the $h$-principle on open manifolds, according to Gromov's $h$-principle for open and diff-invariant relations \cite{cieliebakIntroductionPrinciple2024a,gromovPartialDifferentialRelations1986}. Then, focusing on closed manifolds:
    \begin{itemize}
        \item [(a)] Does the $h$-principle fail for fat distributions of corank $k>1$? 
        \item [(b)] If the answer to (a) is positive, does there exist an \emph{overtwisted} class of fat distributions?
        \item [(c)] Is there a fat manifold with overtwisted contactisation? \qedhere
    \end{itemize}
\end{question}
We conjecture that the answer to (a) should be positive, at least for $k=2$. This aligns with the conjecture of Mart\'inez-Aguinaga and the second author about $(4,6)$ fat distributions in \cite{MartinezAguinagaPino}. 

\subsection{Outline of the article}

The article is organized as follows. In \cref{sec:Fat} we introduce \emph{fat manifolds} and \emph{prelegendrians}. We prove all the basic properties of these, such as the existence of contactisations and Legendrian lifts. Finally, we introduce \emph{formal prelegendrian embeddings}. \cref{sec:Review_coreal} reviews co-real submanifolds in almost complex manifolds. These will play a crucial role in our constructions of prelegendrians. In \cref{sec:examples_fat} we introduce several examples of corank-$2$ fat manifolds, including the \emph{space of complex contact elements}; we also determine their contactisations. This will be the main class of fat manifolds to be studied throughout the article. In \cref{sec:Leg_front} we give a quick recap of Legendrian fibrations and introduce \emph{simple} fronts. \cref{sec:Leg_stablization} reviews the \emph{stabilization} of higher dimensional Legendrians and \emph{loose Legendrians}. In \cref{sec:Leg_spinning} we review the Legendrian front-spinning operation. In \cref{sec:Preleg_front} we introduce the notion of \emph{prelegendrian front}, and establish \cref{thm:PreLegendrianFronts} and the existence of basic prelegendrian Reidemeister moves. Afterwards, we define a \emph{prelegendrian spinning} operation along certain co-real submanifolds as a method to produce prelegendrian fronts in $(\C^{2n+1},\D_\std)$. In \cref{sec:Preleg_stab} we adapt the Legendrian front operations from \cref{sec:Leg_stablization} to the prelegendrian setting. In \cref{sec:Existence_of_preleg} we prove the existence of special co-real submanifolds. By our prelegendrian spinning operation this implies the existence of compact prelegendrians in $(\C^{2n+1},\D_\std)$. These submanifolds are also used to prove the existence of prelegendrian stabilizations, i.e. \cref{thm:PrelegendrianStabilization}. In \cref{sec:Non_isotopic_preleg} we combine the previous constructions to produce our exotic prelegendrians: the proofs of \cref{thm:InfinitelyPrelegendrianTorus}, \cref{cor:ExoticPrelegendrianPairsBall} and \cref{cor:ExoticPrelegendrianPairGrassmanian} are provided. Our results about general fat manifolds of corank-$2$ are proven in  \cref{sec:general}.

\textbf{Acknowledgements:} AdP is funded by the NWO grant Vidi.223.118 “Flexibility and rigidity of tangent distributions". During this project, EF received support from an AMS-Simons Travel Grant. WZ is supported by the Spanish FPI predoctoral program PRE2020-092185, as well as by the ICMAT project PID2022-142024NB-I00, which covered travel and local expenses associated with his visits to AdP and EF; he is grateful to the Department of Mathematics at the University of Georgia and the Mathematical Institute of Utrecht University for their warm hospitality during these visits. This work forms part of the PhD thesis of WZ. The authors are indebted to Robert Cardona, Georgios Dimitroglou Rizell, Fabio Gironella, Peter Lambert-Cole, Marina Logares and Mar\'ia Pe Pereira for their careful reading and many valuable comments, and to Sushmita Venugopalan for numerous helpful conversations during the early stages of this project.

\section{Fat manifolds}\label{sec:Fat}

In this section we introduce the main characters of the article: \emph{fat manifolds} and \emph{prelegendrian submanifolds}. We prove the key properties about them: the existence of contactisations (\cref{thm:ContactisationExistence}) and Legendrian lifts (\cref{prop:LiftingPreisotropics}). We conclude the section by introducing formal prelegendrian embeddings, which is necessary to set up the $h$-principle problem.

\subsection{Fat Distributions}
Let $\D^r\subseteq TM$ be a distribution of rank $r$ on a smooth manifold $M$. The \emph{Lie-bracket} of $\D$ with itself is the $C^\infty(M)$-module of vector fields \begin{equation*}
    [\D,\D] \defi \{[u_1,u_2] \st u_1, u_2\in\Gamma(\D) \}.
\end{equation*}
Throughout this article we will always assume that $\D$ is \emph{regular}, meaning that the module $[\D,\D]$ has constant pointwise rank, hence defines a distribution that we denote by the same symbol. 

The \emph{curvature map} of $\D$ is the bundle morphism
\begin{equation*}
    \Omega:\D \wedge \D \to TM/\D, \quad (X,Y)\mapsto -[X,Y] \mod \D.
\end{equation*}

For each distribution $H\subseteq TM$ such that $\D\subseteq H$, there is also an associated \emph{$H$-directional curvature}
\begin{align*}
    \Omega_H: \D\times \D &\to TM/H\\
    (X,Y)&\mapsto -[X,Y] \mod H.
\end{align*}

It is often convenient to dualize this picture. The annihilator of $\D$ is the subbundle
\begin{equation*}
    (TM/\D)^* \simeq \D^\perp = \{\alpha \in T^*M:\alpha|_{\D}=0\},
\end{equation*}
consisting of all covectors vanishing on $\D$. Dualizing $\Omega$ and using Cartan’s formula, we obtain the \emph{dual curvature} map
\begin{equation*}
    \Omega^*:\D^\perp \to \wedge^2\D^*, \quad \a\mapsto \a\circ\Omega = \a \circ \Omega_{\ker\alpha}= \d \a|_\D.
\end{equation*}
This map encodes the same information as $\Omega$, but from the perspective of $1$-forms rather than vector fields.

The basic link between curvature and bracket-generating properties is well known:
\begin{lemma}\label{lem:equiv_step2}
    Let $\D$ be a distribution on a manifold $M$. The following statements are equivalent:
    \begin{enumerate}
        \item $[\D,\D]=TM$, i.e. $\D$ is bracket-generating of step $2$.
        \item The dual curvature $\Omega^*$ is a monomorphism.
        \item The curvature $\Omega$ is an epimorphism.
    \end{enumerate}
\end{lemma}

We are interested in a special subclass of step-$2$ distributions whose curvature exhibits maximal non-degeneracy.
\begin{definition}
A distribution $\D$ of corank $k$ on a manifold $M$ is said to be \emph{fat} at a point $p\in M$ if, for every non-vanishing $1$-form $\alpha \in \D^\perp$ at $p$, the $2$-form $\Omega^*(\alpha)$ defines a symplectic form on $\D_p$. 

A distribution $\D$ is \emph{fat} if it is fat at every point $p \in M$. The pair $(M,\D)$ is called a \emph{fat manifold}. 
\end{definition}
In other words, $(M,\D)$ is a fat manifold if and only if for every (local) non-zero section of $\alpha:U\rightarrow \D^\perp$, the vector bundle $(\D|_{U},d\alpha)$ is symplectic. In particular, fat distributions of corank $1$ are precisely contact distributions.
It follows that the rank of a fat distribution $\D$ is always an even number $2n$. We will usually write $(M^{2n+k},\D^{2n})$ to denote fat manifolds of corank $k$. 

\begin{remark}
    A fat distribution $\D^{2n}$ is necessarily bracket-generating of step $2$, by \cref{lem:equiv_step2}. Moreover, fat distributions are maximally non-integrable in the following sense: The distribution $\D^{2n}$ is fat if and only if for every hyperplane distribution $H\subseteq TM$ with $\D\subseteq H$, the $H$-directional curvature
    \begin{equation*}
        \Omega_H:\D\wedge\D\rightarrow TM/H, (X,Y)\mapsto -[X,Y] \mod H,
    \end{equation*}
    is a symplectic form with values in the line bundle $TM/H$. In the case $H=\ker(\alpha)$, $\Omega_H=\Omega^*(\alpha)$. In particular, fat distributions come equipped, at least locally, with a $\S^{k-1}$-family of conformal symplectic structures. 
\end{remark}

Inspired by \cite{BhowmickDatta:Horizontal_fat}, we adopt their construction and refer to it as the \emph{fatization}.
\begin{example}
    Let $(W^{2n},\lambda_1,\ldots,\lambda_k)$ be a smooth manifold $W^{2n}$ equipped with $(k-1)$-sphere of Liouville forms, i.e. a $k$-tuple $\lambda_1,\ldots,\lambda_k\in \Omega^1W$ of $1$-forms such that 
    \begin{enumerate}
        \item $\d\lambda_i$ is a symplectic form for all $i=1,\ldots,k$.
        \item For every $(a_1,\ldots,a_k)\in \S^{k-1}$ the combination $a_1\d\lambda_1+a_2 d\lambda_2+\cdots+a_k\d\lambda_k$ is  symplectic.
    \end{enumerate}
    A concrete example is $\R^4$. It admits a natural $2$-sphere of Liouville forms coming from the complex structures $i,j$ and $k$ given by the unit quaternions. More generally, exact hyperkähler manifolds provide a natural source of examples.
    
    The \emph{fatization} of $(W^{2n},\lambda_1,\ldots,\lambda_k)$ is the corank $k$ fat manifold
    \begin{equation*}
        (W\times \R^k, \D=\ker(\d z_1-\lambda_1)\cap\cdots\cap \ker (\d z_k-\lambda_k ) ),
    \end{equation*}
    where $(z_1,\ldots,z_k)$ are the standard coordinates on $\R^k$. To produce compact examples one can quotient $\R^k$ to a torus $\T^k$.

    The fatization is a fat manifold. Indeed, every non-zero form $\alpha\in \D^\perp$ can be uniquely expressed as
    \begin{equation*}
        \alpha=t a_1 (\d z_1-\lambda_1)+ta_2(\d z_2-\lambda_2)+\cdots ta_2(\d z_k-\lambda_k),
    \end{equation*}
    where $(t,(a_1,\ldots,a_k))\in(0,\infty)\times \S^{k-1}$ are spherical coordinates. It follows that $\d\alpha=t(a_1 \d\lambda_1 +\cdots  a_k\d\lambda_k)$ is symplectic on $\D$.

    In the non-exact setting, one can consider a $k$-tuple of symplectic forms such that any combination is symplectic. If all the forms in the tuple are of integral class, one can consider the associated $\S^1$-bundles. Taking their product produces a $\T^k$-bundle with fat connection.
\end{example}

Beyond the contact case, fatness is rare and highly constrained. The following result provides a precise set of numerical restrictions:
\begin{theorem}[Rayner \cite{Rayner}]
    Let $\D^r$ be a distribution of positive rank $r$ and corank $k$ on a manifold $M^{r+k}$. If $\D^r$ is fat, then the following constraints hold:
    \begin{itemize}
        \item $r$ is even and $r\geq k+1$.
        \item If $k\geq2$, then $4$ divides $r$.
        \item The $(r-1)$-sphere $S^{r-1}$ admits $k$-many linearly independent vector fields.
    \end{itemize}
    Conversely, given any pair $(r,k)$ satisfying the above conditions, there is a germ of fat distribution of rank $r$ and corank $k$.
\end{theorem}

\subsection{Contactisation and symplectisation of a fat manifold} \label{sec:contactisation}

Given a manifold $M$, we denote by $\lambda_\can$ the canonical Liouville $1$-form in $T^*M$. The manifold $T^*M$ carries a natural $\R^*$-action (resp. $\R^+$-action), induced by the associated Liouville vector field. The quotient of $T^*M \setminus 0$ under this action is the space $\P T^*M$ of hyperplanes in $M$ (resp. the space $\S T^*M$ of co-oriented hyperplanes in $M$). Since the action conformally expands the symplectic form, it defines a canonical contact structure $(\P T^* M, \xi_\can)$ known as \emph{space of contact elements on $M$} (resp. $(\S T^* M,\xi_\can)$ known as \emph{space of co-oriented contact elements}). 

Consider the bundle projections $\pi_\P: \P T^* M\rightarrow M$ and $\pi: \S T^* M \rightarrow M$. Then, the canonical contact structure on $\P T^* M$ is 
\begin{equation*}
        \xi_\can(\a) =(\d_\a\pi_\P)^{-1}(\ker \a),\quad \a \in \P T^* M.
\end{equation*}
The canonical contact structure on $\S T^* M$ is $\xi_\can =\ker \lambda_\can$, hence at a point $\a\in \S T^*M$ is also given by $\xi_\can (\a)=(\d_\a\pi)^{-1} (\ker \a)$. The canonical double covering $(\S T^* M, \xi_\can)\rightarrow (\P T^* M,\xi_\can)$ is the co-oriented double contact cover of $(\P T^* M,\xi_\can)$, and commutes with the bundle projections.

Fix a distribution $\D\subseteq TM$ on $M$ of corank $k$. Its annihilator $\D^\perp \subseteq T^* M$ is preserved by the previous actions. Thus, we define the $\RP^{k-1}$-subbundle $\P \D^\perp\subseteq \P T^*M$ (resp. the $\S^{k-1}$-subbundle $\S \D^\perp\subseteq \S T^* M$). Notice that $\P \D^\perp$ (resp. $\S \D^\perp$) is nothing but the space of hyperplanes (resp. co-oriented hyperplanes) in $M$ that contain $\D$. The bundle projections will be denoted as before. 

The following characterisation of fatness is the key input that we will use in this article:
\begin{theorem}\label{thm:ContactisationExistence}
    Let $\D\subseteq TM$ be a distribution on $M$. The following are equivalent:
    \begin{enumerate}
        \item $\D$ is fat.
        \item $\D^\perp \setminus 0_M \subseteq (T^*M,\d \lambda_{\can}) $ is a symplectic submanifold.
        \item $\P\D^\perp \subseteq (\P T^*M,\xi_\can)$ is a contact submanifold.
        \item $\S\D^\perp \subseteq (\S T^*M,\xi_\can)$ is a contact submanifold.
    \end{enumerate}
\end{theorem}
\begin{proof}
    The equivalence between (2), (3) and (4) is a standard fact in contact and symplectic geometry \cite{cieliebakIntroductionPrinciple2024a, geigesIntroductionContactTopology2008}. We prove the equivalence between (1) and (2).

    Write $m=\dim(M)$, $k=\operatorname{corank}(\D)$ and $r=\operatorname{rank}(\D)$. Then, $\dim(\D^\perp) = m + k = r + 2k$ is even if and only if $r$ is even. From now on, we will assume that $r=2n$ is even, hence $\dim(\D^\perp)=2(n+k)$. 
    
    Let $\mathbf{q} = (q_1, \ldots, q_m)$ be a local coordinate chart on an open set $U \subseteq M$. Trivialize the distribution $\D$ on $U$ by a collection of $1$-forms $\{\alpha_1, \ldots, \alpha_k\}$ such that:
    \begin{equation*}
        \D \overset{\text{loc.}}{=}\ker\a_1\cap\cdots \cap \ker\a_k.
    \end{equation*}
    That is, the collection $\{\alpha_1, \ldots, \alpha_k\}$ forms a local frame for $\D^\perp$. Denote the corresponding fibre coordinates by $(a_1, \ldots, a_k)$.

    Hence, $(\mathbf{q}; a_1, \ldots, a_k)$ gives local coordinates on $\D^\perp$. The canonical Liouville $1$-form on $T^*M$ restricts to:
    \begin{equation*}
        \lambda = \sum_{i=1}^k a_i \alpha_i,
    \end{equation*}
    and its differential is:
    \begin{equation} \label{eq:LiouvilleAnnihilator}
        \omega \defi \d\lambda =  \d\Big(\sum_{i=1}^k a_i\a_i\Big) = \sum_{i=1}^k \Big( \d a_i\wedge \a_i + a_i\d\a_i \Big)
    \end{equation}
    
    Since the $1$-forms $\alpha_i$ are independent of the fibre coordinates $a_i$, the maximal wedge power of $\omega$ has the form:
    \begin{equation}\label{eq:top_power_fat}
        \omega^{(n+k)} = C \bigwedge_{i=1}^k\big(\d a_i\wedge \a_i \big) \wedge \big( \sum_{i=1}^k a_i\d\a_i \big)^{n},
    \end{equation}
    for a non-zero combinatorial number $C$. Realize that $T\D^\perp \cong TM \oplus \mathrm{Vert} \cong \D \oplus TM/\D \oplus \mathrm{Vert}$. The first term in \eqref{eq:top_power_fat} is a volume form on the subspace $TM/\D \oplus \mathrm{Vert}$. Therefore, $\omega$ is symplectic in $\D^\perp\setminus 0_M$ if and only if the second term in \eqref{eq:top_power_fat} is a volume form on $\D^{2n}$ for all $(a_1,\ldots, a_k)\neq (0,\ldots, 0)$, which is precisely the definition of fatness. This concludes the proof. 
\end{proof}

\begin{definition}
    Let $(M^{2n+k},\D^{2n})$ be a fat manifold.
    \begin{enumerate}
        \item The \emph{projectivised contactisation} of $(M,\D)$ is  $\cC_\P(M,\D)\defi (\P\D^\perp,\xi_\can)$,
        \item The (co-oriented) \emph{contactisation} of $(M,\D)$ is $\cC(M,\D)\defi (\S \D^\perp,\xi_\can)$,
        \item The \emph{symplectisation} of $(M,\D)$ is $\cS(M,\D)\defi (\D^\perp\setminus 0_M,\d\lambda_\can)$. \qedhere
    \end{enumerate}
\end{definition}

\subsection{Preisotropic and prelegendrian submanifolds} \label{subsec:preisotropic_submanifold}

Let us introduce a class of submanifolds adapted to the fat geometry:
\begin{definition}\label{def:Preisotropic}
    Let $(M^{2n+k}, \D^{2n})$ be a fat manifold. A submanifold $\Lambda\subseteq (M^{2n+k},\D^{2n})$ is said to be \emph{preisotropic} if
    \begin{equation*}
        (T\Lambda  + \D)|_{\Lambda}\subseteq TM|_{\Lambda}
    \end{equation*}
    has constant corank $1$. 
\end{definition}
Equivalently, $\Lambda \subseteq (M^{2n+k},\D^{2n})$ is preisotropic if
\begin{equation*}
    \dim(T\Lambda\cap \D)=\dim(\Lambda)-k+1.
\end{equation*}

A \emph{preisotropic immersion} is a parametrization of a possibly immersed preisotropic submanifold, and a \emph{preisotropic embedding} is a parametrization of an embedded preisotropic submanifold. 

\begin{remark}
    In the contact case, preisotropic submanifolds are precisely the isotropics. 
\end{remark}

Every preisotropic submanifold  $\Lambda\subseteq (M^{2n+k},\D^{2n})$ carries a hyperplane distribution 
\begin{equation*}
    H_\Lambda\defi (T\Lambda  + \D)_{|\Lambda},
\end{equation*}
that we will refer to as the \emph{preisotropic hyperplane}.
A \emph{preisotropic co-orientation} is a co-orientation of $H_\Lambda$, i.e. an orientation of the line bundle $l_\Lambda\defi TM|_{\Lambda}/H_\Lambda$. The preisotropics treated in this article carry an obvious preisotropic co-orientation (Subsection \ref{convention:prelegendrianCoorientation}), unless otherwise stated. 

\begin{remark}
    In the contact case $(M^{2n+1},\D^{2n})$, the preisotropic hyperplane is the contact structure, $H_\Lambda=\D_{|\Lambda}$. In particular, in co-oriented contact manifolds every preisotropic carries a natural preisotropic co-orientation.
\end{remark}

Since $(M^{2n+k},\D^{2n})$ is fat, the curvature of $\D$ along the preisotropic hyperplane $H_\Lambda$ defines a symplectic form
\begin{equation*}
    \Omega_{H_\Lambda}:D|_{\Lambda}\wedge \D|_{\Lambda} \longrightarrow l_\Lambda
\end{equation*}
with values on the line bundle $l_\Lambda$.
\begin{lemma}\label{lem:DimensionPreisotropic}
Let $(M^{2n+k},\D^{2n})$ be a fat manifold and $\Lambda\subseteq (M^{2n+k},\D^{2n})$ a preisotropic submanifold. Then, $T\Lambda\cap \D|_{\Lambda}\subseteq (\D|_{\Lambda}, \Omega_{H_\Lambda})$ is an isotropic sub-bundle. In particular, $\rk(T\Lambda\cap \D)\leq n$, equivalently, $\dim(\Lambda)\leq n+k-1$.
\end{lemma}
\begin{proof}
    The computation is local so we may assume that $\Lambda$ has a fixed preisotropic co-orientation, i.e $H_\Lambda=\ker \alpha$ for some $1$-form $\alpha$, and identify $\Omega_{H_\Lambda}=\d\alpha$.
Since $\iota^*\alpha=0$, for the inclusion map $\iota:\Lambda\hookrightarrow  M$, we conclude that $\iota^* \d\alpha=0$ on $T\Lambda\cap \D$, from which the result follows.
\end{proof}

This motivates the following definition:
\begin{definition}\label{def:Prelegendrian}
    A preisotropic submanifold $\Lambda\subseteq (M^{2n+k},\D^{2n})$ is said to be \emph{subcritical} if $\dim(\Lambda)< n+k-1$. A \emph{prelegendrian submanifold} is a preisotropic submanifold $\Lambda$ of the maximal dimension $\dim(\Lambda)=n+k-1$.
\end{definition}

\subsubsection{Isotropic and Legendrian lifts to the contactisation}

The following lifting property justifies the terminologies \emph{preisotropic} and \emph{prelegendrian}:
\begin{proposition}\label{prop:LiftingPreisotropics}
Let $(M^{2n+k},\D^{2n})$ be a fat manifold. 
 
\begin{itemize}
     \item If $u: S\to (M^{2n+k},\D^{2n})$ is a preisotropic immersion, then there exists a unique isotropic immersion $\hat{u}:S\to \cC_\P(M,\D)$, that makes the following diagram commute:
    \begin{equation}\label{diagram:projectivised_legendrian_lift}
    \begin{tikzcd}
                                                 & \cC_\P(M,\D) \ar[d] \\
          S  \ar[r,"u"] \ar[ur,"\hat{u}",dashed] &  (M,\D)   
    \end{tikzcd}
    \end{equation}
    \item 
    If $u: S\to (M^{2n+k},\D^{2n})$ is a co-oriented preisotropic immersion, then there exists a unique isotropic immersion $\widetilde{u}:S\to \cC(M,\D)$, that makes the following diagram commute:
    \begin{equation}\label{diagram:legendrian_lift}
\begin{tikzcd}
                                             & \cC(M,\D) \ar[d] \\
                                             & \cC_\P(M,\D) \ar[d] \\
      S  \ar[r,"u"] \ar[ur,"\hat{u} "]  \ar[uur, "\tilde{u}", dashed] &  (M,\D)   
\end{tikzcd}
        \end{equation}

    \end{itemize}

        Moreover, if $u$ is an embedding, then $\hat{u}$ and $\tilde{u}$ are also embeddings.
\end{proposition}
\begin{proof}
Recall $\cC_\P(M,\D)=(\P\D^\perp,\xi_\can)\subseteq (\P T^*M,\xi_\can)$ and $\cC(M,\D)=(\S \D^\perp,\xi_\can)\subseteq (\S T^*M,\xi_\can)$. We denote the projections by $\pi_{\P}:\cC_\P(M,\D)\to (M,\D)$ and $\pi:\cC(M,\D)\to (M,\D)$ respectively. 
    
For each $p\in S$, since $u$ is preisotropic, there is a unique $1$-form $\hat{\alpha}_p\in \D^\perp_{u(p)}$ defined up to $\R^*$-scaling, such that 
\begin{equation*}
    \d u(T_pS)+\D_{u(p)} = \ker \tilde{\alpha}_p.
\end{equation*}
Analogously, if $u$ is co-oriented preisotropic, there is $\tilde{\alpha}_p$, uniquely defined up to positive $\R^+$-scaling, inducing the given co-orientation.

The lifts are defined by $\hat{u}(p)\defi [\hat{\alpha}_p]$ and $\tilde{u}(p)\defi [\tilde{\alpha}_p]$, respectively. Differentiating the identities $\pi_{\P}\circ \hat{u}=u$ and $\pi\circ \tilde{u}=u$, we obtain:
\begin{equation*}
    \d\pi_{\P}(\im \d_p\hat{u}) = \im \d_p u \subseteq \ker \hat{\alpha}_p
\end{equation*}
and 
\begin{equation*}
    \d\pi(\im \d_p\tilde{u}) = \im \d_p u \subseteq \ker \tilde{\alpha}_p.
\end{equation*}

From here it follows that both maps are isotropic immersions. Indeed, the immersion property is obvious. For the isotropic property, realize that at the point $[\hat{\alpha}_p]\in \P\D^\perp$ the canonical contact structure is defined as $\xi_\can ([\hat{\alpha}_p]) = d\pi_{\P}^{-1} \ker(\hat{\alpha}_p)$, the same argument applies in the co-oriented case. 
\end{proof}

\begin{definition}
    Let $(M^{2n+k},\D^{2n})$ be a fat manifold and $u:S\to (M,\D)$ a (co-oriented) preisotropic immersion. The isotropics $\mathcal{I}_\P(u)\defi\hat{u}:S\to \cC_\P(M,\D)$ and $\mathcal{I}(u)\defi \tilde{u}:S\to \cC(M,\D)$, from the proposition above, are called the \emph{projective isotropic lift} and the \emph{isotropic lift} of $u$, respectively.
\end{definition}

In the prelegendrian case, we will use the notation $\cL_\P(u)=\mathcal{I}_\P(u)$ and $\cL(u)=\mathcal{I}(u)$ to denote the lifts, and we will refer to them as \emph{Legendrian lifts}.  When dealing with submanifolds (not maps) we will use the same notation to denote  the lifts, i.e. if $\Lambda\subseteq (M,\D)$ is preisotropic, resp. prelegendrian, we will write $\mathcal{I}(\Lambda)$, resp. $\cL(\Lambda)$, to denote the lifts, and similarly for the projective ones. 

The construction of the lift is continuous for families of preisotropics. In particular, the isotropic lifts produce effective invariants to distinguish preisotropics. For later use, we highlight the following:

\begin{corollary}\label{cor:leg_lift_invariant}
    Let $(M^{2n+k},\D^{2n})$ be a fat manifold, and $\Lambda_0, \Lambda_1:S\hookrightarrow (M,\D)$ two (co-oriented) prelegendrian embeddings. Then,
    \begin{itemize}
        \item If $\cL_\P(\Lambda_0)$ and $\cL_\P(\Lambda_1)$ are not Legendrian isotopic, then $\Lambda_0$ and $\Lambda_1$ are not prelegendrian isotopic. 
        \item If $\cL(\Lambda_0)$ and $\cL(\Lambda_1)$ are not Legendrian isotopic, then $\Lambda_0$ and $\Lambda_1$ are not co-oriented prelegendrian isotopic.
    \end{itemize}
\end{corollary}

\begin{remark}
    In practice, we will want to distinguish prelegendrians without the co-orientation information, using their lift to the (co-oriented) contactisations. For this, one would need to distinguish the set of all possible co-oriented Legendrian lifts. When the prelegendrian is connected, there are just two different lifts, and both are related by the co-orientation reversing contactomorphism of $\cC(M,\D)$ given the antipodal action in the $\S^{k-1}$-fibres. Therefore, it is enough to distinguish $\cL(\Lambda_0)$ and $\cL(\Lambda_1)$ by means of a Legendrian invariant preserved by this action, to conclude that $\Lambda_0$ and $\Lambda_1$ are not prelegendrian isotopic.  All the Legendrian invariants treated in this article satisfy that property. 
\end{remark}

\subsection{Formal preisotropics and prelegendrians}

The study of preisotropic immersions in fat distributions was initiated by Gromov \cite{gromovCarnotCaratheodorySpacesSeen1996}, who referred to them as \emph{partially horizontal immersions}. It was subsequently developed by Bhowmick \cite{Bhowmick:HorizontalContactPartially}, who in certain cases completed the $h$-principle analysis via the Hamilton–Moser implicit function technique.

To state the $h$-principle and prepare for the exotic prelegendrians appearing in our main results, we introduce the following notion:
\begin{definition}
    Let $(M^{2n+k},\D^{2n})$ be a fat manifold, and let $S$ be a smooth manifold.
    \begin{enumerate}
        \item A \emph{formal preisotropic immersion} is a pair $(f,F)$, such that:
    \begin{itemize}
        \item $f:S\to M$ is a smooth map,
        \item $F:TS\to TM$ is a bundle monomorphism covering $f$ such that 
    \begin{itemize}
        \item [(a)]$H_p=F(T_pS) + \D_{f(p)}$ has corank $1$, for all $p\in S$, and 
        \item [(b)]$F(T_pS)\cap \D_{f(p)}\subseteq (\D_{f(p)}, \Omega_{H_p})$ is an isotropic subspace, for all $p\in S$.
    \end{itemize}
\end{itemize}
    \item A \emph{formal preisotropic embedding} is a pair $(f,F_s)$ such that:
    \begin{itemize}
        \item $f:S\hookrightarrow M$ is an embedding,
        \item $F_s:TS\to TM$ with $s\in[0,1]$ is a $1$-parameter family of bundle monomorphisms covering $f$ such that $F_0=\d f$,
        \item $(f,F_1)$ is a formal preisotropic immersion.
    \end{itemize}
    \item A \emph{formal prelegendrian immersion (resp. embedding)} is a formal preisotropic immersion (resp. embedding) of dimension $n+k-1$. \qedhere
    \end{enumerate}
\end{definition}
A formal preisotropic immersion is \emph{holonomic} if $F=df$. Similarly, for embeddings we ask $F_s = df$ for each $s$. Note moreover that formal preisotropic immersions $(f,F)$ have naturally associated co-orientation bundles $T_{f(p)}M/(F(T_pS) +\D_{f(p)})$, so we can ask that a co-orientation is fixed.

For immersions, the following $h$-principle holds:
\begin{theorem}{{\cite[Theorem 5.3]{Bhowmick:HorizontalContactPartially}}}
    Let $(M^{2n+k},\D^{2n})$ be a fat manifold and let $S$ be a smooth manifold. Then, any subcritical formal preisotropic immersion $(f,F)$ from $S$ to $(M,\D)$ is homotopic to a holonomic one.  In the prelegendrian case, the same statement holds provided that $S$ is an open manifold.
\end{theorem}

\begin{remark}
    In the contact case, our definition recovers the usual notion of formal isotropic immersions and embeddings. In this setting, subcritical isotropic embeddings satisfy the $h$-principle \cite{cieliebakIntroductionPrinciple2024a}; this follows from the $h$-principle for isocontact embeddings in codimension $\geq 4$. By contrast, Bhowmick establishes the result above using microflexibility and microextension, which yield immersions rather than embeddings. The present paper pursues the failure of the $h$-principle for prelegendrian embeddings. The $h$-principle problem for subcritical preisotropic embeddings and prelegendrian immersions remains to be explored. 
\end{remark}

\subsection{Formal preisotropic lifts}

One of the main goals of this article is to construct prelegendrians that are formally isotopic but not prelegendrian isotopic. To distinguish such formally equivalent prelegendrians, we will make use of the Legendrian isotopy class of their Legendrian lifts (\cref{prop:LiftingPreisotropics}). The next result shows that formally equivalent prelegendrians \emph{cannot} be distinguished by the formal class of their Legendrian lifts. 
\begin{proposition}\label{prop:formal_preisotropic_lift}
    Let $(M^{2n+k},\D^{2n})$ be a fat manifold, and $S$ a smooth manifold.
    \begin{itemize}
        \item Suppose $(f,F):S\to(M,\D)$ is a formal preisotropic immersion. Then there exists a formal isotropic immersion $(\hat{f},\widehat{F}):S \to \cC_\P(M,\D)$, canonical up to contractible choices, that makes the following diagram commute:
        \begin{equation*}
        \begin{tikzcd}[cramped]
        	&& {\cC_\P(M,\D)} \\
        	\\
        	S && {(M,\D)}
        	\arrow[from=1-3, to=3-3]
        	\arrow["{(\hat{f},\widehat{F})}", dashed, from=3-1, to=1-3]
        	\arrow["{(f,F)}", from=3-1, to=3-3]
        \end{tikzcd}
        \end{equation*}
        \item Suppose $(f,F):S\to(M,\D)$ is a co-oriented formal preisotropic immersion. Then there exists a formal isotropic immersion $(\tilde{f},\widetilde{F}) : S \to \cC(M,\D)$, canonical up to contractible choices, that makes the following diagram commute:
        \begin{equation*}
            \begin{tikzcd}[cramped,sep=scriptsize]
            	&&& {\cC(M,\D)} \\
            	\\
            	&&& {\cC_\P(M,\D)} \\
            	\\
            	S &&& {(M,\D)}
            	\arrow[from=1-4, to=3-4]
            	\arrow[from=3-4, to=5-4]
            	\arrow["{(\tilde{f},\widetilde{F})}", dashed, from=5-1, to=1-4]
            	\arrow["{(\hat{f},\widehat{F})}"', from=5-1, to=3-4]
            	\arrow["{(f,F)}"', from=5-1, to=5-4]
            \end{tikzcd}
        \end{equation*}
        
    \end{itemize}

    Moreover, the same statement holds for formal preisotropic embeddings. 
\end{proposition}
\begin{proof}
We will work out the lifting problem to $\cC_\P(M,\D)$ for formal preisotropic immersions $(f,F): S \rightarrow (M,\D)$. We will address embeddings at the end. The case of $\cC(M,\D)$ is left for the reader.

The proof amounts to showing that there is an affine bundle $A \rightarrow S$ whose sections are the formal isotropic lifts. This will imply that the space of lifts is contractible.

We let $H := F(TS)+\D \subseteq f^*TM \rightarrow S$ be the preisotropic hyperplane bundle associated to $(f,F)$. $H$ uniquely defines a map $\widetilde f: S \rightarrow \P\D^\perp$ lifting $f$ so that $\ker(\widetilde f(s)) = H_s$. We therefore obtain a bundle $\widetilde f^*\xi \rightarrow S$. The differential of the projection $\pi$ defines a bundle epimorphism $d\pi: \widetilde f^*\xi \rightarrow H$.

The bundle $\xi$ has a canonical conformal symplectic structure, which we denote by $\omega$. We obtain $(\widetilde f^*\xi,\widetilde f^*\omega)$ over $S$. We define $A \rightarrow S$ to be the subbundle of $\textrm{Hom}(TS,\widetilde f^*\xi)$ consisting of lifts of $F$ whose image is isotropic. We prove that it is affine by exhibiting suitable local trivialisations. Working locally, we choose a 1-form $\beta$ annihilating $H$. Note that $\widetilde f$ is its projective class. We moreover fix a local frame $\{\alpha_i\}$ of $\D^\perp$, which we assume satisfies $\alpha_k = \beta$.

According to Equation \eqref{eq:LiouvilleAnnihilator}, the symplectic form in $\beta^*T\D^\perp$ reads:
\begin{equation*}
    d\beta + \sum_{i=1}^k da_i \wedge \alpha_i.
\end{equation*}
We are leaving some pullbacks via $\pi$, $f$, and $\widetilde f$, implicit. Restricting to $\widetilde f^*\xi$ yields:
\[ \widetilde f^*\omega =  d\beta|_H + \sum_{i=1}^{k-1} da_i \wedge (\alpha_i)|_H. \]
This can be written more conceptually as follows. Let $Q := H/f^*\D \rightarrow S$, which we see as a concrete complement of $f^*\D$ within $H$ using a splitting. The formula above for $\widetilde f^*\omega$ identifies the vertical subspace of $\widetilde f^*\xi$ with $Q^*$. In this manner we have a decomposition
\[ (\widetilde f^*\xi,\widetilde f^*\omega) = (f^*\D \oplus Q\oplus Q^*,d\beta|_H + \Omega_\can), \]
where $\Omega_\can$ is the canonical symplectic form on $Q \oplus Q^*$. Do note that $d\beta|_H$ is symplectic on $f^*\D$, but its behaviour in the rest of $H$ we do not know. In particular, $f^*\D$ and $Q\oplus Q^*$ are not necessarily symplectically orthogonal with respect to $\widetilde f^*\omega$.

Consider now the subbundle $F(TS) \subseteq H$. We can decompose $F(TS) = [F(TS) \cap \D] \oplus Q$. Lifts of $F(TS)$ to $\widetilde f^*\xi$ are then given by the graph construction, applied to a pair of homomorphisms $\widetilde F = (\phi: F(TS) \cap \D \rightarrow Q^*,\psi: Q \rightarrow Q^*)$. Let us discuss what conditions one should impose on these maps so the isotropy condition holds.

First, the graph of $F(TS) \cap \D$ under any $\phi: F(TS) \cap \D \rightarrow Q^*$ is isotropic, since $F(TS) \cap \D \subseteq (\D,d\beta|_\D)$ is isotropic on $(\widetilde{f}^*\xi,\widetilde{f}^*\omega)$ (due to the preisotropic condition) and $\D$ and $Q^*$ are symplectically orthogonal. In contrast, the graph of $\psi$ is Lagrangian on $(Q\oplus Q^*,\Omega_\std)$ if and only if $\psi$ is symmetric, seen as a bundle map $Q \times Q \rightarrow \R$. The remaining constraint is that the graphs of $\phi$ and $\psi$ should be symplectically orthogonal so they jointly form an isotropic subspace.

Let us fix $\psi$. Having made this choice, we claim that the lift $\widetilde F$ is pinned uniquely. To see this, we show that there is a unique $\phi$ whose graph is symplectically orthogonal to the graph of $\psi$. We can simplify the linear algebra by noting that $\psi$ defines a symplectic change of basis $\widetilde \psi$ on $(Q \oplus Q^*,\Omega_\std)$ that acts as the identity on $Q^*$. This implies that, up to symplectic shear, we can assume that $\psi$ is the zero map. Assuming this, the isotropy condition demands that $\phi$ is given by the expression $v \mapsto -(\iota_v d\beta)|_Q$. Undoing the shear shows that $\phi$ depends smoothly on $\psi$.

The proof is now complete, since we have shown that the affine bundle $A \rightarrow S$ of isotropic lifts is modelled on the bundle of symmetric tensors $\psi: Q \times Q \rightarrow \R$. For formal embeddings, we observe that the space of lifts for a bundle monomorphism is also affine and thus contractible. This implies that $F_s$ lifts without homotopical choices and one can interpolate at $s=1$ to an isotropic monomorphism thanks to the argument given above for formal immersions.
\end{proof}

We will use the following crucial consequence, which follows from the parametric nature of the argument given in \cref{prop:formal_preisotropic_lift}:
\begin{corollary}
The Legendrian lifts of two formally isotopic prelegendrians are formally Legendrian isotopic.
\end{corollary}

\section{Co-real submanifolds in almost complex manifolds}
\label{sec:Review_coreal}

In this section we recall basic facts about linear complex algebra, almost-complex manifolds, and co-real submanifolds that we will use in subsequent sections. We need them to (1) construct examples of corank-$2$ fat manifolds, and to (2) build examples of compact prelegendrians in these fat manifolds. This section also serves to fix notation.

\subsection{Complex linear algebra}

Throughout this subsection we fix a real vector space $V$ of dimension $2n$ endowed with a \emph{complex structure} $J:V\to V$, thereby turning $V$ into a complex vector space. There is a dual complex structure $J^*$ on $V^*$ defined as $J^*(\alpha)=\alpha \circ J$.
\begin{definition}
Given $(V,J)$:
    \begin{itemize}
        \item A subspace $W\subseteq V$ is called $J$-invariant if $JW=W$.
        \item A $\C$-valued $\R$-linear map $\beta:V\to\C$ is called $J$-linear if $\beta\circ J = i\,\beta$.
        \item For a subspace $H\subseteq V$, its complex part is the $J$-invariant subspace $H^J\defi H\cap JH$.
        \item A subspace $S\subseteq V$ is called \emph{real} if $S^J=\{0\}$. 
        \item  A subspace $S\subseteq V$ is called \emph{co-real} if $S+JS=V$; equivalently, the annihilator $S^\perp\subseteq V^*$ is real. 
        \item The complexification of a linear map $\a:V\to\R$ is the $J$-linear map $\alpha^J\defi \alpha - i\,\alpha\circ J$. \qedhere
    \end{itemize}
\end{definition}

Note that if $S$ is real then $\dim S\le n$, while if $S$ is co-real then $\dim S\ge n$. The following observation will be used later to build co-real submanifolds.
\begin{lemma}\label{lem:corel+any_is_coreal}
    Let $S\subseteq V$ be co-real, and $H\subseteq V$ be any subspace. Then, $S+H$ is co-real.
\end{lemma}

We denote the real and complex Grassmannians of $V$, respectively, by:
\begin{gather*}
    \Gr_{k}(V)\defi \{H\subseteq V: H \text{ is a $k$-dimensional subspace}\},\\
    \Gr_{2k}(V,J)\defi \{W\subseteq V: W \text{ is $J$-invariant}\}\subseteq \Gr_{2k}(V).
\end{gather*}
Sometimes it is more convenient to work with the dual picture, called Grassmannian duality:
\begin{lemma}\label{lem:grasmm_dual}
    There are canonical bijections:
    \begin{align*}
        \Gr_k(V) &\rightarrow \Gr_{2n-k}(V^*)\\
        H&\mapsto H^\perp,
    \end{align*}
    and
    \begin{align*}
        \Gr_{2k}(V,J) &\rightarrow \Gr_{2n-2k}(V^*,J^*)\\
        W&\mapsto W^\perp.
    \end{align*}
\end{lemma}

Since we will mostly work with (real or complex) hyperplane Grassmannians, we will sometimes write $\Gr(V)\defi \Gr_{2n-1}(V)$ and $\Gr(V,J)=\Gr_{2n-2}(V,J)$, whenever there is no possible confusion.
\begin{definition}
The \emph{Hopf map} is defined as:
\begin{align*}
    \Hopf: \Gr_{2n-1}(V)&\rightarrow \Gr_{2n-2}(V,J)\\
    H &\mapsto H^J.
\end{align*}
Equivalently, $\Hopf(H)$ is the maximal $J$-invariant subspace contained in $H$.  
\end{definition}
This is well defined since every real hyperplane is co-real. Applying Grassmannian duality and identifying with (real/complex) projectivisation, we obtain the dual Hopf map:
\begin{align*}
    \Hopf^*: \Gr_{1}(V^*)\simeq\P(V^*) &\rightarrow \P_\C(\Hom_\C(V,\C))\simeq\Gr_{2}(V^*,J^*)\\
    [\a]_\R &\mapsto [\a^J]_\C
\end{align*}

In the case $V=\R^{2n}\cong\C^n$ (i.e. $J=J_{\mathrm{std}}$), the above dual Hopf map yields the standard Hopf fibration:
\begin{align*}
    \pi_{\mathrm{Hopf}}:\RP^{2n-1}&\rightarrow \CP^{n-1}\\
    [p_1:\cdots:p_{2n}]_\R &\mapsto [p_1+i\,p_2:\cdots:p_{2n-1}+i\,p_{2n}]_\C
\end{align*}

With the notation above we recall the following linear-algebra lemma, that will be used later.
\begin{lemma}\label{lem:dual_complex_basis}
    Let $W\in\Gr_{2n-2}(V)$ be a subspace of codimension $2$, and let $\alpha_1,\alpha_2\in V^*$ be linear maps with $W=\ker\alpha_1\cap\ker\alpha_2$. Then $W$ is $J$-invariant if and only if $\alpha_1^J$ and $\alpha_2^J$ are $\C$-linearly dependent in $\Hom_\C(V,\C)$.
\end{lemma}

\begin{proof}
    By Grassmannian duality (\cref{lem:grasmm_dual}), $W$ is $J$-invariant if and only if the plane $\operatorname{span}_\R\{\alpha_1,\alpha_2\}\subseteq V^*$ is $J^*$-invariant (i.e. a complex line), which is equivalent to $\alpha_1^J$ and $\alpha_2^J$ being $\C$-linearly dependent.
\end{proof}

\subsection{Co-real submanifolds}

Let $(X,J)$ be an almost complex manifold of real dimension $2n$.
\begin{definition}
    An immersion $f:S\to (X,J)$ is called \emph{co-real} if for every $q\in S$ the subspace $\d f_q(T_qS)\subseteq T_{f(q)}X$ is co-real (with respect to $J_{f(q)}$).

    A submanifold $S\subseteq (X,J)$ is called \emph{co-real} if $TS\subseteq TX$ is a co-real subbundle.
\end{definition}

The class of co-real immersions/embeddings fits into the framework of \emph{directed} immersions/embeddings \cite{gromovPartialDifferentialRelations1986,cieliebakIntroductionPrinciple2024a}, and the corresponding open subset of the Grassmannian bundle is ample. Consequently, Gromov’s convex integration technique yields the following $h$-principle:
\begin{theorem}[Gromov]
    Co-real embeddings satisfy the $C^0$-dense, parametric $h$-principle, relative both in domain and parameter.
\end{theorem}

Next, we will state several results about co-real submanifolds. Some of them will seem a bit unmotivated or technical, but will become relevant in our construction of prelegendrian stabilizations and compact prelegendrians. The first one is a (formal) topological obstruction to the existence of co-real embeddings in $\C^n$:
\begin{proposition}[\cite{Well:co-real}]
\label{lem:coreal_euler}
    If $N\subseteq\C^n$ is an oriented, co-real submanifold, then its Euler characteristic satisfies $\chi(N)=0$.
\end{proposition}
\begin{proof}
    Since $N$ is co-real in $\C^n$, there is a splitting
    \begin{equation*}
        T\C^n|_N \simeq (TN)^J \oplus \eta\ \oplus \nu_{N\backslash \C^n},
    \end{equation*}
    with $\nu_N$ the normal bundle, $\eta\subseteq TN$ and $J(\eta)=\nu_N$. In particular,
    \begin{equation*}
        TN \simeq (TN)^J\oplus \nu_N.
    \end{equation*}
    Hence $e(TN)=e((TN)^J)\cup e(\nu_N)$. Being in $\C^n$, we deduce \cite{milnorCharacteristicClasses1974} that the self-intersection $[N].[N]$ vanishes, which implies that $e(\nu_N)=0$ and therefore $e(TN) = 0$. It follows that $\chi(N)=\langle e(TN),[N]\rangle=0$.
\end{proof}

The following result will be fundamental in our construction of exotic prelegendrians (\cref{thm:InfinitelyPrelegendrianTorus}), since it will allow us to find a co-real torus to prelegendrian spin along. 
\begin{lemma}\label{lem:bundle_over_coreal_is_coreal}
    Let $(X,J)$ be an almost complex manifold such that $\pi:X\to N$ is a vector bundle over a compact manifold $N$. Assume that the zero section $0_N\subseteq X$ is co-real. Then, 
    \begin{itemize} 
    \item For every connection $\mathrm{Hor}\subseteq TX$ of $X$, then there exists a sufficiently small tubular neighbourhood $U$ of $0_N$ such that $\mathrm{Hor}_{|U}\subseteq TX_{|U}$ is a co-real subbundle.  
    \item Every submanifold $W\subseteq X$ fibered over $N$ is fibered isotopic to a co-real submanifold. 
    \end{itemize}
\end{lemma}
\begin{proof}
    The first part is a consequence of the compactness of $N$ and the fact that co-reality is an open condition, since $\mathrm{Hor}_p=T_p N$ for $p \in N\equiv 0_N$.
    
    For the second part, fix a connection $\mathrm{Hor}$ of $X$ that restricts to a connection of $W$, and consider the dilation $\delta_t:X\rightarrow X$, $t>0$, along the fibre directions. It follows from the previous part and \cref{lem:corel+any_is_coreal} that $\delta_t(W)$ is co-real for $t>0$ sufficiently small. 
\end{proof}

\subsubsection{Sections of the complex tangencies}

The following result will play a relevant role in proving that prelegendrian stabilization does not change the formal prelegendrian isotopy class. The $N$ appearing in the statement will later be the co-real submanifold along which we perform the stabilization and $\cF$ will be a prelegendrian front.
\begin{proposition}\label{prop:existence_of_rotate_cpx_tang_vf}
    Let $N\subseteq \cF\subseteq (X,J)$, where $\cF$ is an oriented hypersurface and $N$ is a co-real submanifold. Let $M\defi \overline{\Op_\cF(N)}\subseteq \cF$ be the closure of a sufficiently small tubular neighbourhood of $N$ in $\cF$. Then there is an inward pointing section $Y\in \Gamma(TM^J|_{\partial M})$.
    
    Suppose moreover that the normal bundle $\nu_N^\cF$ of $N$ in $\cF$ has a nowhere-vanishing section. Then, every nowhere-vanishing, inward-pointing section $Z\in \Gamma(TM^J|_{\partial M})$ extends to a nowhere-vanishing section $\tilde Z\in\Gamma(TM^J)$.
\end{proposition}
\begin{proof}
    Consider the bundle $M \rightarrow N$. The boundary $\partial M$ is a fibered submanifold and can be assumed to be co-real (\cref{lem:bundle_over_coreal_is_coreal}) by shrinking $M$. Applying \cref{lem:existence_of_nonzero_section_fiber_ver} to $S = T\partial M \subset TM|_{\partial M } \subset TX|_{\partial M}$ produces a section of $F^J/S^J$. This yields the claimed $Y$ by taking a representative, which can be assumed to be inward-pointing up to a change of sign.
    
    The second claim will follow once we show that there is a nowhere-vanishing section $s$ of $\Gamma(TM^J)$ that is nowhere outside-pointing. Indeed, that being the case, we can use a bump function close to $\partial M$ to interpolate between $s$ and $Z$, remaining nowhere-vanishing.

    Consider the bundle $M \rightarrow N$. Let $\mathrm{Hor} \subseteq TM$ be a connection on $M \rightarrow N$, extending a connection on $T\partial M \rightarrow N$. Since $N \subseteq X$ is co-real, we may assume that $\mathrm{Hor} \subseteq TX|_{M}$ is a co-real subbundle, possibly after shrinking $M$ (\cref{lem:bundle_over_coreal_is_coreal}). By hypothesis, $\nu_N^\cF$ has a nowhere vanishing section, so the vertical bundle $TM/\mathrm{Hor}$ does as well. We have thus a flag of vector bundles $S = \textrm{Hor} \subseteq F = TM \subseteq E = TX|_M$ over $M$ satisfying the hypotheses of \cref{lem:existence_of_nonzero_section_fiber_ver} below; it follows that there exists a non-vanishing section $s: M\to TM^J\cap\mathrm{Hor}$. It is tangent to $\partial M$ by construction, concluding the proof.
\end{proof}

We now prove the claimed auxiliary result.
\begin{lemma}\label{lem:existence_of_nonzero_section_fiber_ver}
    Let $S\subseteq F \subseteq E$ be a flag of vector bundles over a manifold $A$ such that
    \begin{itemize}
        \item $E\rightarrow A$ is a complex vector bundle, with fiberwise complex structure $J:E\to E$,
        \item $S\subseteq E$ is a co-real subbundle, and
        \item The normal bundle $F/S$ admits a nowhere zero section.
    \end{itemize}
    Then, both $F^J/S^J$ and $F^J\cap S$ admit non-vanishing sections.
\end{lemma}
\begin{proof}
Choose a complement $K$ of $S^J$ in $S$. Since $S$ is coreal, $J(K)$ is a complement of $S$ in $E$. We obtain thus an isomorphism:
\begin{equation*}
    E \simeq S^J \oplus K \oplus J(K).
\end{equation*}
Consider now $Q := F \cap J(K)$. It is a complement of $S$ in $F$ and thus isomorphic to $F/S$. It follows that both $Q$ and $J(Q)$ admit non-vanishing sections. The claim then follows by noting that $J(Q) = J(F) \cap K \subseteq S \subseteq F$, which implies that $Q \subset F^J$ and that $J(Q) \subseteq F^J \cap S$.
\end{proof}

\section{\texorpdfstring{Examples of corank $2$ Fat Distributions}{Examples of corank 2 Fat Distributions}} \label{sec:examples_fat}

Our philosophy is that corank-$2$ fat distributions should be viewed as a generalisation of holomorphic contact structures. We expect them to form a much larger class, ameanable to topological methods. However, as discussed in the introduction, to what extent they are genuinely more flexible is still a wide open problem.

In this section we focus on \emph{spaces of complex contact elements} over almost complex manifolds. This is the family of examples most relevant to our study of prelegendrians. We will show that their contactisations are canonically contactomorphic to the spaces of contact elements of the underlying smooth manifold.

\subsection{The space of complex contact elements}

We now describe a class of corank-$2$ fat distributions arising from the study of almost-complex Grassmannians, generalising the space of contact elements to the almost-complex setting. The study of such structures, from the perspective of holomorphic contact geometry, was carried out in \cite{ForstnericLarusson:Projectivised_cotangent}.

Let $(X,J)$ be an almost-complex manifold of real dimension $2n$. The \emph{almost-complex Grassmannian} of complex codimension one is the fibre bundle
\begin{equation*}
    \pi:\Gr_{2n-2}(X,J)\to (X,J),
\end{equation*}
whose fibre over $x\in X$ is
\begin{equation*}
    \Gr_{2n-2}(X,J)_x = \Gr_{2n-2}(T_xX,J_x).
\end{equation*}
For brevity we write $\Gr(X,J)$ for $\Gr_{2n-2}(X,J)$ and refer to it as the \emph{almost-complex Grassmannian} of $(X,J)$. Analogously to the construction of the space of contact elements in contact geometry, we have:
\begin{proposition} \label{prop:almostComplexGrass}
    Let $(X,J)$ be an almost complex manifold. The distribution $\D_\can\subseteq T\Gr(X,J)$ defined by:
    \begin{equation*}
        \D_\can(W) \defi \d\pi^{-1}(W), \quad W\in \Gr(X,J),
    \end{equation*}
   is fat. The distribution $\D_\can$ is called the \emph{canonical fat distribution}. The fat manifold $(\Gr(X,J),\D_\can)$ is referred to as the \emph{space of complex contact elements} over $(X,J)$.
\end{proposition}
The master thesis \cite{Nina:fat_dist} treats the case $\dim X=4$. The general case can be established with an analogous computation. Instead of pursuing this route, we just remark that this will follow from the fact that its contactisation is the space of contact elements (\cref{thm:contactisation_grasm}).

It is straightforward to check that the canonical fat distribution is invariant under any diffeomorphism preserving the almost-complex structure. More generally, consider $(X,J)$ an almost complex manifold and $g:X\to X$ a diffeomorphism. Then, $g$ induces a pushforward almost complex structure $g_*J$ on $X$, and the map $\Gr(g):\Gr(X,J)\to \Gr(X,g_*J)$ is a diffeomorphism that preserves $\D_\can$.

\begin{remark}
    Fat structures of corank at least $2$ have local invariants. From this perspective, the distributions we are considering are rather special. Namely, there is a submersion $\pi_\Gr:(\Gr(X,J),\D_\can)\to (X,J)$ whose fibres are horizontal submanifolds (i.e., tangent to $\D_\can$). It is at the moment unknown whether a general corank-$2$ fat distribution admits such a submersion locally.
    
    A certain local invariant of corank-$2$ fat distributions was introduced in \cite{BhowmickDatta:Horizontal_fat}: the degree. The distributions $\D_\can$ have always degree $2$. A germ of fat distribution of type $(8,10)$ and degree different from $2$ was constructed in \cite{BhowmickDatta:Horizontal_fat}. Related results appeared recently in \cite{MartinezAguinaga:Reeb}.

    We expect our constructions of prelegendrians to extend to arbitrary fat distributions of corank $2$, but this remains an open question.
\end{remark}

\subsubsection{Contactisation of the space of complex contact elements}

As we mentioned, the contactisation of the space of complex contact elements  $(\Gr(X,J),\D_\can)$ has a particularly nice and geometric form:
\begin{theorem}\label{thm:contactisation_grasm}
    Let $(X,J)$ be an almost complex manifold. Then, the space of complex contact elements $(\Gr(X,J),\D_\can)$ is a fat manifold. Moreover:
    \begin{itemize}
        \item Its projectivised contactisation $\cC_\P(\Gr(X,J),\D_\can)$ is canonically contactomorphic to the space of contact elements $(\P T^*X,\xi_\can)$.
        \item Its contactisation $\cC(\Gr(X,J),\D_\can)$ is canonically contactomorphic to the space of co-oriented contact elements $(\S T^*X,\xi_\can)$.
    \end{itemize}
    Moreover, the contactisation bundle projections are given by the following commutative diagram:
  \begin{equation*}
        \begin{tikzcd}[sep=small]
            {\RP^1} && {\S^1} \\
        	\\
        	{(\P T^*X ,\xi_\can)} && {(\S T^*X,\xi_\can)} \\
        	\\
        	& {(\Gr(X,J),\D_\can)} \\
        	& {(X,J)}
        	\arrow[hook, from=1-1, to=3-1]
        	\arrow[from=1-3, to=1-1]
        	\arrow[hook, from=1-3, to=3-3]
        	\arrow["{{\mathrm{Hopf}}}"', from=3-1, to=5-2]
        	\arrow[from=3-3, to=3-1]
        	\arrow["{{\mathrm{Hopf}}}", from=3-3, to=5-2]
        	\arrow["{\pi_\Gr}", from=5-2, to=6-2]
        \end{tikzcd}
    \end{equation*}
    where the diagonal maps are the fiberwise geometric Hopf maps described in \cref{sec:Review_coreal}.
\end{theorem}

\begin{proof}
    An element of $\Gr(X,J)$ is a pair $(x,W)$ with $W \subseteq T_xX$ a complex hyperplane. The distribution $(\D_\can)_{(x,W)}$ is the preimage of $W$ under the projection $\pi_\Gr:\Gr(X,J) \to (X,J)$. 

    Consider now the contactisation $\cC_\P(\Gr(X,J),\D_\can) = (\P\D_\can^\perp,\xi_{\P\D_\can^\perp})$. An element of $\P\D_\can^\perp$ is a pair $((x,W); H)$ with  $(x,W)\in \Gr(X,J)$ and $H\subseteq T_{(x,W)}\Gr(X,J)$ a real hyperplane containing $(\D_\can)_{(x,W)}$. It follows that $H$ is the preimage under $d_{(x,W)}\pi_\Gr$ of a real hyperplane $\tilde H \subseteq T_xX$ containing $W$. Dimensionality implies that $\tilde H\cap J\tilde H=W$. The assignment
    \begin{align*}
        \Phi:\P(\D_\can^\perp) &\longrightarrow \P T^*X\\
        \big( (x,W); H\big) & \mapsto (x,\tilde{H})
    \end{align*}
    is thus a diffeomorphism fibered over $\Gr(X,J)$. We obtain the following commuting square:
    \begin{equation*}
        \begin{tikzcd}[sep=scriptsize]
        	{\P\D_\can^\perp} &&& {\P T^*X} &&& {\big((x,W);H\big)} &&& {(x,\tilde{H})} \\
        	\\
        	{\Gr(X,J)} &&&&&& {(x,W)} \\
        	\\
        	{(X,J)} &&&&&& x
        	\arrow["\Phi", from=1-1, to=1-4]
        	\arrow["{\pi_\P}"', from=1-1, to=3-1]
        	\arrow["\rho", from=1-4, to=5-1]
        	\arrow[maps to, from=1-7, to=1-10]
        	\arrow[maps to, from=1-7, to=3-7]
        	\arrow[maps to, from=1-10, to=5-7]
        	\arrow["{\pi_\Gr}"', from=3-1, to=5-1]
        	\arrow[maps to, from=3-7, to=5-7]
        \end{tikzcd}
    \end{equation*}

    To see that $\Phi$ is a contactomorphism we differentiate the diagram. According to the definition of $\xi_{\P\D_\can^\perp}$, for any element $H\in \P\D_\can^\perp$, we have that
    \begin{equation*}
        v \in (\xi_{\P\D_\can^\perp})_H \iff \d_H\pi_\P(v)\in H \iff \d \pi_\Gr \circ \d_H\pi_\P(v) \in \tilde{H} \iff \d_{\tilde{H}} \rho\circ \d_H\Phi(v) \in \tilde H.
    \end{equation*}
    This implies that $\d_H\Phi (\xi_{\P\D_\can^\perp}) = \xi_{\P T^*X}$, by the definition of tautological contact distribution $\xi_{\P T^*X}$. This concludes the argument. The same reasoning applies to the co-oriented contactisation $\cC(\Gr(X,J),\D_\can)$.  
\end{proof}

\subsection{The standard fat structure on $\C^{2n+1}$} \label{subsec:Std_fat_Str}

The previous constructions can be specialised to define the standard fat distribution on $\C^{2n+1}$.

Equip $\R^{2n+2}\cong \C^{n+1}$ with an almost complex structure $J$. After choosing a complex framing, we trivialise the almost-complex Grassmannian as
\begin{equation*}
    \Gr(\R^{2n+2},J) \simeq \R^{2n+2}\times \CP^{n}.
\end{equation*}
Fix a complex projective hyperplane $H\simeq \CP^{n-1}$ in $\CP^{n}$, and identify $\CP^n\setminus H \simeq \C^n$. Thus we obtain an open subset
\begin{equation*}
    \R^{2n+2}\times \C^n\subseteq \R^{2n+2}\times \CP^n \simeq \Gr(\R^{2n+2},J).
\end{equation*}

The canonical fat structure $\D_\can$ on $\Gr(\R^{2n+2},J)$ restricts to a fat distribution $\D_{\std,J}$ on $\R^{2n+2}\times \C^n\simeq \R^{4n+2}$, called the \emph{standard fat structure twisted by $J$}. Different choices of almost complex structure, hyperplane, and framing produce homotopic fat distributions on $\R^{4n+2}$.
\begin{definition}
In the case of the standard complex structure $J\equiv J_\std$ on $\C^{n+1}$, the construction yields the \emph{standard fat structure} $\D_\std$ on $\C^{2n+1}$.
\end{definition}


\begin{corollary}\label{cor:contactsation_fat_std}
    The contactisation $\cC(\C^{2n+1},\D_\std)$ is contactomorphic to $(J^1(\R^{2n+1}\times \S^1),\xi_\std)$.
\end{corollary}
\begin{proof}
    By definition, $(\C^{2n+1},\D_\std) \subseteq \Gr(\C^{n+1},J_\std)$ is obtained from $\Gr(\C^{n+1},J_\std)$ by taking an affine chart and, by \cref{thm:contactisation_grasm}, the contactisation is
    \begin{equation*}
        \C^{n+1}\times \S^{2n+1}=\ \S T^*\C^{n+1} \simeq \cC(\Gr(\C^{n+1},J_\std)) \to \Gr(\C^{n+1},J_\std)=\C^{n+1}\times \CP^n.
    \end{equation*}
    with the canonical contact structure. Writing the complex coordinates on $\C^{n+1}$ as $(x,\ldots,x_n,z)$, and the induced homogeneous coordinates on $\CP^n$, we get
    \begin{equation*}
        \cC(\C^{2n+1},\D_\std) = \C^{n+1}\times \Hopf^{-1}(\{z\neq 0\})= \C^{n+1}\times S,
    \end{equation*}
    where $\Hopf:\S^{2n+1}\subseteq \C^{n+1}\to \CP^n$ is the standard Hopf map and $S\defi \S^{2n+1}\setminus\{z=0\}$.

    On the other hand, there is a contactomorphism
    \begin{align*}
        \S T^*\R^{2n+2} = \R^{2n+2}\times \S^{2n+1} &\to J^1(\S^{2n+1}) = T^*\S^{2n+1} \times \R\\
        (\mathbf{p},\mathbf{q})&\mapsto (\mathbf{q},\mathbf{p}-\langle \mathbf{q},\mathbf{p} \rangle\mathbf{q} , \langle \mathbf{q},\mathbf{p} \rangle),
    \end{align*}
    that induces a contactomorphism
    \begin{equation*}
        (\C^{n+1}\times S,\xi_\can) \simeq (J^1(S),\xi_\std).
    \end{equation*}
    Since $S\simeq \R^{2n}\times \S^1$, the result follows. 
\end{proof}

\subsection{Affine charts in the Grassmannian} \label{subsec:Affine_chart}

    It is sometimes useful to work inside an affine Grassmannian chart, as in the construction of $(\C^{2n+1},\D_\std)$. In general, we cannot do this globally (there may not exist a global section of $\Gr(X,J)\to X$), so we restrict to the product case $X=Q\times\R$.
    
\subsubsection{Contact case}
    Let $X=Q\times\R$ with coordinate $z$ on the second factor and \emph{vertical vector field} $\partial_z$. Define the horizontal Grassmannian of real hyperplanes by
    \begin{equation*}
        \P(T_\mathrm{hor}^* X) \defi \{ H\subseteq T_pX : \partial_z \notin H\}\subseteq \P T^*X.
    \end{equation*}
    Being an open submanifold (a fiberwise affine chart) of $\P T^*X$, it inherits the canonical contact structure $\xi_\can$.

    Similarly, for co-oriented hyperplanes we write $\S(T^*_{\mathrm{hor}}X)\subseteq \S T^*X$.  In this case, the fibres are not connected, and there is a canonical disjoint union decomposition
    \begin{equation*}
        \S(T^*_{\mathrm{hor}}X) = \S(T^*_{\mathrm{hor}}X)^+ \sqcup \S(T^*_{\mathrm{hor}}X)^-
    \end{equation*}
    into the so-called the spaces of \emph{positive} and \emph{negative} horizontal (co-oriented) contact elements. We note:
    \begin{lemma}\label{lem:HorizontalContactElements}
        There exists a canonical contactomorphism between $(\S T^*_\mathrm{hor}(Q\times \R)^\pm,\xi_\can)$ and $(J^1Q,\xi_\can)$.
    \end{lemma}


\subsubsection{Almost complex case}
    
    Assume that $X=Q\times\R$ is equipped with an almost complex structure $J$.  The \emph{horizontal complex Grassmannian} is
    \begin{equation*}
        \Gr_\mathrm{hor}(X,J)\defi \{ W\subseteq T_x X: \partial_z\notin W \text{ and $W$ is $J$-invariant} \}\subseteq \Gr(X,J),
    \end{equation*}
    i.e. the fiberwise open locus of complex hyperplanes avoiding the vertical line. Equivalently, $W$ is horizontal if and only if it does not contain the complex line $\mathrm{span}_\R\langle\partial_z,J\partial_z\rangle$. Thus, when $\dim_\R(Q\times\R)=2n+2$, the fibre of $\Gr_{\mathrm{hor}}(X,J)\to X$ is diffeomorphic to $\C^n$. The fiberwise Hopf map is well defined on both horizontal Grassmannians:
    \begin{equation*}
        \mathrm{Hopf}:\S(T^*_{\mathrm{hor}}X)^\pm\longrightarrow \Gr_{\mathrm{hor}}(X,J).
    \end{equation*}
    In particular, there a natural inclusion
    \begin{equation}\label{eq:InclusionHorizontal}
        (\S(T^*_{\mathrm{hor}}X)^\pm,\xi_\can)\hookrightarrow \cC(\Gr_\mathrm{hor}(X,J),\D_\can)
    \end{equation}

    In the model case $Q=\R^{2n+1}$ and $X=\R^{2n+1}\times \R$ with $J_\std$, the horizontal complex Grassmannian reproduces the standard structure:
    \begin{equation*}
        \big(\Gr_{\mathrm{hor}}(\C^{n+1},J_{\mathrm{std}}),\D_\can\big)\cong\big(\C^{2n+1},\D_\std\big),
    \end{equation*}
    and we obtain the inclusion
    \begin{equation}\label{eq:InclusionHorizontalCn}
        \big(J^1\R^{2n+1},\xi_\can\big)\hookrightarrow \cC\big(\C^{2n+1},\D_\std\big).
    \end{equation}

\subsubsection{Our convention on prelegendrian co-orientations} \label{convention:prelegendrianCoorientation}

    Assume that $\Lambda\subseteq \Gr(X,J)$ is a prelegendrian, equipped with a prelegendrian co-orientation (see \cref{subsec:preisotropic_submanifold}), whose Legendrian lift lies in $(\S (T^*_{\mathrm{hor}}X)^\pm, \xi_\can)$. Then, $\Lambda$ is automatically contained in $\Gr_{\mathrm{hor}}(X,J)$. 
    
    Observe that the opposite co-orientation yields a lift into $(\S (T^*_{\mathrm{hor}}X)^\mp, \xi_\can)$. We henceforth we assume that all prelegendrians in $\Gr_{\mathrm{hor}}(X,J)$ are co-oriented so that their lifts lie in the positive part $(\S T^*X,\xi_\can)^+\cong (J^1Q,\xi_\can)$.

\subsection{Holomorphic contact manifolds}

A closely related class of corank $2$ fat distributions is:
\begin{definition}
    Let $M$ be a complex manifold of complex dimension $2n+1$. A holomorphic vector subbundle $\D\subseteq TM$ of complex corank $1$ is called a \emph{holomorphic contact structure} if, for every point $p\in M$, there exists an open \nbh $U\subseteq M$ such that $\D|_U = \ker \Theta$ is defined by a holomorphic $1$-form $\Theta$ satisfying:
    \begin{equation*}
        \d\Theta\wedge(\Theta)^n\neq 0.
    \end{equation*}
    The pair $(M,\D)$ is called a \emph{holomorphic contact manifold}.
\end{definition}

A fundamental example is $\C^{2n+1}$ equipped with the \emph{standard holomorphic contact form}, expressed in complex coordinates $(x_1, \dots, x_n, y_1, \dots, y_n, z)$ as:
\begin{equation*}
    \Theta_\std = \d z + \sum_{i=1}^{n}y_i\d x_i,
\end{equation*}
whose associated holomorphic contact structure is $\ker (\Theta_\std)$. As suggested by the terminology:
\begin{proposition}
    The standard fat structure $\D_\std$ on $\C^{2n+1}$ coincides with standard holomorphic contact structure $\ker(\Theta_\std)$.
\end{proposition}

\begin{proof}
    Write the complex coordinates on $\C^{n+1}$ as $(x_1,\ldots,x_n,z)$. The canonical fat distribution $\D_\can$ on the complex Grassmannian $\Gr(\C^{n+1})=\C^{n+1}\times \CP^n$ is given by
    \begin{equation*}
        \D_\can\big( (x_1,\cdot,x_n,z);[w_1:\cdots:w_{n+1}] \big) = \ker (w_1\d x_1+\cdots+ w_{n+1}\d z)\big).
    \end{equation*}
    Choose the hyperplane at infinity $H\defi \{w_{n+1}=0 \}\subseteq \CP^n$ and affine coordinates $y_j\defi \frac{w_j}{w_{n+1}}$. Then, the standard fat distribution is given by
    \begin{equation*}
        \D_\std \big( (x_1,\ldots,x_n,z);(y_1,\ldots,y_n) \big) = \ker (y_1\d x_1+\cdots+y_n\d x_n + \d z),
    \end{equation*}
    which is precisely  $\ker \Theta_\std$, as required.
\end{proof}

This, together with the holomorphic version of Darboux's theorem \cite[Theorem~A.2]{alarconHolomorphicLegendrianCurves2017}, implies:

\begin{corollary} \label{lem:hol_contact_is_fat}
    Let $(M,\D)$ be a holomorphic contact manifold. Regard $M$ as a smooth manifold and $\D$ as a real distribution of corank $2$. Then $\D$ is a fat distribution of corank~$2$ on $M$.
\end{corollary}

\begin{remark}
One can also prove \cref{lem:hol_contact_is_fat} using the algebraic characterisation of corank-$2$ fat distributions; see for instance \cite{BhowmickDatta:Horizontal_fat,Nina:fat_dist}.
\end{remark}

\subsection{The standard fat structure on $\CP^{2n+1}$} \label{subsec:complex_projective}

The odd dimensional complex projective spaces are endowed with canonical holomorphic contact structures, as we now explain. The corresponding fat distributions will be relevant for us.

First, recall that $\C^{2n+2}\setminus \{0\} \to \CP^{2n+1}$ is a $\C^*$-bundle. Its associated line bundle, the \emph{tautological bundle}, is usually denoted by $\O(-1)$. Write $z_0,\ldots,z_{2n+1}$ for the coordinates in $\C^{2n+2}$.
\begin{lemma} \label{lem:standardHolomorphicCP2n1}
Consider the holomorphic $1$-form $\widetilde\Theta$ on $\C^{2n+2}$
\begin{equation*}
    \widetilde\Theta = \sum_{j=0}^{n} (z_{2j}\d z_{2j+1} -z_{2j+1}\d z_{2j}).
\end{equation*}
It is a primitive of the standard holomorphic symplectic form. It satisfies the following properties:
\begin{itemize}
    \item[(a) ] It annihilates the radial vector field $\sum_i z_i\partial_{z_i}$.
    \item[(b) ] It is homogeneous of degree $2$ under the scaling action.
\end{itemize}
Consequently, it descends to a holomorphic $1$-form $\Theta_\std$ on $\CP^{2n+1}$ with values in $\O(2)$.
\end{lemma}
\begin{proof}
The first three claims are immediate computations. For the last claim, we use (a) to observe that $\widetilde\Theta$ annihilates the vertical bundle $\textrm{Vert}$ of $\pi: \O(-1) \to  \CP^{2n+1}$. As such, at each $z \in \O(-1)$, we can regard $\widetilde\Theta_z$ as a well-defined form on $T_z\O(-1)/\textrm{Vert} \simeq T_{\pi(z)}\CP^{2n+1}$.

Suppose that $v$ and $z$ are locally defined holomorphic sections of $T\CP^{2n+1}$ and $\O(-1)$, respectively. Then $\widetilde\Theta_z(v)$ is a locally defined holomorphic function on $\CP^{2n+1}$. It is linear on $v$, but not on $z$. Indeed, according to (b), applying the scaling action $z \mapsto \lambda z$ scales $\widetilde\Theta_z(v)$ by $\lambda^2$. That is, $\widetilde\Theta$ can be regarded as a holomorphic $1$-form $\Theta_\std$ on $\CP^{2n+1}$ with values in $\O(2)$, the line bundle of fibrewise degree-$2$ homogeneous polynomials on $\O(-1)$.
\end{proof}

\begin{definition} \label{def:standardHolomorphicCP2n1}
Let $\Theta_\std$ be the holomorphic $1$-form introduced in \cref{lem:standardHolomorphicCP2n1}. The \emph{standard holomorphic contact structure} $\D_{\std}$ on $\CP^{2n+1}$ is the hyperplane distribution $\ker(\Theta_\std)$. We also refer to it as the \emph{standard fat structure} on $\CP^{2n+1}$.
\end{definition}

\begin{remark}
As noted in \cite[Lemma 2.1]{AlarconForstnericLarusson}, the restriction of $\D_{\std}$ to affine charts of $\CP^{2n+1}$ is holomorphically contactomorphic to the standard $(\C^{2n+1}, \D_{\std})$.
\end{remark}

Let $\S\O(1)$ be the $\S^1$-bundle associated to $\O(1)$ and note that $\S\O(1) \simeq \S^{4n+3}$. Then:
\begin{proposition}\label{prop:contactisation_CP^2n+1}
$\cC(\CP^{2n+1},\D_\std)$ is contactomorphic to $(\RP^{4n+3},\xi_\std)$. $\cC_\P(\CP^{2n+1},\D_\std)$ is contactomorphic to the lens space $(L(4,1,\ldots,1),\xi_\std)$.
\end{proposition}
\begin{proof}
The co-oriented contactisation covers the unoriented one $2$-to-$1$. As such, the second claim follows from the first. We focus on the co-oriented case.

Given a locally-defined section $z$ of $\O(-2)$, we plug it to yield a complex-valued form $\Theta_\std(z)$; by construction this is complex linear in $z$. Then, we take its real part $\Re((\Theta_\std)(z))$. Complex linearity implies that the imaginary part is simply $\Im ((\Theta_\std)(z)) = -\Re((\Theta_\std)(iz))$. It follows that the space of real $1$-forms annihilating $\ker(\Theta_\std)$ is spanned by $\Re((\Theta_\std)(z))$, as we range through the complex multiples of $z$.

Let us denote $\p: \S\O(-2) \rightarrow \CP^{2n+1}$ for the projection. The reasoning above implies that $\S\O(-2) \simeq \RP^{4n+3}$, endowed with the real $1$-form $\alpha_z = \Re((\Theta_\std)_{\p(z)}(z) \circ d_z\p)$, is contactomorphic to the co-oriented contactisation $\cC(\CP^{2n+1},\D_\std)$. We claim that $\ker(\alpha)$ is contactomorphic to the standard structure $\xi_\std$ in projective space. To see this, it is enough to show that its double cover is the standard contact sphere.

Observe that the real part of $\widetilde\Theta$ is the standard real Liouville form, so it restricts to the standard contact structure in the sphere. Fix a point $w \in \S^{4n+3} = \S\O(-1)$ and a vector $v \in T_w\S^{4n+3}$. Let $w^2 \in \S\O(-2)$ and $\tilde v \in T_{w^2}\RP^{4n+3}$ be their images under the two-fold covering map. Then:
\[ \alpha_{w^2}(\tilde v) = \Re((\Theta_\std)_{\p(w^2)}(w^2) \circ d_{w^2}\p)(\tilde v) =  \Re((\Theta_\std)_{\pi(w)}(w^2) \circ d_w\pi)(v) = \Re(\widetilde\Theta_w)(v), \]
according to the definition of $\Theta_\std$. This proves the claim.
\end{proof}

\begin{remark}
    A different proof of \cref{prop:contactisation_CP^2n+1}, without making use of holomorphic geometry, can be given by explicitly constructing a $2$-to-$1$ covering from the standard contact sphere $\S^{4n+3}$ to $\cC(\CP^{2n+1},\D_\std)$. Since the argument may be of independent interest, we outline the main steps here. 
    
    Consider the unit quaternions $i,j,k$ in $\C^{2n+2}=\H^{n+1}$. The argument is based on the following observations, whose proofs are left to the reader:
    \begin{itemize}
        \item [(i)] The differential of the Hopf map $\pi:\S^{4n+3}\rightarrow \CP^{2n+1}$, restricts to a fibrewise isomorphism of bundles $\C\langle jz \rangle  \to T\CP^{2n+1}/\D_\std$.

        \item [(ii)] Let \[ e^{i\theta}:\S^{4n+3}\rightarrow \S^{4n+3},z \mapsto e^{i\theta}(z) \defi \cos(\theta) z + \sin(\theta) iz, \]
    with $\theta\in[0,2\pi]$ be the Hopf flow. Then
    \begin{equation*}
        \d_{e^{i\theta}(z)}\pi (ke^{i\theta}(z)) = \d_z\pi\big(\cos(2\theta) kz + \sin(2\theta) jz\big).
    \end{equation*}

        \item [(iii)] Consider the contact structure $\xi_{\std,k}=T\S^{4n+3}\cap k T\S^{4n+3}$, given by the complex $k$-tangencies. Using (ii) it can be shown that the map \begin{align*}
    \Phi: (\S^{4n+3},\xi_{\std,k}) &\to (\S(\D_\std^\perp),\xi_\can) = \cC(\CP^{2n+1},\D_\std)\\
    z&\mapsto \big(\pi(z),\d_z\pi(kz) \big)
\end{align*}
is a contact double cover. Here $\d_z\pi(kz)$ is regarded as an element of the cotangent bundle by using the Fubini-Study metric (i.e. the metric induced by the Hopf map).
   \item [(iv)] Since $(\S^{4n+3},\xi_\std)\cong (\S^{4n+3},\xi_{\std,k})$ by a linear contactomorphism, this concludes the argument. \qedhere
    \end{itemize}
\end{remark}

\section{Legendrian Fronts}\label{sec:Leg_front}

In view of \cref{thm:contactisation_grasm}, contactisations of fat complex Grassmannians are spaces of contact elements. The latter are examples of \emph{Legendrian fibrations}, which admit canonical front projections that can be used to study the Legendrians therein. In this section we provide an overview of these concepts. Crucially, we introduce a subclass of Legendrian fronts, which we call \emph{simple}. This class will be specially relevant in our study of \emph{prelegendrian fronts} in \cref{sec:Preleg_front}. The primary references are \cite{Arnold:Dynamical_systems_SG}, \cite{arnoldSingularitiesDifferentiableMaps2012}, and \cite{Arnold:Singularities_fronts}.

\subsection{Legendrian fibrations}

We begin by recalling the notion of Legendrian fibrations and front projections.
\begin{definition}
Let $(Y^{2n+1},\xi)$ be a $(2n+1)$-dimensional contact manifold. A fibre bundle
\begin{equation*}
    \pi_\F:(Y^{2n+1},\xi)\to B^{n+1}
\end{equation*}
over a manifold $B$ is called a \emph{Legendrian fibration} if all its fibres are Legendrian submanifolds of $Y$, i.e. if $\ker (d\pi_\F)\subseteq \xi$. The bundle map $\pi_\F$ is referred as a \emph{front projection}. 
\end{definition}

\begin{example}
The following are examples of Legendrian fibrations:
    \begin{itemize}
        \item $\pi_\F:(J^1Q,\xi_\std)\to J^0Q$ with $Q$ a smooth manifold.
        \item The concrete case $\pi_\F:(J^1\R^n,\xi_\std)\to J^0\R^n$ with $Q=\R^n$ is the \emph{standard front projection}. Locally, every Legendrian fibration is contactomorphic to this one; see for instance {\cite[p.~72]{Arnold:Dynamical_systems_SG}}.
        \item The space of contact elements over a manifold $X^n$ is a $\RP^{n-1}$-Legendrian fibration $\pi_\P:(\P T^*X,\xi_\can)\rightarrow X$. Similarly, the space of co-oriented contact elements is an $\S^{n-1}$-Legendrian fibration $\pi:(\S T^* X, \xi_\can)\rightarrow X$. \qedhere
    \end{itemize}
\end{example}

\begin{definition}
    A \emph{front map} is a diagram consisting of an embedding of a Legendrian submanifold $\iota:L^n\hookrightarrow (Y^{2n+1},\xi)$ and a Legendrian fibration $\pi_\F:(Y^{2n+1},\xi)\to B^{n+1}$:
    \begin{equation*}
        f: L^n\xhookrightarrow{\iota} (Y^{2n+1},\xi) \xrightarrow{\pi_\F}B^{n+1}.
    \end{equation*}
    By abuse of notation, we often refer to $f$ itself as the front map. The image $\mathcal{F}\coloneqq f(L)\subseteq B^{n+1}$ is called the \emph{wave front} (or simply the \emph{front}) of $L$ or $f$.
\end{definition}
Front maps $f$ are smooth maps from an $n$-dimensional manifold to an $(n+1)$-dimensional ambient manifold. They form a distinguished class of smooth maps with some special singularities. Roughly speaking, at each point $p\in L$, whether regular or singular, $f$ still has a well-defined tangent map $\iota$ of full rank. These tangent spaces record, and thus recover, the original Legendrian submanifold.

Concretely, in the local case of the standard front projection $\pi_\F:(\R^{2n+1},\xi_\std)\to\R^{n+1}$, one can use the front $\cF\subseteq \R^{n+1}$ to recover the $\mathbf{y}$-coordinates by taking\footnote{Here we identify a non-vertical oriented hyperplane $H\subseteq T\R^{n+1}$ with the unique $1$-form $\d z-\mathbf{y}\d \mathbf{x}$ on $\R^{n+1}$ that annihilates it.} the tangencies of $T\cF$, at the very least in the smooth points of $\cF$.

\subsection{Front singularities and simple fronts}

Let $f:L^n\to B^{n+1}$ be a front map. We now study the singularities of the map $f$ at a point $p\in L$. By singularity, we mean an equivalence class of front map germs. Since any front map can be locally described by the standard front projection, we always present such germs as $(\R^{n},0)\to(\R^{n+1},0)$. 

The possible singularities of a germ of a front are studied by means of generating functions or generating families of hypersurfaces \cite{Arnold:Singularities_fronts}. Although in the lower dimensional cases ($\dim L<6$), the possible front singularities are stable and completely classified  \cite{Arnold:Dynamical_systems_SG,arnoldSingularitiesDifferentiableMaps2012}, in full generality front singularities are intractable. In this article, we will focus on very \emph{simple} singularities:
\begin{definition} \label{def:cusp}
    Let $f:L\xhookrightarrow{\iota}(Y,\xi)\xrightarrow{\pi_\F}B^{n+1}$ be the front map of a Legendrian submanifold $L$, and $\cF\defi f(L)\subseteq B^{n+1}$ the corresponding front. We say that $f$ has a \emph{cusp singularity} or $A_2$-singularity at $q\in L$, and that $p\defi f(q)\in \cF$ is a \emph{cusp point} of the front, if the germ of the front map $f|_{\Op(q)}:\Op(q) \to B^{n+1}$ is equivalent to 
    \begin{align*}
        \hat{f}:\R^{n-1}\times \R&\to \R^{n}(\mathbf{x})\times \R(z)\\
        (q_1,\ldots,q_{n-1},t)&\mapsto \Big((q_1,\ldots,q_{n-1},-3t^2); 2t^3 \big) \qedhere 
    \end{align*} 
\end{definition}

The class of (non-generic) fronts that we will work with is defined as follows:
\begin{definition}\label{def:simple_front}
    A front $\cF\subseteq B^{n+1}$ is said to be \emph{simple} if there exists a stratification
    \begin{equation*}
        \cF = \Sigma_{A_1}(\cF)\sqcup \Sigma_{A_1A_1}(\cF)\sqcup \Sigma_{A_2}(\cF)
    \end{equation*}
    such that
    \begin{itemize}
        \item for all $p\in \Sigma_{A_1}(\cF)$, $p$ is a smooth point of $\cF$, i.e., $\cF$ is a smooth hypersurface near $p$;
        \item for $p\in \Sigma_{A_1A_1}(\cF)$, $\cF$ has a transversal self-intersection at $p$;
        \item for $p\in \Sigma_{A_2}(\cF)$, $p$ is a cusp point of $\cF$.
    \end{itemize}

    In this case, the sets $\Sigma_{A_1}(\cF)$, $\Sigma_{A_1A_1}(\cF)$ and $\Sigma_{A_2}(\cF)$ are called \emph{smooth strata}, \emph{self-intersection strata} and \emph{cusp strata} of $\cF$, respectively. We define the subset of singular points by
    \begin{equation*}
        \Sigma(\cF) \defi \Sigma_{A_1A_1}(\cF)\sqcup \Sigma_{A_2}(\cF). \qedhere
    \end{equation*}
\end{definition}

The following follows from standard transversality arguments:
\begin{lemma}
    Let $\cF\subseteq B^{n+1}$ be a generic simple front, then the subset of singularities $\Sigma(\cF)\subseteq B$ is a codimension $2$ submanifold in $B^{n+1}$.
\end{lemma}

\section{Legendrian stabilization and loose Legendrians}\label{sec:Leg_stablization}

In this section we review the notions of Legendrian stabilization and loose Legendrians in higher-dimensional contact manifolds, as introduced in \cite{Eliashberg:Stein,EES:NonIsotopicLegendrians,Murphy:Loose}. All the results presented here are well known, as are their proofs. We do this in order to clarify certain technical constructions that are important to the purposes of this article.

\subsection{\texorpdfstring{$N$-Pushing}{N-Pushing}}\label{subsec:N-pushing}

We introduce an operation on fronts, called \emph{$N$-pushing}, which alters the Legendrian isotopy class of a submanifold while preserving its formal Legendrian class, provided certain conditions on $N$ are satisfied. This operation is useful to show that two Legendrians belong to the same formal class.
    
This construction first appeared in \cite{Eliashberg:Stein, EES:NonIsotopicLegendrians}, where it was referred to as a ``stabilization''. However, the particular construction from \cite{EES:NonIsotopicLegendrians} does not, in general, produce a loose Legendrian submanifold \cite{Murphy:Loose}. For this reason, we adopt the term \emph{pushing}, and reserve the term \emph{stabilization} for the operation that produces a loose Legendrian.

Let $Q$ be a manifold. Given a function $f\in C^\infty(Q)$, we denote by $j^1f\subseteq J^1Q=T^*Q\times \R$ the Legendrian submanifold of $(J^1Q,\xi_\std)$ defined as the $1$-jet of $f$. Given $c\in \R$, we write $c_Q: Q\to \R$ for the constant function with value $c$. Consider the disconnected Legendrian $L_0=j^1 0_Q\sqcup j^1 1_Q$ and a closed submanifold $N\subseteq Q$ together with a smooth function $f:Q \rightarrow [0,2]$ such that 
\begin{itemize}
    \item [(i)] $f_{|Q \backslash \Op(N)}\equiv 0$, 
    \item [(ii)] $f(p)>1$ for $p\in N$,
    \item [(iii)] $f(p)>0$ for $p\in \Op(N)$, and
    \item [(iv)] The critical points of $f$ in $\Op(N)$ have critical values greater than $1$.
\end{itemize}
        
The \emph{$N$-pushing}  of $L_0$ is the disconnected Legendrian
\begin{equation*}
    L_N =j^1 f \sqcup j^1 1_Q
\end{equation*}
 
\begin{definition}
    Let $(Y^{2n+1},\xi)$ be a contact manifold of dimension $2n+1 \geq 5$ and let $L\subseteq Y$ be a Legendrian submanifold. Suppose we are given an open codimension $0$ submanifold $U \subseteq Y$ and an identification
    \begin{equation} \label{eq:modelForPushing}
        (U, L\cap U,\xi) \simeq (J^1 Q, L_0,\xi_\std)
    \end{equation}
     for some smooth manifold $Q$, as well as a closed submanifold $N \subset Q$.
     
     The \emph{$N$-pushing} of the Legendrian $L$ is the Legendrian $\mathrm{p}_N(L)$ obtained by replacing $(U, L\cap U,\xi)$ by  $(J^1 Q, L_N,\xi_\std)$.
\end{definition}

Do note that the trivialisation appearing in Equation \ref{eq:modelForPushing} is part of the data involved in $N$-pushing, but we leave it implicit. The following result is well known. We will adapt its proof later to the prelegendrian setting.
\begin{proposition}[Ekholm-Etnyre-Sullivan, Murphy] \label{prop:FormalClassNPushing}
    Let $L\subseteq (Y,\xi)$ be a Legendrian submanifold, and suppose that the Euler characteristic $\chi(N)=0$. Then, the $N$-pushing $\mathrm{p}_N(L)$ is formally Legendrian isotopic to $L$ in $(Y,\xi)$.
\end{proposition}
\begin{proof}
    It suffices to check that the Legendrians $L_0=j^1 0_Q \sqcup j^1 1_Q$ and $L_N=j^1 f \sqcup j^1 1_Q$ are formally Legendrian isotopic in $J^1Q$ through a compactly supported formal isotopy. We parametrise explicitly:
    \begin{align*}
        j^1 0_Q : \mathbf{x} & \mapsto (\mathbf{x},0,0),\\
        j^1 1_Q : \mathbf{x} & \mapsto (\mathbf{x},1,0),\\
        j^1 f : \mathbf{x} & \mapsto (\mathbf{x},f(\mathbf{x}),\d_\mathbf{x}f).
    \end{align*}
    We must construct a formal isotopy of Legendrian embeddings between $j^1 0_Q$ and $j^1f$ that does not meet the Legendrian $j^1 1_Q$, and is supported in a small neighbourhood of the support of $f$. This isotopy will be defined by a homotopy of formal sections $g_t$ of $J^1Q$. The idea is to consider the obvious homotopy between $j^1 0_Q$ and $j^1 f$ induced by the linear interpolation between the functions $0_Q$ and $f$ and perturb it. Do note that the obvious homotopy does not work precisely because it intersects $j^1 1_Q$ along the critical points of $f$.

    Fix $\varepsilon>0$ sufficiently small so that $\d f\neq 0$ on $f^{-1}(1-\varepsilon,1+\varepsilon)$. Since $\chi(\overline{\Op_Q(N)}) = \chi (N) =0$, there exists a $1$-form $\beta$ on $Q$ such that
    \begin{itemize}
        \item $\beta$ is non-zero on $f^{-1}(1-\varepsilon,2]\cong \Op_Q(N)$.
        \item $\beta = \d f$ on $f^{-1}(-\infty,1+\varepsilon)$.
    \end{itemize}
    That is, $\beta$ is a nowhere-vanishing extension of $df|_{f^{-1}(1)}$ to $N$. We define a homotopy of embeddings $g_t: Q \to J^1Q$ in three stages:
    \begin{itemize}
        \item For $t\in[0,1]$, we set
        \begin{equation*}
            g_t(\mathbf{x}) := (\mathbf{x},0, t\beta(\mathbf{x})).
        \end{equation*}
        No intersection with $j^11_Q$ can occur since the underlying $0$-jets are disjoint.
        \item For $t\in[1,2]$, we consider:
        \begin{equation*}
            g_t(\mathbf{x}) := (\mathbf{x}, (t-1)f(\mathbf{x}), \beta(\mathbf{x})).
        \end{equation*}
        In the region $f^{-1}(1-\varepsilon,2]$, $\beta$ is non-zero, so $g_t$ and $j^11_Q$ differ in their derivative term. In the region $f^{-1}(-\infty,1-\epsilon)$, $g_t = j^1f$ and $f$ has no critical points, so it is also distinct from $j^11_Q$.
        \item For $t\in[2,3]$, define
        \begin{equation*}
            g_t(\mathbf{x}) = (\mathbf{x}, f(\mathbf{x}), (t-2)\d f + (3-t)\beta(\mathbf{x})). 
        \end{equation*}
        Inspecting their $0$-jets we see that $g_t$ can meet $j^11_Q$ only over the region $\{f=1\}$. However, there it holds $\beta=\d f$, which is non-zero.
    \end{itemize}
    We have thus defined a homotopy of embeddings between $L_0=j^1 0_Q \sqcup j^1 1_Q$ and $L_N=j^1 f \sqcup j^1 1_Q$. To upgrade this to a homotopy of formal Legendrians we note that, at a given $\sigma \in J^1Q$, the possible tangent spaces for a Legendrian graphical over $Q$ are parametrised by the fibre of $J^2Q$ over $\sigma$, which is an affine space and thus contractible. This implies that upgrading to formal Legendrians graphical over $Q$ is unique up to homotopy, also relatively.
\end{proof}

\subsection{Legendrian stabilization and Loose Legendrians} \label{subsec:Leg_stabilzation}

Recall the notion of cusp singularity for a Legendrian (\cref{def:cusp}). We will now apply $N$-pushing in the vicinity of such a cusp. To do this as explicitly as possible, we introduce some models.

We consider a $1$-dimensional Legendrian cusp $L_\mathrm{cusp}^0\subseteq J^1 \R$ defined by $L_\mathrm{cusp}^0\cap \{x\leq 0\}=j^1 0_\R\sqcup j^1 1_\R \cap \{x\leq 0\}$, while over $\{x\geq 0\}$ the front $\pi_\F(L_\mathrm{cusp}^0)$ is given by a single cusp connecting $j^0 0_\R$ and $j^0 1_\R$. See \cref{fig:stabilization}
\begin{figure}[h]
    \centering
    \includegraphics[width=1\linewidth]{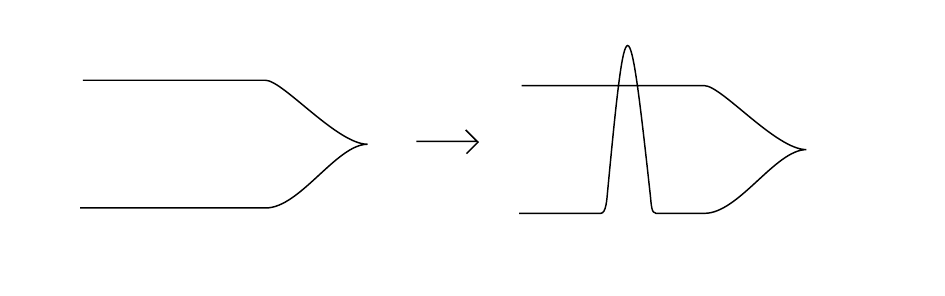}
    \caption{Left: the front of the cusp model $L^0_\mathrm{cusp}$. Right: the stabilized front $L_\mathrm{loose}$.}
    \label{fig:stabilization}
\end{figure}

Let $Q$ be a $(n-1)$-dimensional manifold, and define $X=\R\times Q$. Consider $1$-jet bundle $J^1X = J^1\R \times T^*Q$. Taking the product of the $1$-dimensional cusp with $Q$ yields a Legendrian submanifold $L_\mathrm{cusp}\subseteq J^1\R \times T^*Q$:
\begin{equation*}
    L_\mathrm{cusp} \defi L_\mathrm{cusp}^0 \times Q \subseteq J^1\R\times T^*Q,
\end{equation*}
note that $Q\subseteq T^*Q$ denotes the $0$-section.

Notice that over $\{x\leq 0\}$ we have that
\begin{equation*}
    L_\mathrm{cusp}\cap \{x\leq 0\}= (j^1 0_\R\sqcup j^1 1_\R \cap \{x\leq 0\})\times Q=j^1 0_{\{x\leq 0\}\times Q}\sqcup j^1 1_{\{x\leq 0\}\times Q},
\end{equation*}
which is the situation described in the previous section. Therefore, given a closed submanifold $N\subseteq \R\times Q\cap \{x< 0\}$ with $\chi(N)=0$, we define the $L_\mathrm{loose}\defi \mathrm{p}_N(L_\mathrm{cusp})$.

 \begin{definition}
    Let $(Y^{2n+1},\xi)$ be a contact manifold of dimension $2n+1 \geq 5$, $L\subseteq Y$ a Legendrian submanifold and $U\subseteq Y$ an open \nbh such that
    \begin{equation*}
        (U, L\cap U,\xi) \simeq (J^1 X, L_\mathrm{cusp},\xi_\std),
    \end{equation*}
    for some manifold $X=\R\times Q$. The \emph{stabilization} of $L$ is the Legendrian $s(L)$ obtained by replacing $(U, L\cap U,\xi)$ by  $(J^1 X, L_\mathrm{loose},\xi_\std)$.
\end{definition}

\begin{definition}
    Let $(Y^{2n+1},\xi)$ be a contact manifold of dimension $2n+1 \geq 5$. A Legendrian submanifold $L\subseteq Y$ is \emph{loose} if $L=s(L')$ for some Legendrian submanifold $L'$.
\end{definition}

The relevance of loose Legendrians and the justification of dropping $N$ from the notation in the definitions relies on the following $h$-principle of Murphy:
\begin{theorem}[Murphy \cite{Murphy:Loose}] \label{thm:h-principle_loose_leg}
    Two loose Legendrians are Legendrian isotopic if and only if are formally Legendrian isotopic. Moreover, every formal Legendrian is formally Legendrian isotopic to a genuine (loose) Legendrian by a $C^0$-small isotopy.
\end{theorem}

\section{Legendrian spinning} \label{sec:Leg_spinning}

In low dimensions, generic fronts are simple. In contrast, a front in high dimensions may be extremely singular and thus complicated to represent. In this section we introduce a construction of higher-dimensional fronts that uses lower-dimensional fronts as input. This construction is called \emph{spinning}; it was introduced first in \cite{EES:NonIsotopicLegendrians} to construct of Legendrian submanifolds in $\R^{2n+1}$. It has been further studied in \cite{EkholmEtnyreSabloff,Golovko,Lambert-Cole:LegendrianProducts}.

\subsection{Front spinning}
Let us recall some basic facts about the $1$-jet bundles, which is a crucial part of the spinning construction.
\begin{lemma}
    Let $L_N$ and $L_M$ be bundles over $N$ and $M$, respectively, with $1$-dimensional fibres. Consider a diagram of embeddings:
    \begin{equation*}
        \begin{tikzcd}
            L_N \arrow[rr, "\varphi"] \arrow[dd]&  & L_M \arrow[dd] \\
                                                &  &                \\
            N \arrow[rr, "e"]                   &  & M             
        \end{tikzcd}
    \end{equation*}
    Then, there is an induced $1$-jet bundle map:
    \begin{equation*}
        \hat{\varphi}: (J^1(L_N),\xi_\std) \to (J^1(L_M),\xi_\std),
    \end{equation*}
    which is an isocontact embedding.
\end{lemma}
In particular, $\hat{\varphi}$ sends isotropics to isotropics. In the case $\dim(N)=\dim(M)$, it defines a map between the corresponding spaces of Legendrians:
\begin{equation*}
    \hat{\varphi}_*: \Leg(J^1(L_N),\xi_\std)\to \Leg(J^1(L_M),\xi_\std).
\end{equation*}
Moreover, this construction  works parametrically. If $\varphi_s$ is a homotopy of fibred embeddings, then the induced isocontact embeddings $\hat{\varphi}_0$ and $\hat{\varphi}_1$ are contact isotopic. In particular, in the equidimensional case, the induced maps between Legendrian embedding spaces are homotopic.

\begin{remark}
    In the case where $L_N = N \times \R$ and $L_M = M \times \R$ are trivialised line bundles, the embedding $e \colon N \hookrightarrow M$ induces an isocontact embedding
    \begin{equation*}
        T^*N \times \R \to T^*M \times \R.
    \end{equation*}

    In this standard construction the $\R$–coordinate is mapped identically, but this is not necessary. In fact, for the spinning construction it will be convenient to add a suitable shift the $\R$-coordinate. This motivates us to consider bundle maps of the form
    \begin{align*}
        \varphi : N\times \R&\to M\times \R,\\
        (n,z)&\mapsto \big(e(n),z+h(n)\big),
    \end{align*}
    where $h\in C^\infty(N;\R)$; by the previous lemma these lift to isocontact embeddings of $1$-jet spaces.
\end{remark}

The second easy observation that we use is the following:
\begin{lemma}
    For every smooth manifolds $N^m$ and $K$, there exists a well-defined front-graph map:
    \begin{equation*}
        G_K: \mathrm{Map}(K,\Leg(N,(J^1\R^m,\xi_\std))) \to \Leg(K\times N,(J^1(K\times \R^m),\xi_\std)).
    \end{equation*}
\end{lemma}
\begin{proof}
    Given $L_k \in \mathrm{Map}(K,\Leg(N,(J^1\R^m,\xi_\std)))$, we define the Legendrian embedding $G_K(L_k) \in \Leg(K\times N,(J^1(K\times \R^m),\xi_\std))$ by the front map:
    \begin{align*}
        G_K(L_k): K\times N \to J^0(K\times \R^m) = K\times J^0\R^m, \quad (k,n)\mapsto (k,\pi_\F(L_k(n))),
    \end{align*}
    where $\pi_\F:J^1\R^m\to J^0\R^m$ is the front projection.
\end{proof}

Now, we can define spinning: 
\begin{definition}
    Let $e:K\times \R^m\to \R^n$ be a smooth embedding of codimension $0$, covered by a fibred embedding $\varphi: J^0(K\times \R^m) \to J^0\R^n$. Let $N$ be a smooth manifold of dimension $m$. The \emph{spinning map} of $N$ along $K$ via $\varphi$ is defined as the composition:
    \begin{equation*}
        \hat{\varphi}_* \circ G_K: \mathrm{Map}(K,\Leg(N,(J^1\R^m,\xi_\std))) \to \Leg(K\times N,(J^1\R^n),\xi_\std). \qedhere
    \end{equation*}
\end{definition}

\subsection{A specific spinning}\label{subsec:special_spinning}

We will only use the spinning map in a very concrete situation, that we describe below. Let $\hat{e}:K\hookrightarrow J^0 \R^n =\R^n\times \R_z$ be a smooth embedding graphical over $\R^n$, that means 
\begin{equation*}
    \hat{e}=(\hat{e}_{\R^n},h): K \to \R^n\times \R_z,
\end{equation*}
where $\hat{e}_{\R^n}:K\hookrightarrow \R^n$ is also an embedding and $h\in C^{\infty}(K)$ is a smooth function. We further assume that the embedding $\hat{e}_{\R^n}$ has trivialisable normal bundle. We fix a framing $\mathrm{Fr}$ of the normal bundle of $\hat{e}_{\R^n}(K)$, to define an embedding:
\begin{equation*}
    e: K \times \R^m \to \R^n.
\end{equation*}
This codimension-$0$ embedding is naturally covered by the fibrewise isomorphism 
\begin{equation*}
    \varphi:J^0(K\times \R^m) \to J^0\R^n, \quad (k,v,z)\mapsto (e(k,v),h(k)+z).
\end{equation*}
Note that the map $\varphi$ is an affine isomorphism, and is invariant along the normal directions.

In this situation, given a family of Legendrian embeddings $L_k\in \mathrm{Map}(K,\Leg(N,(J^1\R^m,\xi_\std)))$, we will denote by
\begin{equation*}
    \spin_{\hat e,\mathrm{Fr}}[L_k]\defi \hat{\varphi}_* \circ G_K (L_k)\in\Leg(N\times K, (J^1 \R^n,\xi_\std))
\end{equation*}
the spinning of the family $\{L_k\}_{k\in K}$. In the cases that we will treat later the framing will be unique so we will simply write $\spin_{\hat e}[L_k]$, and if $L_k\equiv L$ is a constant family we will write $\spin_{\hat e}[L]$.

\begin{figure}
    \centering
    \includegraphics[width=1\linewidth]{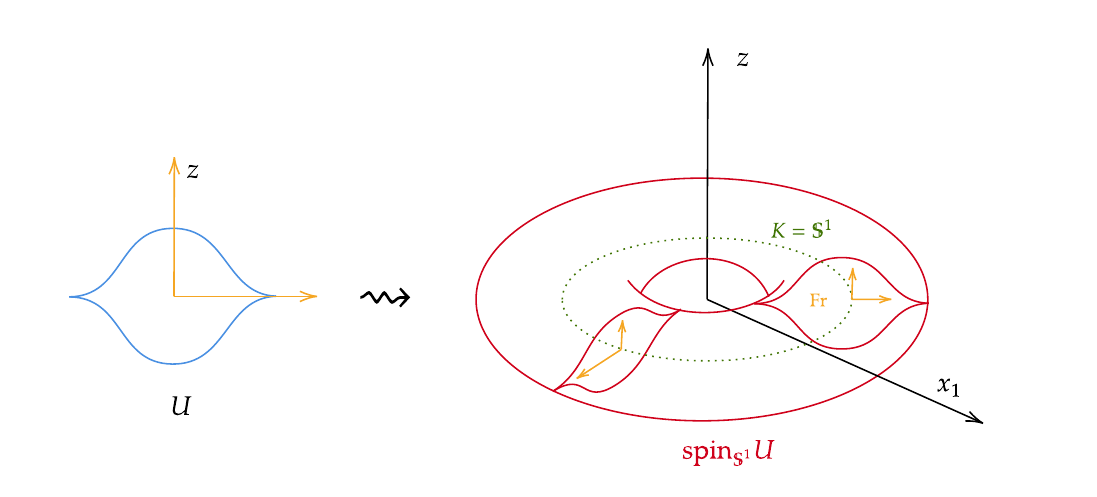}
    \caption{A schematic depiction of the front spinning of the Legendrian unknot $U\subseteq (J^1\R,\xi_\std)$ along the standard circle $\S^1\subseteq \R^2$.}
    \label{fig:FrontSpinning}
\end{figure}

The following result follows by combining the Ekholm-Etnyre-Sullivan computation of the formal classes of a spinning  \cite{EES:LCH}, together with Murphy classification of formal Legendrians \cite{Murphy:Loose}. However, we provide a direct proof since it will be fundamental for our study of prelegendrians.
\begin{lemma}\label{prop:rot=0_formal_isotopic}
    Let $\hat{e}:K^{n-1}\rightarrow J^0 \R^n$ be a smooth embedding, graphical over $\R^n$, with $\chi(K)=0$. Fix moreover a framing $\mathrm{Fr}$.
    
    If $L_0,L_1\in\Leg(J^1\R^1,\xi_\std)$ are two Legendrian knots with $\operatorname{Rot}(L_0)=\operatorname{Rot}(L_1)$, then $\spin_{\hat e,\mathrm{Fr}}[L_0]$ and $\spin_{\hat e,\mathrm{Fr}}[L_1]$ are formally Legendrian isotopic. Moreover, the formal Legendrian isotopy is induced by a sequence of Reidemeister moves of the $1$-dimensional Legendrians $L_0$ and $L_1$, and $N$-pushings along parallel copies of $\hat{e}(K)$.
\end{lemma}

\begin{proof}
    Following the notations above we write $e:K\times \R^2\rightarrow J^0\R^n$, to denote the tubular neighbourhood embedding of $\hat e$ induced by $\mathrm{Fr}$, in which we identify $\{k\}\times\R^2=J^0\R$. 
    
    By the $h$-principle for Legendrian immersions in $\R^3$, there is a generic homotopy of Legendrian immersions $L_t$, $t\in [0,1]$, between the given Legendrian knots. There is a finite number of times $0<t_1<\ldots <t_n<1$ in which a unique transverse self-intersection of $L_{t_i}$ appears at a point $p_i$. This is seen as a tangency in the front projection. Consider the spinning of $L_{t}$ in the time interval $t\in [t_i-\varepsilon, t_i+\varepsilon]$. It is an $N$-pushing along $N = e(K\times \{\pi_\F(p_i)\})$. The result follows then from \cref{prop:FormalClassNPushing}.
\end{proof}

The original definition of spinning was given in \cite{EES:NonIsotopicLegendrians} for $K=\S^1$ and later generalized by Golovko in \cite{Golovko} to $K=\S^k$. They work with the standard inclusion; specifically, they consider
\begin{equation*}
    \hat{e}: \S^k \hookrightarrow \R^{k+1}\times \{0\}\times \R_z \subseteq \R^{k+1}\times \R^{m-1} \times \R_z,
\end{equation*}
equipped with  the canonical normal framing of $\S^k\subseteq \R^{k+1} \times \R^{m-1}$. In this case, we denote the spinning construction simply by $\spin_{\S^k}[L]$ and refer to it as the \emph{standard $\S^k$-spinning}. This yields a map
\begin{equation*}
    \spin_{\S^k}:\Leg(N^{m},(J^1\R^m) )\to \Leg(\S^k\times N^m,J^1\R^{m+k}).
\end{equation*}

\subsubsection{Torus spinning}

For our purposes (constructing prelegendrians via spinning along a general embedding of the torus) we need to verify that torus spinning produces the same Legendrian isotopy class as iterated standard $\S^1$-spinning. The reason is that we understand the effect that $\S^1$-spinning has on Legendrian Contact Homology; see \cite[Example 4.20]{EES:NonIsotopicLegendrians}. 

The \emph{standard $k$-torus} $\T^k\subseteq \R^{k+1}$ is defined inductively as follows:
\begin{itemize}
    \item $\T^1 =\S^1\subseteq \R^{2}$ is the standard inclusion.
    \item Assume that $\T^{k-1}\subseteq \R^k$ is defined. Consider
    \begin{equation*}
        \S^1\subseteq \R^2(x_1,x_2)\times \{0\}\subseteq \R^2(x_1,x_2)\times  \R^{k-1}(x_3,\ldots,x_{k+1})
    \end{equation*}
    endowed with the canonical normal framing $\nu\simeq \mathrm{span}\langle\partial_r,\partial_{x_3},\ldots,\partial_{x_{k+1}}\rangle$. Identifying $\nu\cong \R^{k}$ and using the standard inclusion $\T^{k-1}\subseteq \R^{k}\simeq \nu$, we obtain the standard embedding $\T^k\subseteq \R^{k+1}$.
\end{itemize}
For the standard $k$-torus $\T^k\subseteq \R^{k+1}$ there is also a canonical normal framing; therefore we denote the corresponding spinning simply by $\spin_{\T^k}[L]$ for $L\in \Leg(J^1\R^m,\xi_\std)$, in complete analogy with the $k$-sphere case. The following lemma is immediate from the definition of the standard $k$-torus.
\begin{lemma}\label{lem:torus_spin_is_iter_S1_spin}
    Let $L\subseteq J^1\R^m$ be a Legendrian submanifold. Consider the standard torus $\T^k\subseteq \R^{k+1}$, then
    \begin{equation*}
        \spin_{\T^k}[L] \simeq (\spin_{\S^1})^k[L].
    \end{equation*}
\end{lemma}

\section{Prelegendrian fronts} \label{sec:Preleg_front}

Consider an almost complex manifold $(X,J)$, and the associated fat manifold $(\Gr(X,J),\D_\can)$. By \cref{thm:contactisation_grasm}, the contactisation $\cC(\Gr(X,J),\D_\can)$ is contactomorphic to the unit cotangent bundle $\S T^*X$. This yields the commutative diagram:
\begin{equation*}
    \begin{tikzcd}[sep=scriptsize]
    	{(\Gr(X,J),\D_\can)} &&&& {(\mathbb{S}T^*X,\xi_\can)} \\
    	\\
    	\\
    	&& {(X,J)}
    	\arrow["{\pi_\PF}"', from=1-1, to=4-3]
    	\arrow["\Hopf"', from=1-5, to=1-1]
    	\arrow["{\pi_\F}", from=1-5, to=4-3]
    \end{tikzcd}
\end{equation*}

For simplicity, we assume throughout that all prelegendrians are \emph{co-oriented} (see \cref{subsec:preisotropic_submanifold}).

Given an embedded prelegendrian $\Lambda\subseteq \Gr(X,J)$, we will refer to $\pi_{\PF}(\Lambda)\subseteq (X,J)$ as a \emph{prelegendrian front}. Every generic Legendrian submanifold of $L\subseteq \S T^*X$ can be represented by its front $\mathcal{F}=\pi_\mathrm{F}(L)$ in $X$.  However, not every front $\cF$ in $X$ is a prelegendrian front; strong geometric constraints must be satisfied along the singularities of $\cF$ for this to be the case. Indeed, note that $\Hopf(L)$ is always tangent to $\D_\can$ but it may be singular due to the presence of self-intersections or singularities of mapping.

The main result of this section is a complete characterization of which \emph{simple} fronts are prelegendrian. It is a rephrasing of \cref{thm:PreLegendrianFronts}, from the introduction:
\begin{theorem}[Simple Prelegendrian Fronts]\label{thm:PrelgendrianFronts}
    A simple front $\mathcal{F}\subseteq(X,J)$ is a prelegendrian front if and only if its singular loci $\Sigma= \Sigma(\mathcal{F)}$ are co-real. Moreover, the prelegendrian $\Lambda\subseteq (\Gr(X,J),\D_\can)$ determined by $\mathcal{F}$ is unique.
\end{theorem}

\begin{remark}\label{rmk:preleg_immersion}
    We point out that a ``simple'' front $\mathcal{F}\subseteq(X,J)$ whose cusp locus is co-real, but whose self-intersection locus is either not co-real or not transverse, also defines a unique \emph{immersed} prelegendrian $\Lambda\subseteq (\Gr(X,J),\D_\can)$. The self-intersection locus of $\Lambda$ is precisely the preimage, under the projection $\pi_{\PF}$, of the union of the complex points of the self-intersection locus of $\mathcal{F}$ and the non-transverse self-intersections of $\mathcal{F}$
\end{remark}

\subsection{Proof of \cref{thm:PrelgendrianFronts}}

Since the property of being a prelegendrian front is local, it suffices to work with the germ of $\cF$. We therefore argue case by case according to the singularity type.

Let $p\in\cF\subseteq (X,J)$. Let $2n+2$ be the dimension of $X$. Since $\cF$ is a simple front, there exists a diffeomorphism from an open neighbourhood of $p$ in $X$ onto an open neighbourhood of $0$ in $\R^{2n+2}$:
\begin{equation*}
    \psi:(\Op(p),p)\to (\R^{2n+2},0),
\end{equation*}
such that, via $\psi$, the germ of $\cF$ at $(\Op(p),p)$ is in the normal form of \cref{def:simple_front}, namely:
\begin{itemize}
    \item [(i)] For the $A_1$-singularity, $\cF=\R^{2n+1}\times \{0\}\subseteq \R^{2n+2}$ is the horizontal hyperplane. These are smooth points of the front.
    \item [(ii)] For the $A_1A_1$-singularity, the front $\cF$ is the union of two hyperplanes intersecting transversely at $0$. This corresponds to a transverse self–intersection of the front.
    \item [(iii)] For the $A_2$-singularity, $\mathcal{F}$ is parametrized by the front map
    \begin{equation}\label{eq:param_cusp}
        \begin{aligned}
        f:(\R^{2n+1},0)\to& (\R^{2n+2},0) \\
        (q_1,\ldots,q_{2n},t) \mapsto& (q_1,\cdots,q_{2n},-3t^2,2t^3).
        \end{aligned}
    \end{equation}
    These are the cusp points of the front. 
    
\end{itemize}
For simplicity, we continue to denote by $J$ the pushforward almost–complex structure $\psi_*J$ on $\R^{2n+2}$. Since the almost–complex Grassmannian construction is functorial, we may reduce to the case $X=\R^{2n+2}$, with $\cF$ one of the three types above.

We treat each case separately. We begin with the smooth stratum (i). In this case, the prelegendrian is determined by the complex tangencies of the front (in the contact setting, by the real tangencies). This is the content of the following:

\begin{proposition}[Smooth points]\label{prop:preleg_smooth_point}
    The (oriented) horizontal hyperplane $\mathcal{F}=\R^{2n+1}\times \{0\}\subseteq \R^{2n+2}$ admits a prelegendrian lift:
    \begin{align*}
        g: \R^{2n+1} &\to \Gr(\R^{2n+2},J)= (\R^{2n+2}\times \CP^{n},\D_{\std,J});\\
        \mathbf{q} &\mapsto \big( (\mathbf{q},0); H\cap J_x H \big),
    \end{align*}
    such that the following diagram commutes: 
    
    \begin{equation}
        \begin{tikzcd}
	&&& {\S T^*\R^{2n+2}=\R^{2n+2}\times \S^{2n+1}} \\
	\\
	{\R^{2n+1}} &&& {\Gr(\R^{2n+2},J)=\R^{2n+2}\times \CP^{n}}
	\arrow["\Hopf", from=1-4, to=3-4]
	\arrow["{\cL(g)}", from=3-1, to=1-4]
	\arrow["g", from=3-1, to=3-4]
    \end{tikzcd}
    \end{equation}
    where $\cL(g)(\mathbf{q}) = \big( (\mathbf{q},0); H \big)$ is the Legendrian lift of $g$.
\end{proposition}
\begin{proof}
    The fibrewise Hopf map $\Hopf:\S T^*\R^{2n+2}\to \Gr(\R^{2n+2},J)$ (see \cref{sec:Review_coreal}) sends an oriented hyperplane to its complex part. The Legendrian determined by the hyperplane front $\mathcal{F}$ is parametrised by
    \begin{equation*}
        (q_1,\ldots,q_{2n+1}) \to (q_1,\ldots,q_{2n+1},0; H)\in \R^{2n+2}\times \S^{2n+1}
    \end{equation*}
    Thus the prelegendrian map is the composition
    \begin{align*}
        g:\R^{2n+1}&\to \R^{2n+2}\times \S^{2n+1} \to \R^{2n+2}\times \CP^{n},\\
        \mathbf{q}&\mapsto (\mathbf{q},0;H_0)\mapsto (\mathbf{q},0:H_0\cap J H_0).
    \end{align*}
    Since the first projection parametrises $\mathcal{F}$, $g$ is an embedding; and it is prelegendrian by the definition of $\D_\can$.
\end{proof}

Now, we deal with the self-intersection points (ii). In this case, the prelegendrian lift of each smooth branch is given by the complex tangencies of the front (\cref{prop:preleg_smooth_point}).  The lift of both branches will be embedded if and only if the complex tangencies of each branch are transverse. Since they intersect in a $2n$-dimensional space, it is enough to require this intersection to not be complex: 
\begin{proposition}[Self-intersections] \label{prop:preleg_fton_intersection_point}
    Let $\mathcal{F}=H_1\cup H_2\subseteq \R^{2n+2}$ be the front given by the union of two transversal hyperplanes. Then, it lifts to an embedded prelegendrian  if and only if $\Sigma= H_1\cap H_2$ is co-real.
\end{proposition}

\begin{proof}
    We have that
    \begin{equation*}
        \Sigma^{2n}\cap J(\Sigma^{2n})= (H_1\cap J H_1)^{2n}\cap (H_2\cap JH_2)^{2n} \subseteq \Sigma^{2n}.
    \end{equation*}
    By dimensional reasons, $H_1\cap J H_1 = H_2\cap JH_2$ if and only if $\Sigma$ is $J$-complex.
\end{proof}

Finally, we consider the cusp stratum (iii). According to \cref{prop:preleg_smooth_point}, there is a genuine prelegendrian lift along the smooth stratum. It extends uniquely to the cusp points because it is the projection of the legendrian lift, which is an embedding. We must check that this lift is an immersion. In the contact case this is true because the tangent space of the front rotates continuously when one crosses the cusp locus transversely. It follows that we must check that the complex tangencies also rotate as we cross the cusp locus. The following geometric condition ensures this:
\begin{proposition}[Cusps points] \label{prop:PrelegendrianCusp}
    Let $f:\R^{2n+1}\to \R^{2n+2}$ be the front map given in Equation \eqref{eq:param_cusp}, and let $\cF=f(\R^{2n+1})$. Then $f$ lifts to a prelegendrian embedding if and only if $\Sigma = \Sigma(\cF)$ is co-real.
\end{proposition}
\begin{proof}
    First, write the coordinates of $\R^{2n+2}$ as $(x_1,\ldots,x_{2n+2})$. Define the family of $1$-forms $\a_t\defi t\d x_{2n+1}+\d x_{2n+2}$. A simple computation shows that the tangent space of the front $\cF$ is 
    \begin{equation*}
        T_{f(\mathbf{q},t)}\cF =\ker \alpha_t.
    \end{equation*}
    The singular set of $\cF$ is
    \begin{equation*}
        \Sigma = f(\R^{2n}\times\{0\})= \{ x_{2n+1} = x_{2n+2}=0 \}\subseteq \mathcal{F},
    \end{equation*}
    with tangent space
    \begin{equation*}
        T_{f(\mathbf{q},0)}\Sigma=\mathrm{span}(\partial_{x_1},\ldots,\partial_{x_{2n}}) = \ker (\d x_{2n+1})\cap\ker(\d x_{2n+2}),
    \end{equation*}
    Therefore, the Legendrian embedding associated to $f$ is
    \begin{align*}
        \tilde{g}:\R^{2n+1}\to \R^{2n+2}\times \S^{2n+1} ;\\
        (\mathbf{q},t)\mapsto (f(\mathbf{q},t), \alpha_t),
    \end{align*}
    with dual identification. According to \cref{prop:preleg_smooth_point}, the associated prelegendrian map must be
    \begin{align*}
        g:\R^{2n+1}&\to \R^{2n+2}\times \CP^n;\\
        (\mathbf{q},t)&\mapsto (f(\mathbf{q},t); (\alpha_t)^J),
    \end{align*}
    where
    \begin{equation*}
        (\a_t)^J = t (\d x_{2n+1})^J + (\d x_{2n+2})^J
    \end{equation*}
    is the complexification of $\a_t$.

    Observe that $g$ is injective, and, moreover, is an embedding away from $t=0$. At $\{t=0\}$, the vectors $\{Df_{(\mathbf{q},0)} (\partial_{q_i}), i=1,\ldots, 2n\} \subseteq T_{f(\mathbf{q},t)} \R^{2n+2}$ span a $2n$-dimensional subspace, while $Df_{\mathbf{q},0} (\partial_t)=0$. Therefore, $g$ will be an immersion at a point $(\mathbf{q},0)$ if and only if 
    \[ Dg_{(\mathbf{q},0)} (\partial_t)=(0,D_{(\mathbf{q},0)}(\a_t)^J (\partial_t))\neq 0. \] 

    Since $T_{f(\mathbf{q},0)}\Sigma = \ker (\d x_{2n+1})\cap\ker(\d x_{2n+2})$, \cref{lem:dual_complex_basis} applies, i.e. $T_{f(\mathbf{q},0)}\Sigma$ is $J$-invariant if and only if the complexifications $(\d x_{2n+1})^J$ and $(\d x_{2n+2})^J$ are $\C$–linearly dependent.

    Assume first that $\Sigma$ is co-real. Then, $\{(\d x_{2n+1})^J,(\d x_{2n+2})^J\}$ are $\C$-linearly independent along $\Sigma$. Since being linearly independent is an open condition, this is true in a neighbourhood of the cusp locus. Therefore, $\{(\d x_{2n+1})^J,(\d x_{2n+2})^J\}$ can be completed to a $\C$-basis of $\Hom_\C(T\R^{2n+2},\C)$ locally. In that basis, the value of $g$ in the $\CP^n$ component is of the form $[0:\cdots:0:t:1]$. In particular, the derivative with respect to $t$ is non-zero, hence $g$ is an immersion.

    Conversely, assume that $\Sigma$ is not co-real at some $f(\mathbf{q}_0,0)$, i.e. $T_{f(\mathbf{q}_0,0)}\Sigma$ is $J$-invariant. Hence $(\d x_{2n+1})^J$, $(\d x_{2n+2})^J$ are $\C$-linearly dependent at the point $f(\mathbf{q}_0,0)$. As in the  previous case, we complete $(\d x_{2n+2})^J$ to a basis $\{v_1,\ldots, v_n, (\d x_{2n+2})^J\}$ over $\R^{2n+2}$. Locally around $f(\mathbf{q}_0,0)$ we express $(\d x_{2n+1})^J$ in terms of this basis:
    \[ (\d x_{2n+1})^J = a_1(\mathbf{q},t)v_1+\cdots+ a_n(\mathbf{q},t)v_n + a_{n+1}(\mathbf{q},t) (\d x_{2n+2})^J, \]
    where $a_i(\mathbf{q},t)\in\C$,  $i=1,\ldots,n+1$, are the coefficients. Using the assumption that $(\d x_{2n+1})^J$ and $(\d x_{2n+2})^J$ are $\C$-linearly dependent at $f(\mathbf{q}_0,0)$ we have that:
    \begin{itemize} 
    \item $a_{n+1}(\mathbf{q}_0,t)\neq 0$
    \item $a_{i}(\mathbf{q}_0,t)=O(t)$ for $i=1, \ldots, n$.
    \end{itemize}
    
    Therefore, for $t$ sufficiently small, the value of $g(\mathbf{q}_0,t)$ in the $\CP^n$ component is of the form:
 \begin{equation*}
\begin{aligned}
  [t a_1(\mathbf{q}_0,t):\cdots: t a_n(\mathbf{q}_0,t):1+t a_{n+1}(\mathbf{q}_0,t)]
  &= [O(t^2):\cdots:O(t^2):1+O(t)] \\
  &= [O(t^2):\cdots:O(t^2):1].
\end{aligned}
\end{equation*}
This implies that the derivative of $g$ along the $t$-direction vanishes at the point $(\mathbf{q},0)$, and $g$ fails to be an immersion.
\end{proof}

\subsection{Simple prelegendrian Reidemeister moves}

The study of prelegendrian simple fronts conducted above can be generalised to $1$-parametric families. 
\begin{definition}
We will say that a one–parameter family of Legendrian fronts  $(\cF_s)_{s\in[0,1]}$ in $(X,J)$ undergoes a \emph{Reidemeister I move} if
\begin{itemize}
    \item[(1) ] There is a closed coorientable hypersurface $O \subseteq \cF_0$ contained in the smooth locus such that
    \item[(2) ] $\cF_s$ is supported in $\Op(O) \cong O \times \R^2$ with $\Op(O) \cap \cF_0 \cong O \times \R$, and
    \item[(3) ] $\cF_s$ is the product of $O$ with the standard $1$-dimensional move of $\R$ in $\R^2$ (\cref{fig:R_moves}).
\end{itemize}
The \emph{Reidemeister moves II and III} are defined similarly by considering $O$ to be one of the components of the cusp and self-intersection loci, respectively.
\end{definition}

Then:
\begin{lemma}\label{lem:PrelegendrianReidemeister}
    Suppose that $\cF_0$ is a simple prelegendrian front. Introduce a Reidemeister move along a coreal submanifold $O \subseteq \cF_0$. Then, the resulting family of fronts $(\cF_s)_{s\in[0,1]}$ is prelegendrian.
\end{lemma}
\begin{proof}
By construction, the singularities of $\cF_{1/2}$ occur along an embedded co-real submanifold. The argument in the proof of \cref{prop:preleg_fton_intersection_point} applies verbatim for the RII and RIII moves. For the RI move, the argument from \cref{prop:PrelegendrianCusp} applies.
\end{proof}
The resulting prelegendrian homotopies are referred to as \emph{prelegendrian Reidemeister} I, II and III moves.

\begin{figure}[h]
    \centering
    \includegraphics[width=1\linewidth]{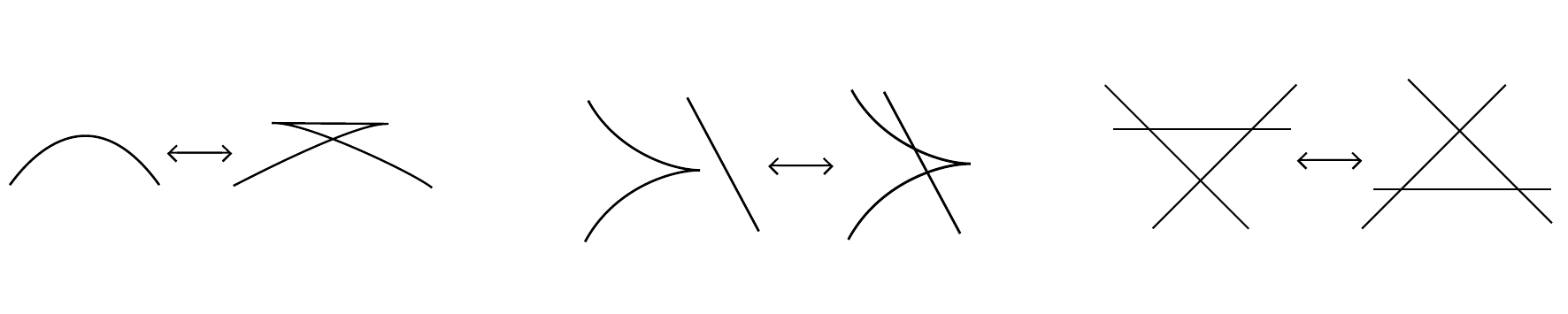}
    \caption{From left to right the fronts $\hat{\cF}_s\subseteq J^0\R$, $s\in[0,1]$, of the $1$-dimensional Reidemeister moves I, II and III. The higher dimensional moves are locally modelled on $\hat{\cF}_s\times\R^n\subseteq J^0\R\times\R^n=J^0\R^{n+1}$, $s\in[0,1]$.}
    \label{fig:R_moves}
\end{figure}

\begin{remark}\label{rmk:PrelegendrianRIMove}
In plain words, \cref{lem:PrelegendrianReidemeister} says that Reidemeister moves for prelegendrian fronts can be performed along co-real submanifolds. In particular, let $\cF$ be a prelegendrian front and $N\subseteq \Sigma_{A_1}(\cF)$ a codimension $1$ co-real submanifold in the smooth stratum. Then, there exists a homotopy of prelegendrian fronts $(\cF_s)_{s\in [0,1]}$, that describes an RI-move along $N$. Moreover, since being co-real is an open condition we may achieve that the region in the prelegendrian front $\cF_1$ in which the RI appears is foliated by co-real submanifolds parallel to $N$.
\end{remark}

\begin{remark}
A generic $1$-parametric family of Legendrians in higher-dimensions will undergo a wider class of bifurcations than those introduced above. We expect the same to be true for prelegendrians. Moreover, the three Reidemeister moves shown here are not generic and will generally appear as combinations of more elementary moves. Developing the theory further in this direction is left as an open question.
\end{remark}

\subsection{Prelegendrian spinning} \label{subsec:preleg_spinning}

\cref{thm:PrelgendrianFronts} provides a criterion ensuring that, under mild assumptions, the Legendrians in $(\R^{4n+3},\xi_\std)=(J^1\R^{2n+1},\xi_\std)$ produced by the Legendrian spinning of \cref{subsec:special_spinning} admit a prelegendrian front and therefore define prelegendrian embeddings in $(\R^{4n+2},\D_{\std,J})$. This yields an effective method to build prelegendrian embeddings and leads us to the construction of exotic prelegendrian tori in \cref{sec:Non_isotopic_preleg}.

The setup is as follows. Consider $\R^{2n+2}=J^0\R^{2n+1}$ equipped with an almost complex structure $J$. As explained in \cref{subsec:special_spinning}, given an embedding $\hat e=(\hat e_{\R^{2n+1}},h):K\hookrightarrow J^0\R^{2n+1}$ of codimension $m+1$ that is graphical over $\R^{2n+1}$, together with a framing of $\hat e_{\R^{2n+1}}$, there is a well-defined Legendrian spinning map. This operation takes as input a Legendrian $L\subseteq (J^1\R^m,\xi_\std)$ and produces a higher-dimensional Legendrian
\begin{equation*}
    \spin_{\hat e}[L]\subseteq (J^1 \R^{2n+1},\xi_\std)
\end{equation*}

Our first observation concerns the singularities of the front of the Legendrian spinning:
\begin{lemma}
    The singular locus of the front of Legendrian spinning $\spin_{\hat e}[L]$ satisfies:
    \begin{equation*}
        \Sigma\big(\pi_\F(\spin_{\hat e}[L])\big) = \varphi\big(K \times \Sigma(\pi_\F(L))\big),
    \end{equation*}
    In particular,
    \begin{equation*}
         \Sigma\big(\pi_\F(\spin_{\hat e}[L])\big) = \bigsqcup_{p\in \Sigma(\pi_\F(L))} \varphi(K \times  \{p\})
    \end{equation*}
    is partitioned into submanifolds diffeomorphic to $K$. These submanifolds may be assumed to be $C^\infty$-close to $\hat{e}(K)=\varphi (K\times \{0\})$ by considering $\pi_\F(L)$ sufficiently \emph{small}.
\end{lemma}

\begin{remark}
The meaning of \emph{smallness} in the lemma is the following: Given a Legendrian in $(J^1 \R^m,\xi_\std)$, we can Legendrian isotope it using the contact isotopy 
\[ \delta_t(\mathbf{x},\mathbf{y},z) = (t\mathbf{x},t\mathbf{y},t^2z), \qquad\textrm{ for }t>0. \]
Taking $t$ small enough produces a new Legendrian that is $C^0$-close to the origin. Morever, each of its branches, graphical over $\R^m$, may be assumed to be $C^\infty$-close to the $1$-jet of the zero function.
\end{remark}

An important consequence is the following:
\begin{corollary}\label{cor:prelegendrianspinning}
    Assume that $\hat{e}:K\hookrightarrow (\R^{2n},J)$ is a co-real embedding.  Let $L\subseteq(J^1 \R^m,\xi_\std)$ be a Legendrian with sufficiently small simple front. Then, the front $\pi_\F\big(\spin_{\hat e}[L]\big)$ is a prelegendrian front. In particular, it defines a prelegendrian embedding
    \begin{equation*}
        \pspin_{\hat{e}}[L]\subseteq (\R^{4n+2},\D_{\std,J})
    \end{equation*}
    that will be referred as \emph{prelegendrian spinning} of $L$ along the (co-real embedding) $\hat{e}$.
\end{corollary}
\begin{proof}
   The previous lemma combined with \cref{lem:bundle_over_coreal_is_coreal} implies that the singularities of the front are co-real whenever $\pi_\F(L)$ is sufficiently $C^\infty$-small. The result follows from \cref{thm:PrelgendrianFronts}.
\end{proof}

\begin{remark}
By definition, the Legendrian lift of the prelegendrian spinning is the Legendrian spinning:
\[ \cL( \pspin_{\hat{e}}[L] ) =\spin_{\hat{e}}[L]. \qedhere \]
\end{remark}

\begin{remark}\label{rmk:PrelegendrianSpinningReidemeister}
    We continue using the notation of the corollary, and work still under the same smallness assumption on the fronts. Assume that $L_s\subseteq (J^1 \R^m,\xi_\std) $, $s\in[0,1]$, is a Legendrian isotopy whose front describes a Reidemeister I, II or III move between two simple fronts $L_0$ and $L_1$. Then, by \cref{lem:PrelegendrianReidemeister}, there is a prelegendrian isotopy
    \begin{equation*}
        \pspin_{\hat{e}}[L_s]\subseteq (\R^{4n+2},\D_{\std,J}), s\in[0,1],
    \end{equation*}
    describing an RI, RII or RII prelegendrian move, respectively.
    In particular, assume that the co-real embedding $\hat{e}$ has codimension $2$ and, therefore, $L$ is a $1$-dimensional Legendrian. Then, the prelegendrian spinning depends solely, up to prelegendrian isotopy, on the Legendrian isotopy class of $L$.
\end{remark}

\section{Prelegendrian stabilization} \label{sec:Preleg_stab}

In this section we develop the prelegendrian $N$-pushing operation along a co-real submanifold $N$, describe its effect on the formal prelegendrian class, and explain how to stabilize prelegendrians. This will lead us to the proof of \cref{thm:PrelegendrianStabilization} in the next section.

\subsection{Prelegendrian $N$-pushing}

In the contact case, the $N$-pushing construction is defined in the local model of the $1$-jet bundle, where a natural front projection is available, and then transported via a Darboux chart to any contact manifold. For fat distributions of corank $2$ (or higher), there is no Darboux-type theorem. For this reason, we restrict our definition to the almost complex Grassmannian, which admits a prelegendrian front projection. For more general fat manifolds, refer to Section \ref{sec:general}. 

\subsubsection{Model case}

First we work in a local model where $X \defi J^0Q = Q\times\R(z)$, which allows us to consider the horizontal Grassmannians:
\begin{equation*}
    \begin{tikzcd}[sep=small]
        (\Gr_\mathrm{hor}(X,J),\D_{\std,J}) &&&&  (\mathbb{S}T_\mathrm{hor}^*X,\xi_\can ) \simeq (J^1Q,\xi_\std) \\
        \\
        \\
        && (X,J)
	\arrow["{{\pi_\PF}}"', from=1-1, to=4-3]
	\arrow["{{\Hopf}}"', from=1-5, to=1-1]
	\arrow["{\pi_\mathrm{F}}", from=1-5, to=4-3]
    \end{tikzcd}
\end{equation*}

We use the notation of \cref{subsec:N-pushing}, both in the following statement and its proof:
\begin{proposition}\label{prop:preleg_N-pushing_affine}
    Let $\Lambda\subseteq \Gr_\mathrm{hor}(X,J)$ be the prelegendrian with $\pi_\PF(\Lambda) = j^00_Q\sqcup j^01_Q \subseteq J^0Q$. Suppose that $N\subseteq j^0 0$ is a co-real submanifold that admits a non-zero section of the normal bundle $\nu_N^Q$. Then, there exists a prelegendrian $\mathrm{p}_N(\Lambda)$ such that
    \begin{itemize}
        \item $\pi_\PF(\mathrm{p}_N(\Lambda)) = j^0f\cup j^01_Q$,
        \item $\cL(\mathrm{p}_N (\Lambda)) = \mathrm{p}_N(\cL(\Lambda))$,
        \item $\Lambda$ and $\mathrm{p}_N(\Lambda)$ coincide outside of $\Gr_\mathrm{hor}(\Op_X(N\times [0,1]),J)$,
        \item $\mathrm{p}_N(\Lambda)$ is formally isotopic to $\Lambda$ as prelegendrians.
    \end{itemize}
    The prelegendrian $\mathrm{p}_N(\Lambda)$ is said to be an \emph{$N$-pushing} of $\Lambda$.
\end{proposition}
\begin{proof}
    Let us write $(\mathbf{x},z)$ for the coordinates in $X = Q \times \R$. Recall that $\Gr_\mathrm{hor}(X,J)$ is a bundle with fibre $\C^n\subset \CP^n$. This fibre encodes the slopes of the complex tangencies of fronts. Equivalently, each element represents a complex hyperplane transverse to the $z$-coordinate. It follows that such an element can be described as the kernel of a complex-valued $1$-form $(\d z-\alpha)^J$; here $\alpha$ is a covector on $Q$ and the superscript $J$ denotes complexification. This description is unique up to $\C^*$-scaling.
    
    Consider the front $\cF\defi j^0f\cup j^01_Q$. It is a simple front, since its singular locus consists only of the transverse self-intersection $j^0 f\cap j^0 1_Q$. After shrinking the model and homotoping the section $j^0 1_Q$ relative to its boundary, we may assume, using the openness of co-reality, that every shifted copy $N\times\{c\}\subseteq Q\times\R$ is co-real for $c\in[-1,3]$. Shrinking also implies that the intersection $j^0 f\cap j^0 1_Q$ is co-real, by \cref{lem:bundle_over_coreal_is_coreal}. Therefore, $\cF$ is a prelegendrian front and thus determines a prelegendrian $\Lambda'$. We set $\mathrm{p}_N(\Lambda)\defi \Lambda'$, and the first three claims follow directly from the construction. 

    It remains to prove the formal statement; we argue as in \cref{prop:FormalClassNPushing}. Note that the prelegendrian associated with the graph $j^0 f$ is parametrised by
    \begin{equation*}
        \mathbf{x}\in Q \longmapsto \big(\mathbf{x},f(\mathbf{x}), (dz-\d_\mathbf{x}f)^J\big).
    \end{equation*}
    We claim that there exists a $1$-form $\beta$ on $Q$ such that
    \begin{itemize}
        \item $[(dz-\beta)^J]\neq [(dz)^J] \in \C^n \subset \CP^{n}$ at each point in $N$. Invoking the shrinking argument introduced at the beginning of the proof the same will be true over the slab $M\times [-1,3]$, where $M = f^{-1}((1-\varepsilon,2])$.
        \item $\beta = \d_\mathbf{x}f$ on $f^{-1}(-\infty,1+\varepsilon)$. 
    \end{itemize}
    Assuming this for the moment, one repeats verbatim the three-stage isotopy of embeddings from \cref{prop:FormalClassNPushing}. Namely, we have
    \begin{equation*}
        g_t: Q\longrightarrow \Gr_{\mathrm{hor}}(Q\times \R,J), \qquad t\in[0,3],
    \end{equation*}
    given by the expressions:
    \begin{align*}
        g_t(\mathbf{x})& \defi \big(\mathbf{x},0,t(dz-\beta)^J\big),\quad &t\in [0,1];\\
        g_t(\mathbf{x})& \defi \big(\mathbf{x},(t-1)f(\mathbf{x)},(dz-\beta)^J\big),\quad &t\in [1,2];\\
        g_t(\mathbf{x})& \defi \big(\mathbf{x},f(\mathbf{x}),(t-2)(dz-\d_\mathbf{x}f)^J+(3-t)(dz-\beta)^J\big),\quad &t\in [2,3].
    \end{align*}
    In particular, $g_0$ and $g_3$ are the prelegendrian projections of $j^1 0_Q$ and $j^1 f$, respectively.

    The homotopy $g_t$ is compactly supported in a neighbourhood of $\mathrm{supp}(f)$ and does not meet the prelegendrian corresponding to $j^0 1_Q$, which is parametrised by
    \begin{equation*}
        \mathbf{x}\mapsto \big(\mathbf{x},1,(dz)^J\big),\quad \mathbf{x}\in Q.
    \end{equation*}
    As in \cref{prop:FormalClassNPushing}, $g_t$ is homotopic to the family induced by the linear interpolation between $0_Q$ and $f$. The prelegendrians in this family may have non-transverse self-intersections but are nonetheless immersed (\cref{rmk:preleg_immersion}). As such, they have well-defined formal prelegendrian data. An application of the homotopy lifting property yields then our desired formal prelegendrian isotopy.
    
    It remains to construct the $1$-form $\beta$. Consider $M \supset N$ as above. According to \cref{lem:existence_of_nonzero_section_fiber_ver} there is a non-vanishing section $\tilde Z$ of $TM^J$ that is inward-pointing along $\partial M$. In particular $df|_{\partial M}(\tilde Z) > 0$. By duality there is then a $1$-form $\beta$ on $M$ that is non-vanishing on $TM^J$ and agrees with $\d f$ along $\partial M$. It thus extends as $\d f$ to the rest of $Q$. Now, since $\d z$ vanishes on $Q$, we deduce that $\d z^J$ vanishes on $TM^J$, which implies $[(\d z-\beta)^J] \neq [\d z^J]$, concluding the proof.
\end{proof}

\subsubsection{General case}

We now describe the $N$–pushing in the general almost–complex Grassmannian setting:
\begin{theorem}[Prelegendrian $N$-pushing]\label{thm:Prelegendria_NPushing}
    Let $\Lambda \subseteq (\Gr(X,J),\D_\can)$ be a prelegendrian, $\cF\defi \pi_\PF(\Lambda)$, $N$ a closed manifold, and $\Phi: N \times [0,1] \hookrightarrow X$ an embedding such that:
    \begin{itemize}
        \item $N_0\defi \Phi(N\times \{0\})$, $N_1\defi\Phi(N\times \{1\})$ are contained in the smooth strata $\Sigma_{A_1}(\mathcal{F})$.
        \item $N_1$ is a co-real submanifold of $(X,J)$.
        \item $\Phi(N\times [0,1])\cap \mathcal{F} = \Phi(N\times \{0,1\})$,
        \item the normal bundle $\nu_{N_1}^{\cF}$ admits a non-zero section.
    \end{itemize}
    
    Then, there exists a prelegendrian $\mathrm{p}_N(\Lambda)$ such that:
    \begin{itemize}
        \item $\Lambda\cap \Gr(X\backslash\Op(\Phi),J)$ = $\mathrm{p}_N(\Lambda)\cap\ \Gr(X\backslash\Op(\Phi),J)$.
        \item The Legendrian lifts $\cL(\Lambda)$ and $\cL(\mathrm{p}_N(\Lambda))$ in $(\P T^*X,\xi_\std)$ satisfy:
        \begin{equation*}
        \cL(\mathrm{p}_N(\Lambda))=\mathrm{p}_N(\cL(\Lambda)) 
        \end{equation*}
        \item $\mathrm{p}_N(\Lambda)$ and $\Lambda$ are formally prelegendrian isotopic.
    \end{itemize}
    The prelegendrian $p_N(\Lambda)$ is said to be an $N$-pushing of $\Lambda$.
\end{theorem}
\begin{proof}
    Let $Q\defi \Op_{\mathcal{F}}(N_0)$ be a neighbourhood of $N_0$ in $\mathcal{F}$. The embedding $\Phi:N\times[0,1]\hookrightarrow X$ extends to an embedding $\widetilde{\Phi}:Q\times[0,1]\hookrightarrow X$. Thus $\Phi(Q\times[0,1])\subseteq (X,J)$ is identified with the model $(Q\times[0,1],\widetilde{\Phi}^*J)$ of \cref{prop:preleg_N-pushing_affine}, which yields the result.
\end{proof}

\subsection{Prelegendrian stabilization}\label{subsec:PrelegendrianStabilization}

We can now spell out the concrete case of the prelegendrian stabilization. 
\begin{definition}
Let $\Lambda\subseteq (\Gr(X,J),\D_\can)$ be a prelegendrian, $p\in \mathcal{F}= \pi_\mathrm{PF}(\Lambda)$ be a smooth point, and $N\subseteq \Op(p)\cap \mathcal{F}$ a co-real hypersurface. We define a new prelegendrian as follows:
\begin{itemize}
    \item [(i)] Perform a prelegendrian RI move on $\Lambda$ along $N$, see \cref{rmk:PrelegendrianRIMove}. Denote the resulting prelegendrian by $\Lambda_1$. 
    \item [(ii)] Apply the $N$-pushing operation described in \cref{thm:Prelegendria_NPushing} to $\Lambda_1$ along a co-real $N'\cong N$ contained in $\Op(N)\cap \pi_{\mathrm{PF}}(\Lambda_1)$, very close to the cusp locus of the RI-move. Denote the resulting prelegendrian by $s(\Lambda)$
\end{itemize}
The prelegendrian $s(\Lambda)$ is said to be a \emph{stabilization} of $\Lambda$ along $N$.
\end{definition}

The following immediate corollary highlights the key property of prelegendrian stabilization:
\begin{proposition}\label{prop:LooseLegendrianLift}
The prelegendrians $\Lambda$ and $s(\Lambda)$ are formally prelegendrian isotopic. Moreover,
\begin{equation*}
    s(\mathcal{L} (\Lambda))=\mathcal{L}(s(\Lambda)).
\end{equation*}
In particular, $\mathcal{L}(s(\Lambda))$ is loose.
\end{proposition}

\begin{remark}
    Do note that we have not yet shown \cref{thm:PrelegendrianStabilization}, i.e. that prelegendrian stabilization is always possible. The reason is that we have not yet established that a co-real $N$ can always be found. We do this in the next section.
\end{remark}

\begin{remark}
    Even though the Legendrian lift of the prelegendrian stabilization is independent of the choice of co-real submanifold (by Murphy's $h$-principle, \cref{thm:h-principle_loose_leg}), it is unclear whether the stabilized prelegendrian satisfies an h-principle as well. It is an open question whether prelegendrian stabilization depends on the choice of co-real submanifold $N$.
\end{remark}

\section{Existence of compact prelegendrians and prelegendrian stabilizations}
\label{sec:Existence_of_preleg}

In this section we establish the existence of compact prelegendrians in $(\C^{2n+1},\D_{\std})$ and prelegendrian stabilizations. To do this, we will give an explicit family of a co-real submanifolds in $\C^{n+1}$ which is graphical over $\R^{2n+1}\subseteq \C^{n+1}=\R^{2n+2}=\R^{2n+1}\times \R$. The existence of these submanifolds implies then the existence of prelegendrian stabilizations as discussed in the previous section, thus yielding the proof of \cref{thm:PrelegendrianStabilization}.

\subsection{Co-real submanifolds in $\C^{n+1}$ graphical over $\R^{2n+1}$}

We begin by proving the existence of suitable co-real submanifolds to perform our prelegendrian spinning. After this, \cref{cor:prelegendrianspinning} trivially implies the existence of compact prelegendrians in $(\C^{2n+1},\D_\std)$.  
\begin{proposition}\label{prop:existence_coreal_graphical}
    Let $F\subseteq \R^{n}$ be a closed submanifold. Then, there exists a co-real embedding $\hat e:\T^{n+1}\times F\hookrightarrow \C^{n+1}=\R^{2n+2}=\R^{2n+1}\times \R$ which is graphical over $\R^{2n+1}$. In particular, there exists a co-real embedding of $\T^{2n}$.
\end{proposition}
\begin{proof}[Proof of \cref{prop:existence_coreal_graphical}]
We use polar coordinates $(r_i,\theta_i)$, in each $\C$-factor of $\C^{n+1} = \C \times \cdots \times \C$. 
Our construction takes places close to the Clifford torus $\T^{n+1}_\mathrm{Cl} \subseteq \C^{n+1}$, defined as the image of the embedding:
\begin{align*}
   \mathrm{Cl}:\T^{n+1}= \S^1\times\cdots \times \S^1 \quad\hookrightarrow\quad & \C^{n+1};\\
    (\theta_1,\ldots,\theta_{n+1}) \quad\mapsto\quad & \Big( (1,\theta_1);\cdots,(1,\theta_n);(1,\theta_{n+1})\Big).
\end{align*}

The Clifford torus is clearly not graphical over $\R^{2n+1}$. However, the family of perturbations
\begin{align*}
\mathrm{Cl}^\delta(\theta_1,\dots,\theta_{n+1})=\Big((1+\delta \sin(\theta_{n+1}),\theta_1);\cdots,(1,\theta_n);(1,\theta_{n+1})\Big),
\end{align*}
for $0<\delta<1$, is easily seen to be graphical over $\R^{2n+1}$. 

Since being totally real is a $C^1$-open condition we may fix $\delta>0$ small enough such that $\mathrm{Cl}^\delta$ is totally real and graphical over $\R^{2n+1}$. Write $\mathrm{Cl}^\delta=(g,\sin(\theta_{n+1})),$ where  $g:\T^{n+1}\hookrightarrow \R^{2n+1}$ is the embedding obtained by projecting $\mathrm{Cl}^\delta$ over $\R^{2n+1}$.

A global trivialisation of the normal bundle of $\mathrm{Cl}^\delta$ can be given by
\begin{equation*}
    \langle \partial_{r_1}+\partial_{r_{n+1}}, \partial_{r_2}, \ldots, \partial_{r_n} \rangle \oplus \langle \partial_z \rangle, 
\end{equation*}
where $\partial_z$ spans the vertical direction in $\C^{n+1}=\R^{2n+1}\times \R$. Moreover, the first $n$-vectors project to a trivialisation of the normal bundle of $g$ in $\R^{2n+1}$. 

Consider the inclusion $F\hookrightarrow \R^n$; without loss of generality we may assume that it has image in an $\varepsilon$-ball, for $\varepsilon>0$ sufficiently small. After fixing a tubular neighbourhood embedding of $g$, this inclusion induces an embedding  $e:\T^{n+1}\times F\hookrightarrow \R^{2n+1}$.

The embedding $\hat{e}=(e,\sin(\theta_{n+1})):\T^{n+1}\rightarrow \C^{n+1}=\R^{2n+1}\times \R$, is co-real for $\varepsilon>0$ sufficiently small, by an application of \cref{lem:bundle_over_coreal_is_coreal}. This concludes the argument. 
\end{proof}

Note that the Clifford torus $\T^{n+1}_\mathrm{Cl}\subseteq \C^{n+1}$ is isotopic to the standard torus (cf.\ \cref{subsec:special_spinning}). If we take $F=\T^{\,n-1}\subseteq \R^n$ to be the standard torus, then the resulting embedding $e:\T^{2n}\hookrightarrow \R^{2n+1}$ is also isotopic to the standard torus. Hence, by \cref{lem:torus_spin_is_iter_S1_spin}, we obtain:
\begin{lemma}\label{lem:coreal_torus_spin_is_iter_spin}
    In the situation of the above proposition, we can moreover assume that $\hat{e}:\T^{2n}\to \C^{n+1}$ satisfies:
    \begin{equation*}
        \spin_{\hat{e}}[L] \simeq (\spin_{\S^1})^{2n}[L],
    \end{equation*}
    for any Legendrian $L\subseteq J^1\R^m$.
\end{lemma}

\subsection{Existence of prelegendrian stabilizations}

\begin{proof}[Proof of \cref{thm:PrelegendrianStabilization}]
Let $\Lambda$ be a prelegendrian submanifold in $(\C^{2n+1},\D_\std)$. Consider a point $q\in \Lambda$ in the smooth stratum of the prelegendrian front $\pi_{\mathrm{PF}}(\Lambda)$. It follows from \cref{prop:existence_coreal_graphical} that, possibly after a prelegendrian front homotopy of $\Op(\pi_\PF(p))\cap \pi_\PF(\Lambda)$, there exists a codimension $2$ co-real torus $N\subseteq \Op(\pi_{\mathrm{PF}}(p))\cap \pi_{\mathrm{PF}}(\Lambda)\subseteq \C^{n+1}$. This places us in the setup of \cref{subsec:PrelegendrianStabilization}. In particular, the stabilization construction defined there is not an empty operation.

The very same approach works in $(\Gr(X,J),\D_{\can})$. Indeed, the co-real torus from \cref{prop:existence_coreal_graphical} can be assumed to be co-real for a general $J$ if we work in a sufficiently small ball. This is because the property of being co-real is preserved under $C^0$-perturbations of the almost complex structure, and every almost complex structure can be assumed to coincide with the standard one at a given point.
\end{proof}

\section{Exotic prelegendrians} \label{sec:Non_isotopic_preleg}

In this section we prove our main results about the failure of the $h$-principle for prelegendrian embeddings:  \cref{thm:InfinitelyPrelegendrianTorus}, \cref{cor:ExoticPrelegendrianPairsBall} and \cref{cor:ExoticPrelegendrianPairGrassmanian}.

\subsection{Auxiliary Lemmas in Contact Topology}

We recall two basic results in contact topology that we will need later. The first one is well-known. It follows from the functoriality of jet spaces:
\begin{lemma}\label{lem:unversal_cover_J1S^1xR^2n}
    Let $Q$ be a smooth manifold and $\widetilde{Q}$ its universal cover. Then, the contact universal cover of $J^1Q$ is
    \begin{equation*}
        (\widetilde{J^1Q},\widetilde{\xi_\can})  \simeq (J^1\widetilde{Q},\xi_\can).
    \end{equation*}
    
    In particular, the contact universal cover of $(J^1(\R^{m}\times \S^1),\xi_\can)$ is contactomorphic to $(J^1 \R^{m+1},\xi_\std)$.
\end{lemma}
\begin{proof}
The differential of the covering map $\widetilde{Q}\rightarrow Q$ induces a symplectic covering map $(T^*\widetilde{Q},\d\lambda_\can)\rightarrow (T^* Q, \d\lambda_\can)$ which preserves the canonical $1$-form. The result follows.
\end{proof}

The second auxiliary result follows by an standard cut-off argument with contact Hamiltonians:
\begin{lemma}\label{lem:non_isotopic_on_ball}
    Let $(Y^{2n+1},\xi)$ be a contact manifold such that its universal cover $(\widetilde{Y},\tilde{\xi})$ is contactomorphic to $(\R^{2n+1},\xi_\std)$ or $(\S^{2n+1},\xi_\std)$. Fix moreover a Darboux ball $(B,\xi_\std)\subseteq (Y,\xi)$. Suppose $L_0,L_1\subseteq (B,\xi_\std)\subseteq (Y,\xi)$ are two compact Legendrian submanifolds which are not Legendrian isotopic in the ball. Then $L_0$ and $L_1$ are not Legendrian isotopic in $(Y,\xi)$.
\end{lemma}
\begin{proof}
Since any Legendrian isotopy in $\S^{2n+1}$ avoids some point $p\in \S^{2n+1}$, it suffices to treat the case in which the universal cover is $(\R^{2n+1},\xi_{\std})\cong (\S^{2n+1}\backslash\{p\},\xi_\std)$. We prove the contrapositive. Assume that there is a Legendrian isotopy $L_t\subseteq (Y,\xi)$, $t\in[0,1]$, between $L_0$ and $L_1$. Consider a lift $(\tilde{B},\xi_\std)\subseteq (\R^{2n+1},\xi_\std)$ of the Darboux ball $(B,\xi_\std)$, together with lifts of the initial Legendrians $\tilde{L}_0,\tilde{L}_1\subseteq (\tilde{B},\xi_\std)$.
    
We claim that, since there is a Legendrian isotopy $L_t\subseteq(Y,\xi)$, $t\in[0,1]$, there exists a Legendrian isotopy $\tilde{L}_t\subseteq (\R^{2n+1},\xi_\std)$, $t\in[0,1]$, in the universal cover between the Legendrians $\tilde{L}_0$ and $\tilde{L}_1$. Indeed, the Legendrian isotopy can be lifted to a Legendrian isotopy $\hat{L}_t\subseteq (\R^{2n+1},\xi_\std)$, $t\in[0,1]$, starting at $\hat{L}_0=\tilde{L}_0$ but possibly ending at a lift $\hat{L}_1$ of $L_1$ which is different to $\tilde{L}_1$. This lift would be in a different lift $(\hat{B},\xi_\std)$ of the Darboux ball $(B,\xi_\std)$. Since the space of Darboux ball embeddings into a connected contact manifold is connected, the contact embeddings $(\tilde{B},\xi_\std)\hookrightarrow (\R^{2n+1},\xi_\std)$ and $(\hat{B},\xi_\std)\hookrightarrow (\R^{2n+1},\xi_\std)$ are contact isotopic. This implies that $\hat{L}_1$ and $\tilde{L}_1$ are Legendrian isotopic. The concatenation of both Legendrian isotopies is the required Legendrian isotopy $\tilde{L}_t\subseteq (\R^{2n+1},\xi_\std)$, $t\in [0,1]$.

Now we will show that the Legendrian isotopy $\tilde{L}_t$ is homotopic, relative to $t\in \{0,1\}$, to a Legendrian isotopy whose image is in $(\tilde{B},\xi_\std)$. Since the covering map defines a contactomorphism $(\tilde{B},\xi_\std)\rightarrow (B,\xi_\std)$ this would imply the existence of a Legendrian isotopy in $(B,\xi_\std)$ between $L_0$ and $L_1$, thus concluding the proof. 

Fix two auxiliary Darboux balls $(B_0,\xi_\std)$ and $(B_1,\xi_\std)$ in $(\R^{2n+1},\xi_\std)$ such that 
    \begin{itemize}
        \item $\tilde{L}_0,\tilde{L}_1\subseteq B_0$
        \item $\tilde{L}_t\subseteq B_1$
        \item $B_0\subseteq \tilde{B}\subseteq B_1$. 
    \end{itemize}
Since the space of Darboux balls is connected we may further assume that $(B_0,\xi_\std)$ is the unit ball in $(\R^{2n+1},\xi_\std)$. Consider the contact vector field
\begin{equation*}
    X=2z \partial z+\sum_{i=1}^n x_i\partial x_i + y_i\partial y_i,
\end{equation*}

in $(\R^{2n+1},\xi_\std)$, whose contact flow is the contact dilation $\varphi^t_X(\mathbf{x},\mathbf{y},z)=(e^t \mathbf{x}, e^t \mathbf{y}, e^{2t} z)$, $t\in\R$.  With respect to the standard contact form, $X$ has contact Hamiltonian
\begin{equation*}
    H=\alpha_\std (X)=2z-\mathbf{y}\cdot \mathbf{x}.
\end{equation*}

Let $\rho:\R^{2n+1}\rightarrow [0,1]$ be a smooth function such that 
 \begin{itemize}
     \item $\rho_{|B_0}\equiv 0$ and 
     \item $\rho_{|\R^{2n+1}\backslash \tilde{B}}\equiv 1$. 
\end{itemize}
Denote by $\tilde{X}$ the contact vector field associated to the Hamiltonian $\tilde{H}\defi\rho \cdot H$. The required Legendrian isotopy is then
\begin{equation*}
    \varphi^{-T}_{\tilde{X}}(\tilde{L}_t), t\in[0,1],
\end{equation*}
where $T>>0$ is big enough. This concludes the argument. 
\end{proof}

The previous lemma yields the following consequence, which, although not used later in the manuscript, may be of independent interest.
\begin{corollary}
     Let $(Y^{2n+1},\xi)$ be a contact manifold whose universal cover $(\widetilde{Y},\tilde{\xi})$ is contactomorphic to $(\R^{2n+1},\xi_\std)$ or $(\S^{2n+1},\xi_\std)$. Fix $L\subseteq (B,\xi_\std)$, a Legendrian submanifold contained in a Darboux ball. Then, $L$ is loose in $(B,\xi_\std)$ if and only if it is loose in $(Y,\xi)$.
\end{corollary}
\begin{proof}
    Argue as in the proof of \cref{lem:non_isotopic_on_ball}. Given a loose chart in $(Y,\xi)$, lift it to the universal cover and use contact Hamiltonians to push it into $(B,\xi_\std)$.
\end{proof}

\subsection{Exotic prelegendrian tori}

We now prove our main results.
\begin{proof}[Proof of \cref{thm:InfinitelyPrelegendrianTorus}]
Consider the family of Legendrian Whitehead doubles of stabilized Legendrian unknots $W_s\subseteq (J^1\R,\xi_\std)$ with $s=3,5,\cdots,$ depicted in \cref{fig:W_s-front}. By \cref{prop:existence_coreal_graphical}, there exists a co-real embedding $\hat{e}=(\hat{e}_{\R^{2n+1}},h):\T^{2n}\to\C^{n+1}=\R^{2n+2}=\R^{2n+1}\times\R$, which is graphical over $\R^{2n+1}$.

\begin{figure}[h]
    \centering
    \includegraphics[width=0.75\linewidth]{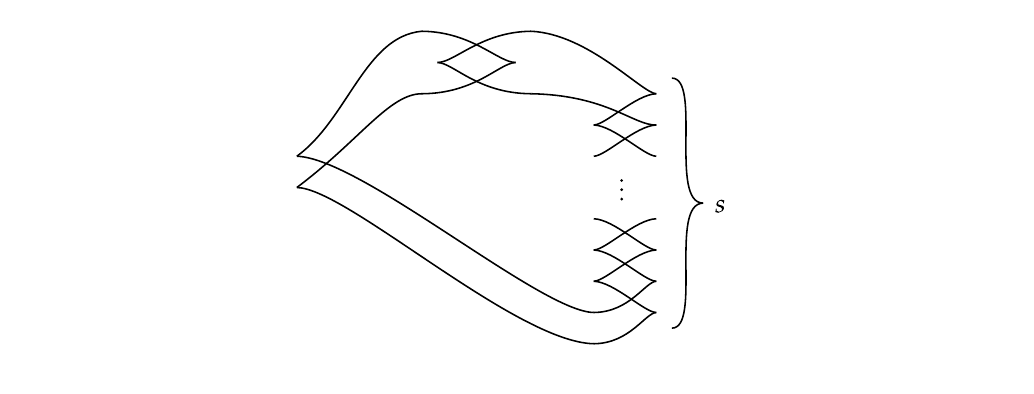}
    \caption{The Legendrian $W_s$, where $s$ is the number of crossings in the right part.}
    \label{fig:W_s-front}
\end{figure}
    
By \cref{cor:prelegendrianspinning} there is well-defined infinite family of prelegendrian $(2n+1)$-tori:
\begin{equation*}
    \Lambda_s=\pspin_{\hat{e}}[W_s]\subseteq (\C^{2n+1},\D_\std)
\end{equation*}
obtained by prelegendrian spinning. 

Observe that $\operatorname{Rot}(W_s)=0$ for all $s$. Therefore, by \cref{prop:rot=0_formal_isotopic} and \cref{rmk:PrelegendrianSpinningReidemeister} the prelegendrian fronts of $\Lambda_s$ are related by a sequence of prelegendrian Reidemeister moves and $N$-pushings along parallel copies of $\hat{e}(\T^{2n})$. \cref{thm:Prelegendria_NPushing} then implies that all these prelegendrians are formally isotopic. A formal prelegendrian isotopy is depicted in \cref{fig:W_s_homotopy}.
\begin{figure}[h]
    \centering
    \includegraphics[width=1\linewidth]{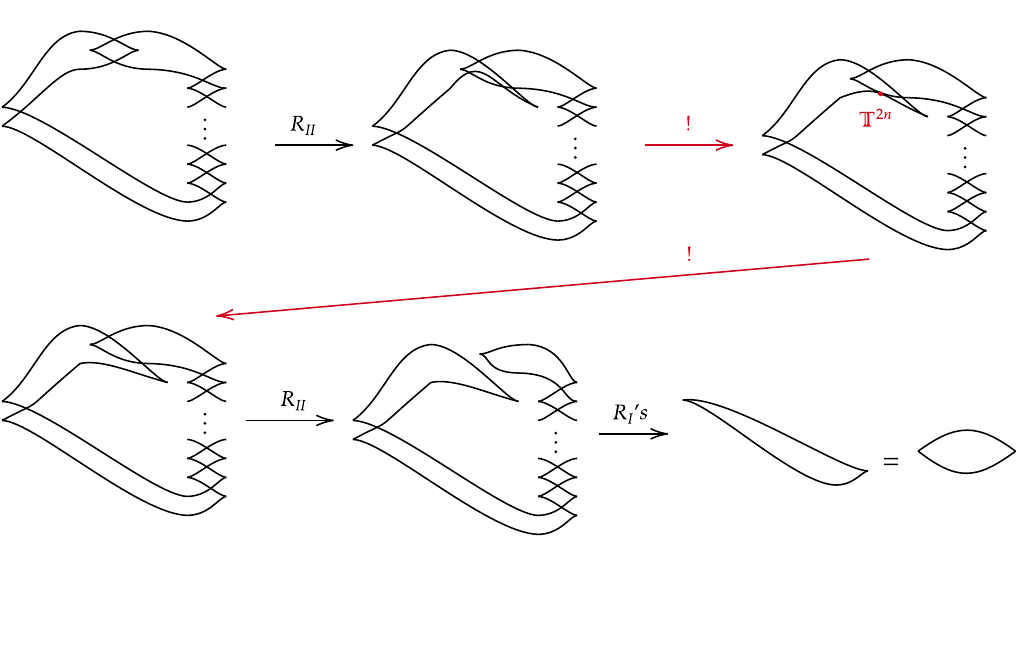}
    \caption{A formal prelegendrian isotopy between of $\Lambda_s=\mathrm{prespin}_{\hat e} [W_s]$ and $\mathrm{prespin}_{\hat{e}}[U]$, where $U\subseteq (J^1\R,\xi_\std)$ is the standard Legendrian unknot. The first and fourth arrows indicate prelegendrian RII moves. The second and third arrows describe prelegendrian $N$-pushing along the co-real $\T^{2n}$. The last arrow represents several prelegendrian RI moves along parallel copies of the co-real $\T^{2n}$.}
    \label{fig:W_s_homotopy}
\end{figure}

Now we will show that the Legendrian lifts $\cL(\Lambda_s)$ are pairwise not Legendrian isotopic, which by \cref{cor:leg_lift_invariant} will imply that the prelegendrians $\Lambda_s$ are pairwise not prelegendrian isotopic and hence the result. 

By construction,
\begin{equation*}
    \cL(\Lambda_s)= \spin_{\hat{e}}[W_s]\subseteq (J^1\R^{2n+1},\xi_\std).
\end{equation*}
Moreover, as Legendrians in $(J^1\R^{2n+1},\xi_\std)$, the $\cL(\Lambda_s)$ and $(\spin_{\S^1})^{2n}[W_s]$ are Legendrian isotopic by \cref{lem:coreal_torus_spin_is_iter_spin}. It was shown in \cite[Example 4.20]{EES:NonIsotopicLegendrians} that the Legendrian contact homologies of $\cL(\Lambda_s)$ are all non quasi-isomorphic, so the $\cL(\Lambda_s)$ are pairwise not Legendrian isotopic in $(J^1\R^{2n+1},\xi_\std)$. To conclude we need to tell them apart in the contactisation $\cC(\C^{2n+1},\D_\std)$. 
    
Notice that there is a natural inclusion
\begin{equation*}
    (J^1\R^{2n+1},\xi_\std)\cong (\S T^*_\mathrm{hor}(\R^{2n+2})^+,\xi_\can)\subseteq \cC(\Gr_\mathrm{hor}(\C^{n+1},J_\std),\D_{\can})=\cC(\C^{2n+1},\D_\std),
\end{equation*}
where the first contactomorphism follows from \cref{lem:HorizontalContactElements} and the inclusion is the one given in Equation \eqref{eq:InclusionHorizontalCn}. Moreover, by \cref{cor:contactsation_fat_std} we have $\cC(\C^{2n+1},\D_\std)\cong (J^1 (\R^{2n}\times \S^1),\xi_\std)$. Hence, \cref{lem:unversal_cover_J1S^1xR^2n} tells us that the contact universal cover of $\cC(\C^{2n+1},\xi_\std)$ is $(J^1\R^{2n+1},\xi_\std)$. Therefore, \cref{lem:non_isotopic_on_ball} can be applied to conclude that the Legendrians $\cL(\Lambda_s)$ are pairwise Legendrian non-isotopic in the contactisation $\cC(\C^{2n+1},\D_\std)$.
\end{proof}

\begin{remark}\label{rmk:InfinitelyToriComplexProjectiveSpace}
Recall that $(\C^{2n+1},\D_\std)\subseteq (\CP^{2n+1},\D_\std)$ is an affine chart ( \cref{subsec:complex_projective}). In particular, the infinite family $\Lambda_s$ constructed in the previous proof defines prelegendrians in complex projective space. 

Since the contact universal cover of $\cC(\CP^{2n+1},\D_\std)$ is $(\S^{4n+3},\xi_\std)$ by \cref{prop:contactisation_CP^2n+1}, we can apply \cref{lem:non_isotopic_on_ball} in this setting as well to conclude that the prelegendrians are not prelegendrian isotopic in $(\CP^{2n+1},\D_\std)$.
\end{remark}

\begin{remark}\label{rmk:InfinitelyToriInfinitelyManyFat}
   Let $(Q^{2n+1}\times \R,J)$ be an almost complex manifold, with universal cover $\R^{2n+1}\times \R$. Then, it follows from the same argument that $(\Gr_{\mathrm{hor}}(Q\times \R), \D_\can)$ has infinitely many exotic prelegendrian $(2n+1)$-tori as well. Indeed, we consider the prelegendrian fronts appearing in the proof and we insert them in a very small ball $\Op(q,z)$, where $(q,z)\in Q\times \R$. It will then follow that the front singularities are co-real for $J$, and hence define prelegendrians in $(\Gr_{\mathrm{hor}}(Q\times \R), \D_\can)$. Our assumption on $Q$ then implies that the contact universal cover of $\cC(\Gr_{\mathrm{hor}}(Q\times \R), \D_\can)$ is contactomorphic to $(J^1 \R^{2n+1},\xi_\std)$, so the same proof can conclude in the same manner. This yields infinitely many examples of fat manifolds of corank $2$ with infinitely many exotic prelegendrians.
\end{remark}

\subsection{Proof of \cref{cor:ExoticPrelegendrianPairsBall}}

Consider the perturbation of the Clifford torus $\hat{e}:\T^{n+1} \hookrightarrow \C^{n+1}$ from \cref{prop:existence_coreal_graphical}. The prelegendrian embedding $\Lambda: \T^{n+1}\times \S^{j_0}\times \cdots\times \S^{j_k} \hookrightarrow (\C^{2n+1},\D_\can)$ with non-loose Legendrian lift $\cL(\Lambda)$ will be defined as
\begin{equation*}
    \Lambda\defi \pspin_{\hat{e}} [L]\subseteq (\C^{2n+1},\D_\std),
\end{equation*}
for a certain Legendrian $L\cong \S^{j_0}\times \cdots\times \S^{j_k}\subseteq (J^1\R^n,\xi_\std)$. To conclude that $\cL(\Lambda)=\spin_{\hat{e}}[L]\subseteq (J^1\R^{2n+1},\xi_\std)$ is non-loose, we will show that $\cL(\Lambda)$ is Lagrangian fillable \cite{Ekholm:RSFT,Murphy:Loose}. The second part of the result will then follow from \cref{thm:PrelegendrianStabilization}.

By \cref{lem:coreal_torus_spin_is_iter_spin}, the Legendrians $\spin_{\hat{e}}[L]$ and $(\spin_{\S^1})^{n+1}[L]$
are Legendrian isotopic. Therefore, the proof reduces to showing the existence of a Legendrian $L\cong \S^{j_0}\times \cdots\times \S^{j_k}\subseteq (J^1 \R^n,\xi_\std)$ such that
\begin{itemize}
    \item $L$ has a simple front, so $\Lambda=\pspin_{\hat{e}}[L]\subseteq (\C^{2n+1},\D_\std)$ is well-defined by \cref{cor:prelegendrianspinning}.
    \item $(\spin_{\S^1})^{n+1}[L]\subseteq (J^1 \R^{2n+1},\xi_\std)$ is Lagrangian fillable.
\end{itemize}

Let $U\cong \S^{j_0}\subseteq (J^1 \R^{j_0},\xi_\std)$ be the standard Legendrian unknot. Then, the Legendrian
\begin{equation*}
    L\defi \spin_{\S^{j_k}}\circ \spin_{\S^{j_{k-1}}}\circ \cdots \circ \spin_{\S^{j_1}} [U]\subseteq (J^1\R^n,\xi_\std),
\end{equation*}
 satisfies the required properties: it has a simple front and, moreover, $(\spin_{\S^1})^{n+1}[L]$ is Lagrangian fillable because the $\S^k$-spinning of a fillable Legendrian is fillable \cite[Remark 2]{Golovko}. This concludes the argument.
\qed

\begin{remark}
There is an alternative argument to conclude non-looseness in the above construction, which avoids appealing to Lagrangian fillability. Dimitroglou Rizell and Golovko prove in \cite[Theorem~3.1]{DimitroglouGolovko:LinearRep_CharacteristicAlgebra} that $\S^{m}$–spinning preserves augmentations: if the Chekanov--Eliashberg DGA of a Legendrian submanifold admits an augmentation, then the DGA of its $\S^{m}$--spinning also admits one.
In particular, non-acyclicity of the DGA (equivalently, non-vanishing of linearized LCH) is preserved under spinning.  Thus, whenever $L$ has non-zero Legendrian contact homology,
the same holds for $\spin_{\S^{m}}[L]$, which already implies that the latter is non-loose.
\end{remark}

\begin{remark}
It follows from \cite[Theorem 1.2]{Golovko} that we can actually find infinitely many prelegendrians as in \cref{cor:ExoticPrelegendrianPairsBall} whose Legendrian lifts are not Legendrian isotopic. We expect these prelegendrians to be formally prelegendrian isotopic, yielding new infinite families of exotic prelegendrians with topology not restricted to be tori. However, the formal computation is more involved than the case of tori, since we cannot reduce it to $1$-dimensional Legendrian knot theory, as in \cref{prop:rot=0_formal_isotopic}. This remains an open problem.
\end{remark}

\subsection{Proof of \cref{cor:ExoticPrelegendrianPairGrassmanian}} 

The Legendrian lift of the prelegendrian $\Lambda_M=(M,TM^J)\subseteq (\Gr(X,J),\D_\can)$ is
\begin{equation*}
    \cL(\Lambda_M)=(M,TM)\subseteq (\S T^* X,\xi_\can)=\cC(\Gr(X,J),\D_\can),
\end{equation*}
here the prelegendrian co-orientation is defined by the outward normal at $M=\partial Y$ defined by $Y$. 

We claim that $\cL(\Lambda_M)$ admits an exact Lagrangian filling in $(\mathbb{D}T^* X, \d\lambda_\can)$, hence is non-loose \cite{Ekholm:RSFT,Murphy:Loose} and the result follows from \cref{thm:PrelegendrianStabilization}. This must be well-known, however we include the proof here for completeness. Fix a Riemannian metric in $M$, and consider the associated metric in $T^*M$. Then, if $\nu_p$, $p\in M=\partial Y$, is the outward  unit normal to $M$, we can identify the Legendrian lift as the unit co-normal lift associated to $\nu$, i.e. $\cL(\Lambda_M)=\{(p,\nu^*_p):p\in M\}\subseteq (\S T^*X,\xi_\can)$.

Consider a tubular \nbh of $M$
\begin{align*}
    \Phi:M\times (-\varepsilon,\varepsilon)&\to U\subseteq X\\
        (p,t)&\mapsto \exp_p(t\nu_p),
\end{align*}
Define $\psi:U\to\R$ by $\psi(\Phi(p,t))\defi t$. Then, $\d_p\psi=\nu^*_p$ for $p\in M$, $\|\d\psi\|=1 $, $\psi_{|Y\cap U} <0$ and $ \psi_{|M}=0$. Since $M$ is separating we may extend $\psi$ to a smooth function in $X$, still denoted by $\psi$ such that $\phi_{|Y\setminus U}\leq -\varepsilon$ and $\psi_{|X\setminus Y}>0$.

Fix a step function $\chi:(-\infty,0]\to[-1,0]$ such that
\begin{itemize}
    \item $\chi(0)=0$,
    \item $\chi'(0)=1$,
    \item $0\leq\chi'(t)< 1$ for all $t<0$,
    \item $\operatorname{supp}(\chi')\subseteq (-\varepsilon,0]$.
\end{itemize}

Set $f:Y\to \R$ by $f(p)\defi\chi(\psi(p))$. Then $L\coloneqq \mathrm{graph}(\d f)\subseteq T^*X$ is an exact Lagrangian. By definition,
\begin{equation*}
    \d_p f = \chi'(\psi(p))\d_p\psi.
\end{equation*}
From here it follows that $\|(\d f)_p\|< 1 $ for $p\in \mathrm{Int}(Y)$ and $\d_p f=\nu^*_p$ for $p \in M=\partial Y$. That is, $L\subseteq (\bD T^* X, \d \lambda)$ is an exact Lagrangian filling of $\Lambda_M\subseteq (\S T^* X, \xi_\can)$ and the result follows.
\qed

\section{Constructing prelegendrians in non-standard fat distributions} \label{sec:general}

In this section we translate our previous results to more general fat distributions. To do this we use two main tools: (1) the nilpotentisation, to be introduced next, and (2) Gray stability for contact structures. These tools partially remedy the lack of Darboux and Gray for fat structures of higher corank.

\subsection{The nilpotentisation}

According to Cartan's study of normal forms of distributions, corank-$2$ fat distributions do not admit a Darboux type theorem. However, an approximate local form does exist \cite{Ge:Char_subRiemannian, montgomeryTourSubriemannianGeometries2002}, as we now explain.

Consider $(M^{4n+2},\D^{4n})$ fat and fix a point $p \in M$. Then, there are local coordinates $(x_1,\ldots,x_{4n},z_1,z_2)$ and $1$-forms
\begin{equation*}
    \lambda^i =\d z_i + \sum_{j,k} \Gamma^i_{jk}x_j\d x_k + R_i,\quad i=1,2,
\end{equation*}
such that $\D \overset{loc.}{=} \ker\lambda^1\cap \ker \lambda^2$. Here $R_i$ is a remainder/defect term
\begin{equation*}
    R_i = \sum_j f_{ij}\d z_j + \sum_j g_{ij}\d x_j,
\end{equation*}
with $f,g\in O(|x|^2+|z|^2)$.

The coefficients $\{\Gamma^i_{jk}\}$ depend on $p$ and are the structure constants of the Lie algebra structure on $\D^{2n}_p \oplus TM_p/\D_p$ induced by the curvature of $\D$. As $p$ varies, this yields a weak Lie algebra bundle over $M$, meaning that the isomorphism type of the Lie algebra may vary from point to point. This bundle is called the \emph{nilpotentisation} of $\D$. The Lie algebra at $p$ defines a left-invariant fat distribution on the Lie group that integrates it; it is given precisely by the $1$-forms
\begin{equation*}
    \lambda^i =\d z_i + \sum_{j,k} \Gamma^i_{jk}x_j\d x_k,
\end{equation*}
i.e. by removing the remainder. The following standard result \cite{Bellaiche1996} tells us that this provides an approximate model for $\D$:
\begin{proposition}
Fix a fat distribution $(M^{4n+2},\D^{4n})$, a point $p \in M$, and local coordinates $(x_1,\ldots,x_{4n},z_1,z_2)$, as above. Let $\D_{model}$ be the model fat distribution associated to the nilpotentisation at $p$. 

Then, there is a path of fat distributions $(\D_s)_{s \in [0,1]}$ on $\R^{4n+2}$ such that:
\begin{itemize}
    \item If $s>0$, $\D_s$ is diffeomorphic to the restriction of $\D$ to the chart.
    \item $\D_0 = \D_{model}$.
\end{itemize}
\end{proposition}
The proof amounts to performing a suitable weighted dilation in order to get rid of the remainder. As we take the limit as the dilation factor goes to zero, we obtain the model structure.

For the standard fat space $(\C^{2n+1},\D_\std)$, this space is itself the model, and its Lie algebra is known as the \emph{complex Heisenberg algebra}. A particularly interesting case is $n=1$. By the classification of nilpotent $6$-dimensional Lie algebras \cite{ConsoleFinoSamiou:6dimNilLieAlg,Magnin:LieNilAlg<=7}, there is a unique Lie algebra structure that is fat of corank $2$ for dimension $6$. Hence:
\begin{corollary} \label{cor:nilpotentisation46}
Fix a fat distribution $(M^6,\D^4)$, a point $p \in M$. Then, there is a path of fat distributions $(\D_s)_{s \in [0,1]}$ on $\C^3$ such that:
\begin{itemize}
    \item If $s>0$, $\D_s$ is diffeomorphic to the restriction of $\D$ to the chart.
    \item $\D_0 = \D_\std$.
\end{itemize}
\end{corollary}

\subsection{Prelegendrians in general fat manifolds}

This language of approximate models allows us to prove:
\begin{proof}[Proof of  \cref{thm:general46}]
Use \cref{cor:nilpotentisation46} to produce a path of fat distributions $\D_s$ starting with $\D_0 = \D_\std$ such that $\D_s$ is diffeomorphic to the restriction of $\D$ to some small ball whose radius depends on $s>0$. Now, consider the contactisation of the family $(M,\D_s)$, these define a family of (isotropic) bundle projections:
\begin{equation*}
    \mathrm{pr}_s: (U,\xi_s) \to  (B,\D_s),
\end{equation*}
where $U$ is an open manifold diffeomorphic to $B\times \S^1$ and $B$ is some open $6$-ball around the origin. Consider now the prelegendrians $P$ of $(\C^3,\D_\std) = (B,\D_0)$ constructed in \cref{thm:InfinitelyPrelegendrianTorus}, they define corresponding Legendrian lifts $L$ on $\cC(B,\D_0) = (U,\xi_0)$.

These Legendrian lifts can be assumed to be contained in some compact neighbourhood $K \subseteq U$. Gray stability tells us that there is a small time $s_0>0$ such that a family of isocontact embeddings $\phi_s: (K,\xi_0) \rightarrow (U,\xi_s)$, $s \in [0,s_0]$, exists. Consider then the images $\phi_s(L)$, these are Legendrians in $(U,\xi_s)$. Since $\phi_0(L)$ is transverse to the projection $\mathrm{pr}_0$ and transversality is an open condition, then $\phi_s(L)$ is transverse to $\mathrm{pr}_s$ for sufficiently small $s>0$, hence it projects down to a prelegendrian in $(B,\D_s)$, as desired. This proves the first claim.

For the second claim, we consider any prelegendrian $\Lambda$ in $(M,\D)$ and zoom in at a point $p \in \Lambda$. The zooming argument above, in a chart, allows us to produce a path of prelegendrians $\Lambda_s \subseteq (B,\D_s)$ such that $\Lambda_0 \simeq T_p\Lambda$ is a prelegendrian linear subspace in $(B,\D_0) = (\C^3,\D_\std)$. By the symmetry of the complex Heisenberg algebra, we can assume that $\Lambda_0$ projects to a prelegendrian front; we can then stabilise and denote $\tilde{\Lambda}_0 \defi s(\Lambda_0)$, and we denote its Legendrian lift by $\tilde{L}_0$.

In the contactisations $\mathrm{pr}_s\colon (U,\xi_s)\to (B,\D_s)$, the same construction as in the first claim produces a path of Legendrians $\tilde{L}_s \defi \phi_s(\tilde{L}_0)$ contained in a compact set $K \subseteq U$, for $s \in [0,s_0]$ with $s_0>0$ small. Furthermore, we can assume that the loose chart of each $\tilde{L}_s$ is contained in a compact set $W \subseteq K$.

Note that $\tilde{L}_s$ is $C^\infty$-close to the Legendrian lift $L_s = \cL(\Lambda_s)$. Then $\tilde{L}_s$ is a graph in a Weinstein \nbh\ chart of $L_s$, and by cutting off with a suitable function we produce a new path of Legendrians $L_s'$ such that 
\begin{itemize}
    \item $L_s'$ coincides with $\tilde{L}_s$ in $W$, which contains a loose chart;
    \item $L_s'$ coincides with $L_s$ away from $\Op(W)$.
\end{itemize}
Observe that $L_0' = \tilde{L}_0$. By openness of transversality, we know that $L_s'$ is transverse to $\mathrm{pr}_s$ for sufficiently small $s>0$; this defines the stabilised prelegendrian of $\Lambda$.
\end{proof}

The exact same reasoning shows:
\begin{theorem}
Fix a fat distribution $(M^{4n+2},\D^{4n})$. Suppose it contains a point whose nilpotentisation is the complex Heisenberg algebra. Then, it contains compact prelegendrians.
\end{theorem} 

\subsection{Stability phenomena}

The absence of a Darboux-type theorem for fat distributions of corank greater than $1$ implies that Gray stability does not hold for these classes. Nonetheless:
\begin{proof}[Proof of \cref{lem:paths}]
For the first claim we reason as in the proof of  \cref{thm:general46}. We apply Gray stability in the contactisation to the Legendrian lift. This produces a family of Legendrians that, at least for small time, remain transverse to the projection to the fat manifold.

For the second claim we apply the reasoning appearing in the proof of \cref{thm:general46} parametrically. That is: we apply the nilpotentisation parametrically on $s$. Then, given any prelegendrian for $\D_\std$ we may assume that it is contained in an arbitrary small ball, thanks to a dilation. Since $\D_\std$ is the model structure on the nilpotentisation, we see a copy of the prelegendrian in each $\D_s$, parametrically.
\end{proof}

\begin{proof}[Proof of  \cref{thm:noPaths}]
Suppose that such a path does exist. Since the homotopy $\D_s$ is compactly supported, the contactisations define a path of contact structures that is also compactly supported. As such, Gray stability applies, telling us that the Legendrian lifts are contactomorphic. Since the Legendrian lifts are distinguished by LCH, this is a contradiction.
\end{proof}

\bibliographystyle{abbrv}
\bibliography{Elliptic_references.bib}

\end{document}